\documentclass[12pt,a4paper]{amsart}
\usepackage[utf8]{inputenc}
\usepackage[OT1]{fontenc}
\usepackage{amsmath}
\usepackage{amsthm}
\usepackage{amsfonts}
\usepackage{amssymb}
\usepackage{graphicx}
\usepackage{cite}
\usepackage{hyperref}
\usepackage{enumitem}
\usepackage{xcolor}
\usepackage{tikz}
\usetikzlibrary{cd}
\usepackage[left=2cm,right=2cm,top=2cm,bottom=2cm]{geometry}
\graphicspath{{pic/}}

\newtheorem{thm}{Theorem}[section]
\newtheorem{prop}[thm]{Proposition}
\newtheorem{lem}[thm]{Lemma}
\newtheorem{cor}[thm]{Corollary}

\theoremstyle{definition}
\newtheorem{definition}{Definition}[section]
\newtheorem{remark}{Remark}

\newtheorem*{solution*}{Solution}
\newtheorem{ass}{Assumption}

\DeclareMathOperator{\dive}{div}

\DeclareMathOperator{\const}{const}

 \newcommand{\bbar}[1]{\setbox0=\hbox{$#1$}\dimen0=.2\ht0 \kern\dimen0 \overline{\kern-\dimen0 #1}}
  
 \DeclareMathOperator{\coker}{coker}
  
 \DeclareMathOperator{\End}{\ensuremath{\mathcal{E}\kern-.125em\mathpzc{nd}}}

 \DeclareMathOperator{\Hom}{\mathcal{H}\kern-.125em\mathpzc{om}}
 
 \DeclareMathOperator{\id}{id}

 \DeclareMathOperator{\Proj}{\mathcal{P}\kern-.125em\mathpzc{roj}}

 \renewcommand{\setminus}{\smallsetminus}
 \DeclareMathOperator{\spec}{Spec}

 \newcommand{\udot}{\ensuremath{{\lower .183333em \hbox{\LARGE \kern -.05em$\cdot$}}}}

 \newcommand{\cA}{\mathcal{A}}
 \newcommand{\cB}{\mathcal{B}}
 
 \newcommand{\cD}{\mathcal{D}}
 \newcommand{\cE}{\mathcal{E}}
 \newcommand{\cF}{\mathcal{F}}
 \newcommand{\cG}{\mathcal{G}}
 \newcommand{\cH}{\mathcal{H}}

 \newcommand{\cK}{\mathcal{K}}
 \newcommand{\cL}{\mathcal{L}}
 \newcommand{\cM}{\mathcal{M}}
 
 \newcommand{\cO}{\mathcal{O}}

 \newcommand{\cV}{\mathcal{V}}

 \newcommand{\cZ}{\mathcal{Z}}

\newcommand{\C}{\mathbb{C}}

 \newcommand{\N}{\mathbb{N}}

 \newcommand{\R}{\mathbb{R}}

  \newcommand{\Z}{\mathbb{Z}}

 \newcommand{\p}{\partial}

 \DeclareMathOperator{\sign}{sign}

 \DeclareMathOperator{\spann}{span}

\title[Self-adjoint extensions and the $\alpha$-calculus]{Self-adjointness criteria and self-adjoint extensions of the Laplace-Beltrami operator on $\alpha$-Grushin manifolds}

\author{Ivan Beschastnyi}
\address{Centre Inria d'Universit\'e C\^ote d'Azur, equipe MCTAO, Universit\'e C\^ote d'Azur, LJAD, Nice, France\\ 
CIDMA, Aveiro, Portugal
} \email{ivan.beschastnyi@inria.fr}

\author{Hadrian Quan}
\address{Department of Mathematics, University of Washington, Seattle, WA 98195, USA} \email{hadrianq@uw.edu}
 
\begin{document}
\date{}
\maketitle

\begin{abstract}
    The Grushin plane serves as one of the simplest examples of a sub-Riemannian manifold whose distribution is of non-constant rank. Despite the fact that the singular set where this distribution drops rank is itself a smoothly embedded submanifold, many basic results in the spectral theory of differential operators associated to this geometry remain open, with the question of characterizing self-adjoint extensions being a recent question of interest both in sub-Riemannian geometry and mathematical physics.

    In order to systematically address these questions, we introduce an exotic calculus of pseudodifferential operators adapted to the geometry of the singularity, closely related to the 0-calculus of Mazzeo arising in asymptotically hyperbolic geometry. Extending results of \cite{me, grushin,luca1}, this calculus allows us to give a criterion for essential self-adjointness of the Curvature Laplacian, $\Delta-cS$ for $c>0$ (here $S$ is the scalar curvature). When this operator is not essentially self-adjoint, we determine several natural self-adjoint extensions. Our results generalize to a broad class of differential operators which are elliptic in this calculus.
\end{abstract}

\section{Introduction}

\subsection{Problem statement and previous results}

In this article, we study the Laplace-Beltrami operator on $(n+1)$-dimensional $\alpha$-Grushin manifolds, which have recently garnered much attention in both the sub-Riemannian and mathematical physics community~\cite{grushin,Prandi,alessandro,bes_grushin,me,pozzoli2,EU,luca1,luca3}. These manifolds are defined as follows.
\begin{definition}
\label{def:grushin}
Given a real number $\alpha$, an $\alpha$-Grushin manifold is a triple $(M,\cZ,g)$ consisting of a smooth manifold $M$ of dimension $n+1$, a co-orientable embedded submanifold $\cZ\subset M$ of codimension one, which we call the singular set, and a metric $g$ which in a tubular neighborhood of $\cZ$ takes the form
\begin{equation}
\label{eq:metric_grushin}
g_\alpha = dx^2 + \frac{1}{|x|^{2\alpha}}g_{x,\cZ},
\end{equation}
where $g_{x,\cZ}$ is a smooth family of Riemannian metric on the level sets $x = const$ and $x$ is the distance to $\cZ$. 
\end{definition}

These manifolds have a naturally associated volume form $\omega$ and the Riemannian Laplace-Beltrami operator $\Delta$. However, both of those objects exhibit a singularity at $\cZ$. For this reason, a priori we can define $\Delta$ only on $C^\infty_c (M\setminus \cZ)$. The animating goal of this work began with the following two natural questions, inspired by the mathematical physics literature above, which considers using the scalar curvature $S$ as a geometrically natural singular potential to define an associated Schr\"{o}dinger operator.

\begin{center}
\textit{Question 1: If $\Delta$ is the Laplace-Beltrami operator of an $\alpha$-Grushin manifold, $S$ its scalar curvature and $c\in \R$ a constant, when is the operator $\Delta - c S$ essentially self-adjoint on $C^\infty_c(M\setminus \cZ)$ with respect to the Riemannian $L^2$-space?}
\end{center}

\begin{center}
\textit{Question 2: If $\Delta - c S$ is not essentially self-adjoint, what are its self-adjoint extensions?}
\end{center}

The motivation for those questions comes from the unusual properties of $\alpha$-Grushin manifolds. For values of $\alpha \in \N$ these $\alpha$-Grushin manifolds are well-defined metric spaces and whose geodesics can cross $\cZ$ without developing any singularity. All such geodesics are projections of a Hamiltonian system endowed with an energy Hamiltonian, suggesting the dynamics of these spaces enjoy a form of classical completeness. However, if we try to consider a corresponding quantization of this Hamiltonian system, we find that the Laplace-Beltrami operator can become self-adjoint, as was first proven in~\cite{grushin}. Physically this corresponds to a situation when a classical particle can pass through a singularity, while a quantum particle can not. This phenomenon is now known as \textit{quantum confinement}. Moreover, quantum properties depend significantly on the particular choice of quantization, parameterized here by the constant $c$. In the literature one can find values $c=1/3$~\cite{dewitt,woodhouse}, $c=2/3$~\cite{dewitt}, $c\in[0,2/3]$~\cite{woodhouse}. It should also be noted that for $n\geq 3$ and $c=(n-2)/(4(1-n))$, the operator $\Delta - cS$ is conformally invariant and is referred to as the conformal Laplacian. In~\cite{me} the first author with his collaborators proved that the quantum confinement is absent for $\alpha = n =1$ and $c>0$. The present work continues the investigation of this phenomenon.

Most of the known results for these types of spaces were developed for flat models, which can be described as $\R^{n+1}$ with $\cZ= \{x= 0\} \simeq \R^n$ and a metric
\begin{equation}
\label{eq:model_metric}
g_\alpha = dx^2 + \frac{1}{|x|^{2\alpha}}\left(\sum_{i=1}^n dy_i^2 \right).
\end{equation}
An orthonormal frame in this case is given by
$$
X_0 = \p_x, \qquad X_i = |x|^\alpha\p_{y_i}, \qquad i=1,\dots,n 
$$
and the Laplace-Beltrami operator has the form
\begin{equation}
\label{eq:model_operator}
\Delta^G = \p_x^2 +|x|^{2\alpha}\Delta_{\R^n} -\frac{\alpha n}{x}\p_x,
\end{equation}
where $\Delta_{\R^n}$ is the Euclidean Laplacian.

We first consider when the operator ~\eqref{eq:model_operator} is essentially self-adjoint on \, $C^\infty_c(\R^{n+1}\setminus \{x=0\})$ with respect to its Riemannian $L^2$-space. In this case, the volume form is given by
$$
\omega = \frac{dx \wedge dy_1 \wedge \dots \wedge dy_n}{|x|^{\alpha n}}.
$$
We would like to transform this to the standard Euclidean volume. For this reason, we make a unitary transform:
\begin{equation}
\label{eq:unit_trans}
u \mapsto |x|^\frac{\alpha n}{2}u.
\end{equation}
The Laplace-Beltrami operator in these new coordinates becomes 
$$
P = |x|^{-\frac{\alpha n}{2}} \Delta^G |x|^\frac{\alpha n}{2} = \p_x^2 + |x|^{2\alpha}\Delta_{\R^n} - \frac{\alpha (\alpha n+ 2)}{4x^2}.
$$
Performing a Fourier transform on the $y$ variables we obtain
$$
\hat{P}(\xi) = \p_x^2 - |x|^{2\alpha}|\xi|^2 - \frac{\alpha n(\alpha n + 2)}{4x^2}.
$$
It is clear that self-adjointness and self-adjoint extensions can be determined by this family of one-dimensional operators. One expects that $\Delta^G$ should be self-adjoint if $\hat{P}(\xi)$ are all self-adjoint as well. Each $\hat{P}(\xi)$ is a Schr\"odinger operator with an inverse square potential, whose self-adjoint properties are well-known~\cite{georgescu3}. This strategy was used in~\cite{EU}, where the authors proved the following theorem for $n=1$:

\begin{thm}[\cite{EU}]
Let $G$ be the $\alpha$-Grushin plane of dimension two. Depending on $\alpha$ the following statements hold: 
\begin{enumerate}
\item If $\alpha \in (-\infty,-3] \cup [1,+\infty)$, then $\Delta^G$ is essentially self-adjoint on $C^\infty(\R^{2}\setminus \cZ)$;
\item If $(-3,1)$, then $\Delta^G$ is not essentially self-adjoint on $C^\infty(\R^{2}\setminus \cZ)$ and
\begin{enumerate}
\item For $\alpha \in (-3,-1]$ the deficiency index is equal to two;
\item For $\alpha \in (-1,1)$ the deficiency index is infinite.
\end{enumerate}
\end{enumerate}
\end{thm}
A precise definition of deficiency indices will be given later. For now it is just important that they measure how far a given operator is being from essentially self-adjoint. 

A similar result was proven earlier in~\cite{Prandi} for Grushin cylinders. Using exactly the same idea one can prove an analogous result for all $n\geq 1$.
\begin{thm}
\label{thm:sa_grushin_standard}
Let $G$ be the $\alpha$-Grushin plane of dimension $n+1$. Depending on $\alpha$ the following statements hold: 
\begin{enumerate}
\item $\Delta^G$ is essentially self-adjoint on $C^\infty(\R^{n+1}\setminus \cZ)$ if and only if $\alpha \in (-\infty,-3/n] \cup [1/n,+\infty)$.
\item Assume that $\Delta^G$ is not essentially self-adjoint. If $n = 1$ or $n=2$ and $\alpha\in (-3/n,-1]$ then deficiency index is equal to two. In all the remaining cases the deficiency index is infinite.
\end{enumerate}
\end{thm}

There is no conceptual change if we add a scalar curvature term to $\Delta^G$. The scalar curvature for the $\alpha$-Grushin plane is given by (see Appendix A)
\begin{equation}
\label{eq:scalar_alpha}
S = -\frac{\alpha n(\alpha n + \alpha + 2)}{x^2},
\end{equation}
thus $S$, as a multiplication operator, commutes with the unitary transformation ~\eqref{eq:unit_trans}. Hence after we apply the latter and the partial Fourier transform in the $y$ variable, the resulting operator will have exactly the same form as $\hat{P}(\eta)$ but with a different constant in the numerator of the inverse square term. Consequently, we can indicate three regions on the space $(\alpha,n,c)$ of parameters: where $\Delta^G$ has deficiency index equal to zero, two and infinity.

We should also mention that unsurprisingly the properties of $\Delta - cS$ are closely related to the Schr\"odinger operator with the inverse square potential. Its natural domains were previously studied in~\cite{georgescu3}.

Apart from the already mentioned work~\cite{grushin}, there are scant results in the case of general $\alpha$-Grushin manifolds, and even fewer for more general sub-Riemannian structures. In~\cite{luca1} sufficient conditions for self-adjointness were obtained for a large class of singular manifolds which are even more general than the $\alpha$-Grushin manifolds we study here. Later the same group of authors generalized their results to sub-Riemannian structures~\cite{luca2}. Compared to their work, the method presented here can determine not only when $\Delta$ is not essentially self-adjoint, but can further explicitly construct many of its self-adjoint extensions. This generalizes the results of~\cite{pozzoli2} beyond the flat model $\Delta^G$. Also, operators in~\cite{guerra} can be studied using the $\alpha$-calculus we develop in this paper.

\subsection{Main tools and results}

In this paper we make the following set of assumptions:
\smallskip
\begin{enumerate}[label=(A\arabic*)]
\item The manifold $M$ is closed;
\item The singular set $\cZ$ has only one connected component;
\item $\alpha > -1$.
\end{enumerate}
\smallskip

All of these assumptions are technical. We need the closeness of $M$ in order to not consider the behavior of our operators at infinity, but adding the appropriate analysis can deal also with non-compact cases. Similarly, $\cZ$ can have more than one component with different orders of singularity $\alpha$. All the results will be local and thus each singularity can be treated separately.

The assumption that $\alpha >-1$ is the most important one since it is the result of the method we use. We build a pseudo-differential calculus, the $\alpha$-calculus, that is closely related to the $0$-calculus~\cite{mazzeo1987meromorphic}. In particular, when $\alpha = 0$ our calculus reduces to the $0$-calculus. This construction allows us to find a left parametrix for $\Delta - cS$ and determine quite explicitly the minimal and maximal domains. Papers~\cite{mendoza,bes_grushin} use this technique for studying elliptic wedge operators and can be compared to other works where suitable adapted PDO calculi are constructed, see for example ~\cite{melrose_b, mazzeo1987meromorphic, mazzeo, epstein_melrose} and especially \cite{albin2022sub} which featured a calculus adapted to the geometry of general contact sub-Riemannian manifolds.

The cases $\alpha = -1$ and $\alpha<-1$ require a separate consideration, which is absent from this work. When $\alpha = - 1$ the $x^2(\Delta - cS)$ is a $b$-operator and one can use the $b$-calculus to obtain similar results. For $\alpha < -1$ one needs to develop a different calculus which seems, however, to be very similar to the cusp calculus (see \cite{mazzeo1998pseudodifferential} and more generally \cite{vaillant2001index}). This is one setting we plan to pursue in future works, as much less is known for Grushin-type manifolds in this range of $\alpha$.

It should be noted that there also exists an alternative technique to prove the existence of left parametrices that uses the notion of Lie groupoids~\cite{nistor_fred}. It is certainly possible to find the minimal domain this way, however, it is not very clear how to obtain a good description of the maximal domain in this general and abstract setting.

\medskip

$\alpha$-Grushin manifolds have some properties that simplify their study. For example, in Appendix A we prove the following proposition.
\begin{prop}
\label{prop:curvature_ass}
Consider a $\alpha$-Grushin manifold of dimension $n+1$ with $\alpha>-1$ and compact singular set $\cZ$. Then
\begin{equation}
\label{eq:scalar_ass}
S = -\frac{\alpha n(\alpha n + \alpha + 2)}{x^2} + o(|x|^{-2}), \qquad x\to 0.
\end{equation}
\end{prop} 
Thus we see that the main term in the asymptotics of the curvature depends only on the order of singularity and the dimension of the manifold, but does not depend on the geometry of the manifold $M$ or the singular set $\cZ$. This is a very useful property for the parametrix construction.

A central role in all of the constructions and results is played by the \textit{indicial operator} of $x^2(\Delta - cS)$ that we denote by $I(x^2(\Delta - cS))$. It can be defined as follows. Consider the tubular neighborhood from Definition~\ref{def:grushin} and start stretching it by applying the change of variables $x\mapsto s x$, where $s >0$ is a parameter. Then the coefficient in front of the principal term will be an operator of the form
\begin{equation}
\label{eq:indicial_op}
I(x^2(\Delta - cS)) = (x\p_x)^2 -(1+\alpha n)x \p_x - c\alpha n(\alpha n + \alpha + 2).
\end{equation}
Note that this is true only for $\alpha>-1$. If we formally replace $x\p_x$ with $\lambda$, we obtain the \textit{indicial polynomial} $p(\lambda)$
\begin{equation}
\label{eq:ind_pol}
p(\lambda) = \lambda^2 -(1+\alpha n)\lambda - c\alpha n(\alpha n + \alpha + 2).
\end{equation}
Thanks to Proposition~\ref{prop:curvature_ass}, this polynomial does not depend on the $y$ variable nor the geometry of the singular set. In this case, one says that $x^2(\Delta-cS)$ has\textit{ constant indicial roots}, which is an important condition in the parametrix construction in many of the classical geometric calculi.

All of the properties related to the self-adjointness can be read off the indicial polynomial and its roots $\lambda_{\pm}$. Let 
$$
\mu = (1+\alpha n)^2 + 4c\alpha n(\alpha n + \alpha + 2)
$$
be the discriminant of the indicial polynomial. It turns out that the value of the discriminant codifies very well the situation when $\Delta - cS$ is an essentially self-adjoint operator.

\begin{thm}
\label{thm:self-adjoint_general}
Let $(M,\cZ,g_\alpha)$ be a $\alpha$-Grushin manifold of dimension $n+1$ with $\alpha> -1$ and let $c\in \R$. If $\mu \neq 4$ then the operator $\Delta - cS$ with domain $C^\infty_c(M\setminus \cZ)$ is essentially self-adjoint on $L^2(M,\omega)$ if and only if
$$
\mu > 4.
$$
When $\Delta - cS$ is not essentially self-adjoint, it has infinite deficiency indices.
\end{thm}

The idea of the proof is to characterize as explicitly as possible the closure (or the minimal domain) and the adjoint (or the maximal domain) of $\Delta-cS$. Recall that $D_{\min}(\Delta - cS)$ is described as the closure of $C^\infty_c(M\setminus \cZ)$ in the operator norm
$$
\|u\|_{\Delta - cS} = \|u\|_{L^2(M,\omega)} + \|(\Delta - cS)u\|_{L^2(M,\omega)},
$$
while the domain of the adjoint for a real symmetric differential operator coincides with
$$
D_{\max}(\Delta - cS) = \{u\in L^2(M,\omega)\,:\,(\Delta - cS) \in L^2(M,\omega) \}.
$$
The operator $\Delta - cS$ will be essentially self-adjoint if and only if 
$$
D_{\min}(\Delta - cS) = D_{\max}(\Delta - cS).
$$
In this case, the closure is the unique self-adjoint extension of our operator.

By the classical theorem of Von Neumann~\cite[Theorem 3.1]{sa_general}, we know that
$$
D_{\max}(\Delta - cS) = D_{\min}(\Delta - cS) \oplus \ker ((\Delta - cS)^* +i) \oplus \ker ((\Delta - cS)^* -i).
$$
The spaces $\ker ((\Delta - cS)^* \pm i)$ are known as \textit{deficiency spaces} and their dimensions, usually referred to as \textit{deficiency numbers}, are a measure of how far the closure of an operator is from being self-adjoint.

Using the $\alpha$-calculus we are able to find explicitly for $\mu \neq 4$ both the closure and the adjoint of $\Delta - cS$. Comparison between the two gives exactly the result of Theorem~\ref{thm:self-adjoint_general}. The problem with $\mu = 4$ is that the left-parametrix construction fails in this case. We still can say something about both the closure and the adjoint of $\Delta - cS$ (see Theorems~\ref{thm:closure} and~\eqref{eq:domain_max}), however, the information we are able to retrieve is not enough to determine whether this operator is essentially self-adjoint. 

If we compare Theorem~\ref{thm:self-adjoint_general} to Theorem~\ref{thm:sa_grushin_standard} when $c=0$, we get exactly the same result except the value $\alpha = 1/n$, which is not covered by Theorem~\ref{thm:self-adjoint_general}. In Theorem~\ref{thm:closure} we show that for all possible values of the discriminant $\mu \neq 4$ the closure of $\Delta - cS$ is a single functional space, which we describe explicitly. For $\mu = 4$ our technique does not work, but it is known that in the critical case for the Schr\"odinger operator with the inverse square potential its closure becomes slightly bigger~\cite{georgescu3}, so that the resulting operator is essentially self-adjoint. For $\alpha = n = 1$, $\mu = 4$ self-adjointness of $\Delta - cS$ was proven in~\cite{me} using the fact that this operator is positive. One can expect that a similar proof should hold for all the other critical cases.

Now that we have established that $\Delta-cS$ is not essentially self-adjoint when $\mu < 4$, we would also like to construct all of its self-adjoint extensions. It is known that any self-adjoint extension must lie between the minimal and maximal domains
\begin{equation}
\label{eq:domain_flag}
D_{\min}(\Delta -cS) \subset D_{s.a.}(\Delta -cS) \subset D_{\max}(\Delta -cS).
\end{equation}
Another classical Theorem by Von Neumann~\cite[Theorem 3.4]{sa_general} states that all of the self-adjoint extensions of $\Delta - cS$ are in one-to-one correspondences with unitary operators between the two deficiency spaces $\ker ((\Delta - cS)^* \pm i)$. In the case of Grushin manifolds this can be made much more explicit. The following theorem plays a central role in this description.

\begin{thm}
\label{thm:phg_Grushin}
Let $(M,\cZ,g_\alpha)$ be a $\alpha$-Grushin manifold of dimension $n+1$ with $\alpha> -1$. Let $c\in \R$ and $\mu\neq 4$. Denote 
\begin{equation}\label{eq:theta_def}
    \Theta = \{(1+\alpha)i + j \in \R_{\geq 0} \,:\, i,j\in \N_0 \}
\end{equation}
Suppose $u\in x^\delta L^2(M,\omega_\alpha)$, then solutions of 
$$
(\Delta - cS) u = 0,
$$
admit a weak asymptotic expansion of the form
\begin{equation}
\label{eq:polyhom_ass}
u = \sum_{\substack{\theta \in \Theta\\ p\in \{0,1\}}}  x^{\lambda_+ + \theta}\,(\log x)^p \, u_{\theta,p}^+(y) \; + \sum_{\substack{\theta \in \Theta\\ p\in \{0,1\} }}  x^{\lambda_- + \theta}\,(\log x)^p \, u_{\theta,p}^-(y),  \quad  u_{\theta,p}^\pm \in H^{-r_{\pm,\theta}}(\cZ)
\end{equation}
where $r_{\pm,\theta}=\lambda_\pm + (i+j)-\delta+\tfrac12$. Here we say that this is a weak asymptotic expansion in the sense that there is an expansion of the integral $\int_\cZ u(x,y)\chi(y)dy$ for any test function. A similar asymptotic expansion with $x$ replaced by $-x$ holds when $x\to 0-$. %{\color{red} Sum has coeffs of the form for $\lambda_\pm$ where $\Re(\lambda_\pm)>\delta-\tfrac{1}{2}$, hence $\delta$ suff. large implies no terms}
\end{thm}

This theorem tells us that the elements of $\ker (\Delta - cS)$ have relatively `nice' asymptotics. A priori there is no reason to believe that $u_{\pm,\theta,p}$ are smooth. However, if the principal coefficient in this asymptotic expansion is a smooth function of $\cZ$, then, in fact, all the remaining coefficients will be smooth in turn. We will show, in fact, that the graph closure of functions admitting such asymptotics are dense in the domain of the adjoint (see Theorem~\ref{thm:max_domain_full}). This explicit characterization of the domain of the adjoint is what allows us to use the method of asymmetry forms to find many natural families of self-adjoint extensions~\cite[Section 3.4]{sa_general}. It turns out that the coefficients of~\eqref{eq:polyhom_ass} in front of the $x^{\lambda_\pm}$ terms will completely determine these extensions.  

There are various approaches for construction of self-adjoint extensions. One popular method is Krein-Vi\v{s}ik-Birman theory~\cite{ale_gallone_book}; this approach was used in~\cite{pozzoli2} to construct self-adjoint extensions to the Laplace operator on Grushin cylinder for $\alpha\in (0,1)$. However, this approach requires that the operator of interest is positive-definite. While this is true for the Laplacian itself, it may not be true for perturbations via the scalar curvature. Another approach uses the theory of asymmetry forms (see, for example,~\cite{sa_general}) and unitary operators. This method can also handle non-positive operators. For example, self-adjoint extensions of the Schr\"odinger operator with an inverse square potential were completely classified in~\cite{sa_general}. When dealing with PDEs instead, it can be difficult to apply it directly even in the case of Riemannian manifolds with boundary~\cite{sa_riem}. A particular case where this method was successfully applied is self-adjoint extensions of the Laplace operator on a bounded domain of $\R^n$~\cite{Marilena}. To deal with functional analytic problems, the authors introduce a generalized Green's identity that can be used for functions in $D(\Delta^*)$, but which does not extend easily to our setting.

Instead of classifying all possible self-adjoint extensions, we recover those extensions which seem to be the most important, as was done in~\cite{pozzoli2}. Even though there is an infinite number of self-adjoint extensions, most of the time in the literature one encounters Dirichlet, Neumann, or mixed boundary conditions. In~\cite{pozzoli2}, the authors find a slightly bigger family that may be viewed as their natural generalizations. Here, we use a similar idea but via different techniques and find that the general manifold case is far richer and admits a greater variety of self-adjoint extensions. The key idea is in our study is that whenever $\Delta-cS$ has two distinct indicial roots the elements of $D_{max}(\Delta - cS)/D_{min}(\Delta - cS)$ locally must be of the form
$$
u = u^r_+ + u^r_- + u^l_+ + u^l_- \mod D_{\min}(\Delta - cS).
$$  
where
\begin{align*}
u^r_{\pm}(x,y) &= 
\begin{cases}
a^r_{\pm}(y)x^{\lambda_-}(1+ o(1)), & x>0,\\
0, & x<0;
\end{cases} \\
u^l_{\pm}(x,y) &= 
\begin{cases}
a^l_{\pm}(y)(-x)^{\lambda_+}(1+ o(1)), & x<0,\\
0, & x>0.
\end{cases}
\end{align*}
Here $u^{r,l}_\pm \in \ker (\Delta - cS)$ have the weak asymptotic expansions as in Theorem~\ref{thm:phg_Grushin} and $a^{r,l}_\pm$ are their principal coefficients, belonging to certain Hilbert spaces on the singular set $\cZ$. This description of the adjoint allows us to prove the following theorems.

\begin{thm}
\label{thm:sa_extensions_mu_leq_0}
Let $(M,\cZ,g_\alpha)$ be an $\alpha$-Grushin manifold of dimension $n+1$ and $c\in \R$. Suppose that $\mu<0$. Then, for every unitary $U:\C^2 \to \C^2$, one can assign a self-adjoint extension as follows. Let $\{\varphi_k\}_{k\in \N}$ be the ordered family of eigenfunctions of $\Delta_{0,\cZ}$, which is the Laplace-Beltrami operator on the singular set with respect to the metric $g_{0,\cZ}$ from Definition~\ref{def:grushin}. Let also $a^{r,l}_{\pm,k}$ be the coefficients of the decomposition of the principal terms $a^{r,l}_\pm$ in the basis $\{\varphi_k\}_{k\in \N}$ and define
$$
a_{\pm,k} = \begin{pmatrix}
   a^r_{\pm,k}\\
   a^l_{\pm,k}
\end{pmatrix}.
$$
Then, given a unitary operator $U:\C^2 \to \C^2$, one can construct a self-adjoint extension by taking its domain to be a subspace of $D_{\max}(\Delta - cS)$ satisfying
$$
a_{+,k} = Ua_{-,k}.
$$
\end{thm}

\begin{thm}
\label{thm:sa_extensions_mu_geq_0}
Let $(M,\cZ,g_\alpha)$ be an $\alpha$-Grushin manifold of dimension $n+1$ and $c\in \R$. Suppose that 
\begin{enumerate}
    \item $\alpha \geq 0$;
    \item $\mu \in (0,4)$ and $\mu \neq 2$;
    \item function $h := \left. \sqrt{\mu}+(\dive_{|x|^{\alpha n}\omega}(\p_x)\lambda_- - c\p_x(x^2 S))/p(\lambda_-)\right|_\cZ$ is sign-definite.
\end{enumerate}
Then, for every unitary operator $U:\C^2 \to \C^2$, one can assign a self-adjoint extension as follows. Let $\{\hat\varphi_k\}_{k\in \N}$ be the ordered family of eigenfunctions of $\hat\Delta_{0,\cZ}$, which is a Laplace-Beltrami operator on the singular set with respect to the metric $h^2 g_{0,\cZ}$. Let also $a^{r,l}_{\pm,k}$ be the coefficients of the decomposition of the principal terms $a^{r,l}_\pm$ in the basis $\{\hat\varphi_k\}_{k\in \N}$ and define
$$
A_{1,k} = \begin{pmatrix}
   a^r_{+,k}+ia^r_{-,k}\\
   a^l_{+,k}+ia^l_{-,k}
\end{pmatrix},
\qquad
A_{2,k} = \begin{pmatrix}
   a^r_{+,k}-ia^r_{-,k}\\
   a^l_{+,k}-ia^l_{-,k}
\end{pmatrix}
$$
Then, given a unitary operator $U:\C^2 \to \C^2$, one can construct a self-adjoint extension by taking its domain to be a subspace of $D_{\max}(\Delta - cS)$ for which
$$
A_{2,k} = UA_{1,k}.
$$
\end{thm}
$ $\linebreak

Let us comment on the assumptions of the last theorem. The strategy in proving both theorems will be to use a limiting version of Green's identity at the singular set \eqref{eq:form_mu0_L2} (sometimes referred to as the \emph{boundary pairing formula}) by exploiting the weak asymptotic expansion of the functions $u^{r,l}_{\pm}$. However, when $\mu>0$, the two roots $\lambda_\pm$ are real and different. Hence we need expansion of $u^{r,l}_-$ up to terms of order $o(x^{\lambda_+})$. The first condition guarantees that all of such terms are multiples of the first term, which is used in the derivation of the pairing formula. The next assumption simply removes log terms for simplicity. Finally, the last assumption guarantees that the measure in this pairing formula comes from a Riemannian metric on the singular set $\cZ$, which we use to construct our Fourier basis. The sub-leading coefficient in the expansion of $u^{r,l}_-$ is given by $\sqrt{\mu} - h a^{r,l}_-$.

In the case of Grushin cylinder, when $\alpha \in (0,1)$, $c=0$, $n=1$, we should be able to recover the results of~\cite{pozzoli2}. And indeed, in this case, there is no scalar curvature term, and $\p_x$ is divergent-free. Hence $h = \sqrt{\mu}>0$. Even though $\mu = 1$, for this concrete example, due to the particularity of the problem, there are no log terms. All self-adjoint extensions that the authors of~\cite{pozzoli2} find correspond to gluing individually $k$-th harmonics, as we do in our work. The difference is that they use the Krein-Vi\v{s}ik-Birman theory, which requires the operator of interest to be bounded from below. In our work, we do not make this assumption.

Next, we list analogues of the self-adjoint extensions from~\cite[Theorem 1.3]{pozzoli2} and their relation to the unitary operator $U$ from the Theorem~\ref{thm:sa_extensions_mu_geq_0} : 

\begin{enumerate}
\item The analogue of the Friedrichs extension:
$$
U = \id_2 ,\qquad a^r_- = a^l_- = 0;
$$
\item For $\gamma\in \R$:
$$
U = \begin{pmatrix}
\frac{\gamma - i}{\gamma + i} & 0\\
0 & 1
\end{pmatrix},
\qquad 
\left\{\begin{array}{l}
a^l_- = 0,\\
a^r_+ = \gamma a^r_-;
\end{array}\right.
$$
\item For $\gamma\in \R$:
$$
U = \begin{pmatrix}
1 & 0\\
0 & \frac{\gamma - i}{\gamma + i}
\end{pmatrix},
\qquad 
\left\{\begin{array}{l}
a^l_+ = \gamma a^l_-,\\
a^r_- = 0;
\end{array}\right.
$$
\item For $\gamma\in \R$ and $b\in \C$:
$$
U = \frac{1}{1+|b|^2 - i\gamma }\begin{pmatrix}
1-|b|^2 - i \gamma & -2b\\
-2\overline{b} & -1+|b|^2-i \gamma
\end{pmatrix},
\qquad 
\left\{\begin{array}{l}
a^r_- = b a^l_-,\\
a^l_+ + \overline{b}a^r_+  = \gamma a^l_-;
\end{array}\right.
$$
\item If $\Gamma$ is a Hermitian $2 \times 2$ matrix, then the last extension is given by
$$
U = \frac{\Gamma-i\id_2}{\Gamma+i\id_2}, \qquad 
\begin{pmatrix}
a^r_+\\
a^l_+
\end{pmatrix}
=
\Gamma \begin{pmatrix}
a^r_-\\
a^l_-
\end{pmatrix}.
$$
\end{enumerate}

Even though the authors exclude the previously studied case $\alpha =0$, where there is no singularity, the present self-adjoint extensions are valid in this case as well. In particular, the first listed self-adjoint extension ($U = \id_2$) represents Dirichlet boundary conditions at $\cZ$, while the fifth case with $\Gamma = 0$ ($U = -\id_2$) are the von Neumman conditions at $\cZ$.

In general, it is not possible to reformulate the obtained self-adjointness criteria as boundary value problems, because one has to recover principal terms of both $u^{r,l}_\pm$ which can be separated by several terms of the asymptotic expansion. However, there is a simple case in the previous theorem, namely, when $U = \id$. This would force $a^r_- = a^l_- = 0$. Since we know that those are the principal terms in the asymptotic expansion of $u\in D_{\max}(\Delta - cS)$, we obtain this particular self-adjoint extension as a boundary condition of the form
$$
\lim_{x\to 0}\frac{1}{|x|^{\lambda_-}}u(x,y) = 0,
$$
which corresponds to the Friedrichs extension. This way, we extend the results from~\cite{pozzoli2} to a much more general class of manifolds.

\subsection{The structure of the paper}

In Section~\ref{sec:alpha_results} we discuss the main results of the $\alpha$-calculus that we need for proving Theorems anounced in the previous subsection. We postpone all the proofs until Section~\ref{sec:proofs}. In Section~\ref{sec:natural_domains} we determine the closure and the adjoint of the operator $\Delta-cS$. Theorem~\ref{thm:closure} gives a complete characterization of the closure for $\mu \neq 4$ and a partial characterization for $\mu = 4$. Theorem~\ref{thm:max_domain_full} gives a characterization of the quotient $D_{\max}(\Delta - cS)/D_{\min}(\Delta - cS)$ for all possible values of the discriminant $\mu$. Both of those results are used to prove Theorems~\ref{thm:sa_extensions_mu_leq_0} and~\ref{thm:sa_extensions_mu_geq_0} concerning the self-adjointness criteria in section~\ref{sec:sa_extensions}. We recall briefly the asymmetry form method and prove the two theorems. 

A big part of the paper is Section~\ref{sec:proofs}, where we develop the $\alpha$-calculus, construct the necessary double and triple spaces, prove composition theorems and discuss the parametrix construction. This calculus is very similar to the $0$-calculus and its generalization, the edge calculus~\cite{mazzeo}. For this reason, in some cases we only sketch the proofs to emphasize the differences, while the remaining parts of the arguments are proved word-by-word like in the $0$-calculus.

Finally, in Appendix A we compute the scalar curvature for flat models as well as asymptotics of the scalar curvature for general $\alpha$-Grushin manifolds.

\section{Principal results of the $\alpha$-calculus}

\label{sec:alpha_results}

In this section we discuss concepts related to the construction of the $\alpha$-calculus. Let $M$ be an $\alpha$-Grushin manifold. We cut this manifold along the singular set $\cZ$ and double the boundary to obtain a new manifold with the boundary $\cZ \cup \cZ$. Abusing the notation we call this manifold again $M$. 

We define the {\bf $\alpha$-vector fields} as follows: if $M$ is endowed with a metric as in \eqref{eq:metric_grushin}, we say
\[ \cV_\alpha(M) = \{ W\in C^0(M; TM) \, | \, W\in C^\infty(M^\circ; TM^\circ), \; |W|_{g_\alpha}^2 = \cO(x^2) \} \]
i.e., they are continuous vector fields on $M$, which are smooth in the interior, vanish at the boundary, and those which are $g_\alpha$-orthogonal to the boundary unit normal $N$ vanish to order $(1+\alpha)$, (our assumption that $\p M=\cZ\cup \cZ$ is co-orientable is equivalent to the boundary unit normal being trivialized). The $\alpha$-calculus is a tool to study operators in the universal enveloping algebra of $\cV_{\alpha}(M)$, which in local coordinates near $\partial M$ can be expressed as 
\begin{equation}
\label{eq:operators_of_interest}
    \sum_{j,\beta} a_{j,\beta}(x,y) (x\p_x)^j (x^{1+\alpha}\p_y)^\beta,
\end{equation}
where $(x,y)$ are such that $x$ is a boundary defining function for $\partial M$ (i.e. that $\partial M=\{x=0\}$ and $dx|_{\partial M}\neq 0$) and $(y_1,\ldots, y_n)$ are coordinates on $\partial M$. From this form of the operator, we see that a local frame for the Lie algebra $\cV_\alpha(M)$ is given by
\[ \cV_\alpha(M) = \text{span}_{C^\infty(M^\circ)}\{x\partial_x,\, x^{1+\alpha}\partial_{y_j}\}. \]
In general the ``nice" operators in our calculus will be elliptic combinations of these vector fields, which have some control at infinity.

The animating idea between various geometric pseudodifferential calculi is that one should consider the kernels of differential operators not as distributions living on $M^2$, but as certain distributions living on its desingularization $M_\alpha^2$ called the double stretched product. In order to pass from $M^2$ to $M^2_\alpha$ we need to perform a blow-up of a certain submanifold. Note that $M^2$ is a manifold with corners and hence the correct notion of blow-up is given by quasi-homogeneous blow-up of $p$-submanifolds. This blow-up procedure was thoroughly studied in~\cite{thesis}, where the orders were assumed to be integer. We define the quasi-homogeneous blow-up for non-integer orders at the cost of the blow-down map being not smooth (while remaining manageable).

As $\cV_{\alpha}$ is a finitely generated projective module over $C^\infty(M),$ we can use the 
Serre-Swan theorem, or proceed directly as in \cite[\S8.2]{melrose1993atiyah}, 
and find that there is a vector bundle
\begin{equation}\label{alpha-tgtbdl}
	{}^{\alpha}TM \to M,
\end{equation}
together with a bundle map $j_\alpha:{}^{\alpha}TM \to TM$ such that 
$(j_\alpha)_*C^\infty(M;{}^{\alpha}TM) = \cV_{\alpha} \subseteq C^0(M;TM).$
Eliding the map $j_\alpha$, we say that the space of sections of ${}^{\alpha}TM$ is $\cV_{\alpha}.$
We will refer to ${}^{\alpha}TM$ as the {\it $\alpha$-tangent bundle}. One connection between these last two geometric constructions is that ${}^\alpha TM$ arises now as the normal bundle to the diagonal in the desingularized $M_\alpha^2$.

We discuss the blow-up construction in Subsection~\ref{subsec:quasi_blow-up}. The double-stretched product $M_\alpha^2$ is defined in Subsection~\ref{subsec:double_prod}. Finally, the polyhomogeneous distributions are discussed in Subsection~\ref{subsec:polyhomog_as} and the parametrix construction is given in Subsection~\ref{subsec:parametrix_statement}. The parametrix construction lies at the center of our proofs. It is adapted for studying elliptic operators of the form~\eqref{eq:operators_of_interest},
where due to the singularity at $x=0$, ellipticity is not sufficient for having a good parametrix. Most of the proofs from this section are postponed till Section~\ref{sec:proofs}.

\subsection{Manifolds with corners and quasi-homogeneous blow-up}

\label{subsec:quasi_blow-up}

In this subsection we review the notion of quasi-homogeneous blow-ups. We illustrate the definitions via some simple examples to underline the main differences with homogeneous blow-ups. However, the construction conceptually is the same. For a deeper look see the thesis~\cite{thesis}, as well the survey \cite{grieser2001basics} for a more general introduction to blow-ups in microlocal analysis. As we will see later, in contrast to~\cite{thesis} in this work we allow for non-integer blow-up orders, but the difference in the proofs is minimal.

The natural differential geometric structure for many geometric pseudodifferential operator calculi is given by smooth manifolds with corners. These are second-countable Hausdorff topological spaces modeled by the standard Euclidean corners $\R^n_l = \R^{n-l} \times [0,+\infty)^l$. If $M$ is a manifold with corners, then for every point $q\in M$ we can define the inward-pointing tangent cone $T^+_q M$ in a natural way.

\begin{definition}
Let $M$ be a manifold with corners and $\cB\subset \p M$ a boundary hypersurface. A non-negative function $r: M\to \R$ is called a \textit{boundary defining function} of $\cB$ if
\begin{enumerate}
\item $r(q) = 0$ if and only if $q\in \cB$;
\item $dr(v) \neq 0$, for all $v\in T_+ M$ transversal to $\cB$. 
\end{enumerate}
\end{definition}

Assume first that we would like to blow-up $\R^n_l$ at zero with different rates in different directions. To do so, we need to define the blow-up data which consists of a filtration
$$
\{0\} = V_0 \subset V_1 \subset \dots \subset V_k = \R^n_l, 
$$
such that each $V_i$ is a subcorner, and a vector $\Lambda=(\lambda_1,\dots,\lambda_k)$ such that $0< \lambda_1 < \dots < \lambda_k$. Roughly speaking we would like to stretch $\R^n_l$ with order $\lambda_i$ in the direction $V_i/V_{i-1}$. However, unlike the standard homogeneous blow-up we can only choose filtrations that satisfy certain conditions.

\begin{definition}
\label{def:p-filtration}
Let $I \subset \{1,\dots,n\}$. Let
$$
l_I = \{(x_1,\dots,x_l)\in \R^0_l :x_i = 0 \text{ if } i\in I\}.
$$
A \textit{$p$-filtration} of step $k$ is a filtration
$$
\{0\} = V_0 \subset V_1 \subset \dots \subset V_k = \R^n_l, 
$$
such that $V_i \simeq l_{I_i} \oplus W_i$, where $W_i \subset \R^n$ is an arbitrary subspace and 
$$
\dim W_i -\# I_i <  \dim W_{i+1}-\# I_{i+1}.
$$
\end{definition}

Let us comment on this definition. If there were no corners, we could simply take $V_i$ to be arbitrary spaces because the quotients $V_i/V_{i-1}$ could be identified with subspaces complementing $V_{i-1}$ in $V_i$. In the presence of corners, the quotients $V_i/V_{i-1}$ can be identified with corners that complement $V_{i-1}$ in $V_i$ only if they satisfy the Definition~\ref{def:p-filtration}. Figure~\ref{fig:p-filtrations} illustrates some simple examples.

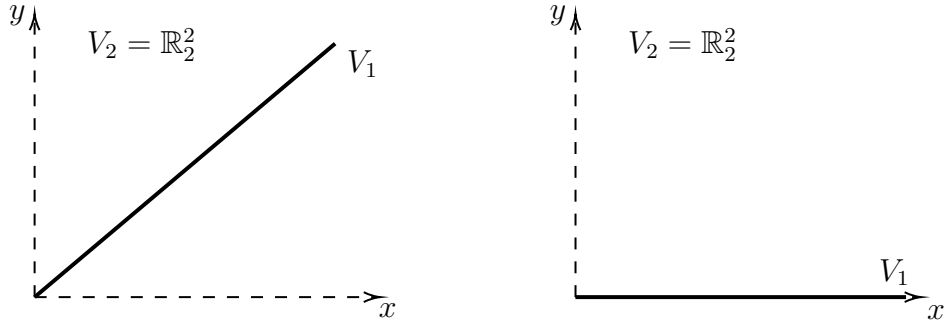
\begin{figure}
    \centering
    \tikzset{every picture/.style={line width=0.75pt}} %set default line width to 0.75pt        

\begin{tikzpicture}[x=0.75pt,y=0.75pt,yscale=-0.75,xscale=0.75]
%uncomment if require: \path (0,300); %set diagram left start at 0, and has height of 300

%Straight Lines [id:da9782658339701678] 
\draw [line width=0.75]  [dash pattern={on 4.5pt off 4.5pt}]  (40,220) -- (40,32) ;
\draw [shift={(40,30)}, rotate = 90] [color={rgb, 255:red, 0; green, 0; blue, 0 }  ][line width=0.75]    (10.93,-3.29) .. controls (6.95,-1.4) and (3.31,-0.3) .. (0,0) .. controls (3.31,0.3) and (6.95,1.4) .. (10.93,3.29)   ;
%Straight Lines [id:da18261724081969177] 
\draw [line width=0.75]  [dash pattern={on 4.5pt off 4.5pt}]  (40,220) -- (268,220) ;
\draw [shift={(270,220)}, rotate = 180] [color={rgb, 255:red, 0; green, 0; blue, 0 }  ][line width=0.75]    (10.93,-3.29) .. controls (6.95,-1.4) and (3.31,-0.3) .. (0,0) .. controls (3.31,0.3) and (6.95,1.4) .. (10.93,3.29)   ;
%Straight Lines [id:da75298571493482] 
\draw [line width=1.5]    (240,50) -- (40,220) ;
%Straight Lines [id:da4293422800684097] 
\draw [line width=0.75]  [dash pattern={on 4.5pt off 4.5pt}]  (400,220) -- (400,32) ;
\draw [shift={(400,30)}, rotate = 90] [color={rgb, 255:red, 0; green, 0; blue, 0 }  ][line width=0.75]    (10.93,-3.29) .. controls (6.95,-1.4) and (3.31,-0.3) .. (0,0) .. controls (3.31,0.3) and (6.95,1.4) .. (10.93,3.29)   ;
%Straight Lines [id:da8197878337336229] 
\draw [line width=0.75]  [dash pattern={on 4.5pt off 4.5pt}]  (400,220) -- (628,220) ;
\draw [shift={(630,220)}, rotate = 180] [color={rgb, 255:red, 0; green, 0; blue, 0 }  ][line width=0.75]    (10.93,-3.29) .. controls (6.95,-1.4) and (3.31,-0.3) .. (0,0) .. controls (3.31,0.3) and (6.95,1.4) .. (10.93,3.29)   ;
%Straight Lines [id:da6872834059275708] 
\draw [line width=1.5]    (620,220) -- (400,220) ;

% Text Node
\draw (246,52.4) node [anchor=north west][inner sep=0.75pt]    {$V_{1}$};
% Text Node
\draw (73,37.4) node [anchor=north west][inner sep=0.75pt]    {$V_{2} =\mathbb{R}_{2}^{2}$};
% Text Node
\draw (600,192.4) node [anchor=north west][inner sep=0.75pt]    {$V_{1}$};
% Text Node
\draw (433,37.4) node [anchor=north west][inner sep=0.75pt]    {$V_{2} =\mathbb{R}_{2}^{2}$};
% Text Node
\draw (267,222.4) node [anchor=north west][inner sep=0.75pt]    {$x$};
% Text Node
\draw (632,223.4) node [anchor=north west][inner sep=0.75pt]    {$x$};
% Text Node
\draw (381,22.4) node [anchor=north west][inner sep=0.75pt]    {$y$};
% Text Node
\draw (21,22.4) node [anchor=north west][inner sep=0.75pt]    {$y$};

\end{tikzpicture}
    \caption{Example of a $p$-filtration $V_1 \subset V_2 = \R^2_2$ (right) and a non-example (left)}
    \label{fig:p-filtrations}
\end{figure}

Let us first blow-up the origin in $\R^n_l$. The blow-up data $(\cF,\Lambda)$ contains a $p$-filtration via subcorners $V_i$ and a positive vector $\Lambda=(\lambda_1,\dots,\lambda_k)$ such that $\lambda_{i+1} > \lambda_i$. Choose subcorners $W_i$ which complement $V_{i-1}$ in $V_i$. If $x_i \in W_i$, we define the $\R^+$-action on $W_i$ as 
\begin{equation}
\label{eq:dilations}
\delta_t(x_i) = t^{\lambda_i} x_i
\end{equation}
and extend the action to the whole $\R^n_l$ by linearity.

This gives a smooth $\R^+$-action on $\R^n_l\setminus\{0\}$ and an equivalence relation, Namely $q\sim q'$ if and only if $\delta_t q = q'$ for some $t>0$. 
\begin{definition}
Let $(\cF,\Lambda)$ be the blow-up data, where $\cF$ is a $p$-filtration of step $k$ and $\Lambda=(\lambda_1,\dots,\lambda_k)$ with $\lambda_{i+1}>\lambda_i$, and $\delta_t$ as defined in the equation~\eqref{eq:dilations}. The blow-up of $\R^n_l$ at $0$ with the blow-up data $(\cF,\Lambda)$ as the set is given by
$$
[\R^n_0;0]_{\Lambda}= \left( (\R^n_l \setminus \{0\}) / \delta_t\right) \sqcup \R^n_l \setminus \{0\}.
$$

The blow-down map
$$
\beta: [\R^n_l;0]_{\Lambda} \to \R^n_l
$$
acts as the identity on $\R^n \setminus \{0\}$ and maps the whole $\left( (\R^n \setminus \{0\}) / \delta_t\right)$ to the origin. 
\end{definition}

\begin{prop}
The following statements hold:
\begin{enumerate}
\item $[\R^n_0;0]_{\Lambda}$ is a smooth manifold with corners;
\item The construction does not depend on the choice of complements, i.e. two blow-ups using two different complements are diffeomorphic;
\item Blow-down map $\beta$ is smooth if $\Lambda$ has integer values and continuous otherwise;
\item $\beta^{-1}$ is a diffeomorphism when restricted to $\R^n \setminus \{0\}$.
\end{enumerate}
\end{prop}
The proofs of 1), 3) and 4) come from explicit constructions of coordinates which we describe next. The proof of 2) can be found in~\cite{thesis}. 

As in the case of homogeneous blow-ups we have two good coordinate systems. One system generalizes the spherical coordinates; the other system is an analog of projective coordinates. Note that we can extend $\R^n_l$ to $\R^n$ by gluing together several of its copies. Let $Q$ be a positive-definite quadratic form on $\R^n$. We can define a pseudo-norm:
$$
\|x\| = \sqrt{\sum_{i=1}^k Q(x_i)^{\frac{1}{a_i}}}.
$$
This pseudo-norm is well defined on $\R^n_l$ as well. Now we can define coordinates on $[\R^n;0]_{(\cF,\Lambda)}$ as follows. We take the level set 
\begin{equation}
    \label{eq:level_set}
    S(\R^n_l)=\{x\in \R^n_l:\|x\|=1\}
\end{equation}
which is a smooth manifold by construction. Then we can describe the blown-up space as
$$
[\R^n_l;0]_{(\cF,\Lambda)} = S(\R^n_l) \times \R^+,
$$
where $\R^+$ variable correspond on $\R^n$ to the value of the pseudo-norm $\|x\|$. Thus we have a trivial fibration over $\R_+$ and thus one can construct charts from the charts of $S(\R^n_l)$. This is essentially a generalization of the polar coordinates. 

Let $\theta$ be a point in $S(\R^n_l)$. If we view $S(\R^n_l)$ as a subset of $\R^n_l$ as in~\eqref{eq:level_set}, then we can define projections $\theta_i$ of $\theta$ to the complements $W_i$. The blow-down map $\beta$ is given by
\begin{equation}
    \label{eq:spherical_coordinates}
    \beta(r,\theta) = (r^{\lambda_1}\theta_1,\dots,r^{\lambda_k}\theta_k).
\end{equation}
To construct the blow-up as a singular change of variables, we can define
$$
(r,\theta)(x) = \left(\|x\|,\frac{x_i}{\|x\|^{\lambda_i}} \right). 
$$

It is often more convenient to use projective coordinates instead. By slight abuse of notation let $x_j$, $j=1,\dots,n$ be coordinate functions adapted to the choice of the complements $W_j \simeq V_j/V_{j-1}$. We can define their order as $\text{ord}(x_j)=\lambda_j$. Assume that we are in a region $x_j \neq 0$ for a fixed coordinate function $x_j$. Then we can define a pair of charts  
\begin{equation}
    \label{eq:projective_coordinates}
    \xi_j = \begin{cases}
\pm |x_i|^\frac{1}{\lambda_i}, & \text{if } j= i,\\
\frac{x_j}{|x_i|^{\lambda_i/\lambda_j}}, & \text{if } j\neq i;
\end{cases}
\end{equation}
where we use two signs $\pm$ if the subset $\{x_k = 0, \forall k\neq j\}$ is a line and only plus if the same subset is a semi-line.

Repeating over all possible $j=1,\dots,n$ will give us a system of coordinate charts. Moreover, they give rise to a smooth manifold structure. It is also straightforward to check that the smooth structure via spherical coordinates and projective coordinates are diffeomorphic.

Similarly, if we wish to blow-up along submanifolds, we need to be careful that the normal bundle is well-defined, since the normal cone can vary from point to point. Thus again we need to restrict to a special class of submanifolds.

\begin{definition}
\label{def:p-manifold}
Given $I \subset \{1,\dots,n\}$, define
$$
L_I = \{x=(x_1,\dots,x_n)\in \R^n_l : x_i = 0, i\in I\}.
$$
A \textit{$p$-submanifold} of a manifold with corners $M$ is a subset $P\subset M$, such that for each $x\in P$ there exists a chart $(U,\varphi)$ around $x$ and an index set $I$ for which
$$
\varphi(P) = \varphi(U) \cap L_I.
$$
Every $p$-submanifold has a well-defined normal bundle, and hence a \textit{$p$-filtration} on a manifold $M$ along a submanifold $P$ is defined as smoothly varying family of $p$-filtrations of the normal bundle. 
\end{definition}

Whenever we have a blow-up data $(\cF,\Lambda)$, where $\cF$ is a $p$-filtration along a $p$-submanifold $P$ of $M$, we can define the blow-up $[M;P]_{(\cF,\Lambda)}$ as the disjoint union of $M\setminus P$ with the spherical normal bundle of $P$. The blow-down map is again a diffeomorphism onto $M\setminus P$ and we can use charts of $M$ to construct some charts of $[M;P]_{(\cF,\Lambda)}$. In order to construct charts close to the spherical bundle, we identify the normal bundle with a tubular neighborhood of $P$ via an exponential map and use either spherical or projective local coordinates. It is necessary to prove that such local coordinates can be glued together to form a system of charts. This is indeed the case as the proof of~\cite[Thm 2.26]{thesis} shows and which holds for non-integer elements as well. The result of this blow-up will be a smooth manifold with corners that has one more boundary hypersurface $\beta^{-1}(P)$ with respect to the original manifold $M$. This hypersurface is usually referred to as the \textit{front face}. Note again that in the case of non-integer weights $\lambda_i$, the blow-down map is not smooth in general, but it is a smooth bijection when restricted to $M\setminus P$.

\subsection{The double stretched product}

\label{subsec:double_prod}

We now construct the double stretched product $M_\alpha^2$, where the kernels of the constructed pseudo-differential calculus live. The double stretched product in the $0$-calculus is just a special case for $\alpha = 0$.

Consider a $\alpha$-Grushin manifold. By cutting along the singular set $\cZ$ we obtain a compact manifold $M$ with boundary $\p M = \cZ \cup \cZ$. $\cZ$ is co-orientable, and hence the normal vector field $X$ at $\cZ$ is well defined. Essentially we want to blow-up along this vector field with order $1$ and along the boundary with order $1+\alpha$. In order to make everything consistent with the definition of blow-up given previously, we consider two separate cases: when $-1<\alpha<0$ and $\alpha\geq 0$. Both cases are very similar, and for this reason we explain the second one in-depth and only sketch the first one.

Assume that $\alpha \geq 0$. Consider the annihilator of a vector field $X$ of the normal to a boundary component. It is given by a collection of $n-1$ one forms. For convenience, we group these 1-forms into a single vector-valued form $\sigma$, which will define $X$ uniquely modulo a scale factor due to the relation $X=\ker \sigma$. Close to the boundary we can always choose coordinates $(x,y) \in \R^n$, such that $X=\p_x$. Then $\sigma =  dy$ locally.

We define the double-stretched product via the following construction. Denote by $\Delta$ the diagonal of $M\times M$. First note that $\p \Delta \subset M \times M$ is a $p$-submanifold. Indeed, in the local coordinates it is given by the set $(0,y,0,y)$. Let $\pi_L,\pi_R:M\times M \to M$ be projections to the left and right factors. We can lift $\sigma$ naturally to a form on $M\times M$ by defining
$$
\sigma_2 = \pi_l^* \sigma - \pi_r^* \sigma.
$$
This is a vector-valued one-form that lives in the annihilator of $\partial\Delta$, which is the dual space to the normal bundle $N(\partial \Delta)$. Thus we can define a subbundle $\ker \sigma_2 \subset N(\partial \Delta) $ and assign to it weight $1$ and the weight $1+\alpha$ to all $N(\partial \Delta)$. The filtration $\cF$
$$
{0}\subset \ker \sigma_2 \subset N(\p \Delta) 
$$
is indeed a $p$-filtration over $\p \Delta$. We assign to this filtration the order vector $\Lambda = (1,1+\alpha)$. This gives the blow-up data $(\cF,\Lambda)$.
\begin{definition}
\label{def:double_product}
$\alpha$-double stretched product is defined as
$$
M^2_\alpha = [M^2;\p\Delta]_{(\cF,\Lambda)}.
$$
\end{definition}

The structure of this resolved space is more clear in various coordinate systems. As previously mentioned, close to the boundary we have coordinates $(x,y)$, where $x=0$ define locally the boundary $\p M$. This gives us coordinates on $M^2$ given by ${(x,y,\tilde{x},\tilde{y})}\in \R^{2(n+1)}$ and in these coordinates $\p\Delta = {(0,y,0,y)}$. Also, we can assume that the normal bundle is given by 
$$
N(\p\Delta) = \spann\{\p_{x},\p_{\tilde x},\p_{y}-\p_{\tilde y}\}.
$$ 
In these coordinates the one form $\sigma_2$ is written as
$$
\sigma_2 = dy-d\tilde{y}.
$$
Hence $\ker \sigma_2 = \spann\{\p_x,\p_{\tilde{x}}\}$.

Thus centering our coordinates at a specific point in $\p \Delta$, we can define local coordinates for $M^2_\alpha$ as $(r,\theta)$, where
\begin{equation}\label{polar-coords}
    r= (x^{2(1+\alpha)} + |y-\tilde{y}|^2 + \tilde x^{2(1+\alpha)})^{\frac{1}{(2(1+\alpha))}} , \quad \theta = \left(\frac{x}{r},\frac{y-\widetilde{y}}{r^{(1+\alpha)n}},\frac{\widetilde{x}}{r}\right) = (\theta_x,\theta_y,\theta_{\widetilde{x}}) .
\end{equation}
and $\theta$ are coordinates on the level-set $r=1$.

Recall that a blow-up produces a new boundary face. Thus $M^2_\alpha$ has three faces. There are two faces coming from $\p M \times M$ and $M \times \p M$ which are called the left and right faces and are denoted $B_{10}$ and $B_{01}$ correspondingly. The third face is the preimage of $\p \Delta$ under the blow-down map which is denoted by $B_{11}$.

\begin{figure}[h]
\includegraphics[scale=.5]{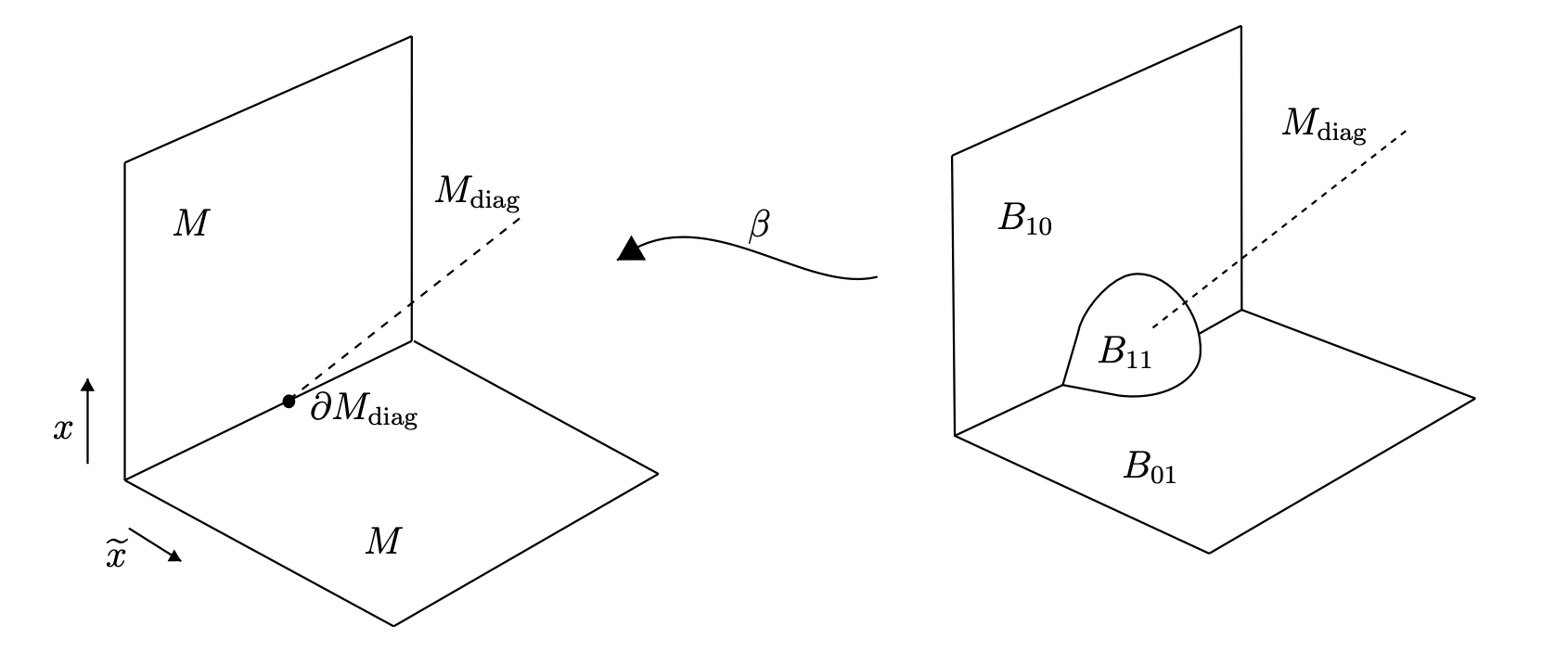}
\caption{The blow-down map $\beta$ of the $\alpha$-stretched product space $M\times_\alpha M$}
\end{figure}

Now assume that $-1<\alpha < 0$. Since $\cZ$ was a co-orientable sub-manifold, for each of the connected components of $\p M$ one can define $T \p M$ as an annihilator of a one-form $\tilde{\sigma}$. Again we lift this form to $M^2$ by taking
$$
\tilde \sigma_2 = \pi_l^* \tilde \sigma - \pi_r^* \tilde \sigma.
$$ 
The new filtration $\tilde \cF$ is given by
$$
0 \subset \ker \tilde\sigma_2 \subset N(\p \Delta)
$$
and the new order vector $\tilde{\Lambda}=(1+\alpha,1)$. It is not difficult to check that we get exactly the same coordinates. The reason for this is that a tubular neighborhood of $\cZ$ allows us to write close to the boundary $M \simeq \cZ \times \R_+$, which already gives a natural splitting of the corresponding normal bundle to the diagonal.

\subsection{Polyhomogeneous asymptotics}

\label{subsec:polyhomog_as}

In the definition of the large $\alpha$-calculus we will use operator kernels that have certain asymptotic expansions in the neighborhoods of the boundary hypersurfaces.

We recall first the definition of conormal functions: Let $\cV_b(X)$ be the space of smooth vector fields tangent to the boundary of $X$. We define the space of conormal functions to be 
\[ \cA^0(W) = \{ u\in C^\infty(W^\circ) : V_1\ldots V_k u\in L^\infty(W),\, V_j\in \cV_b(W), \forall k\in \N \}, \]
and further define
\[ \cA^\ast(X)=\bigcap_{s\in \C, p\in \N} r^s (\log(r))^p \cA^0(X). \]

\begin{definition}
An \textit{index set} $E$ is a discrete subset of $\C \times \N_0$ such that
\begin{equation}\label{eq:indexsetfinitetails}
(s_j,p_j)\in E, \qquad |(s_j,p_j)| \to \infty \Rightarrow \Re(s_j) \to +\infty.    
\end{equation}
An index set is said to be \textit{smooth} if for every $(s,p)\in E$, indices $(s+k,p-l)\in E$ for all $k,l\in \N_0$ with $l\leq p$. 
\end{definition}

The index sets are used to define the space of functions with good asymptotic properties. Assume that $X$ is a manifold with a single boundary component, with boundary defining function $r$, and with associated index set $E$. Given these definitions, the space of \textit{conormal polyhomogeneous distributions with index set $E$} consists of distributions on $X$ which have the following asymptotic expansions. Given $r$ a defining function for $\p X$ we say that given
\begin{equation}\label{eq:simplephg_asymptotics}
  u(r,y) \sim \sum_{(s,p)\in E}r^{s} (\log r)^{p}\, u_{sp}(y),  \quad u_{sp}\in C^\infty(\p X)
\end{equation}
where this asymptotic is interpreted to mean that for each $N\in \mathbb{N}$, 
\begin{equation}\label{eq:asympsumdef}
    u - \sum_{\substack{ (s,p)\in E \\ \Re(s)\leq N}} r^s (\log r)^p \, u_{s,p}(y) \in \dot{C}^N\left([0,1), C^\infty(\p X)\right)  
\end{equation}  

We denote this space of distributions by $\cA^E_{phg}(X)$. If there are more boundary components, then we can impose the previous asymptotics at each boundary component with their own index set.

More generally, assume that manifold $X$ is a manifold with corners and has $k$ codimension one boundary faces $\cB_j$, $j= 1,\dots ,k$. Let $r_k$ be the corresponding defining functions and for each $\cB_j$ assign an index set $E_j$. Denote $\cE = (E_1,\dots,E_k)$ and $\cE(k)$ for the collection of index sets from the boundaries intersecting $\cB_k$. The space of conormal polyhomogeneous distributions $\cA^\cE(X)$ is constructed inductively as a space of distributions which close to the boundary faces $\cB_j$ have expansions

\begin{equation}
    \label{eq:phg_asymptotics}
    u \sim \sum_{(s,p)\in E_k}r_k^{s} (\log r_k)^{p} u_{sp},
\end{equation}
where $u_{sp} \in \cA_{phg}^{\cE(k)}(\cB_k)$. More generally we denote the space of functions of \emph{all possible} polyhomogeneous conormal expansions with respect to some arbitrary smooth index set by $\cA_{phg}^*(X)$.

\begin{remark}
    Note that higher order coefficients $u_{sp}$ are not invariantly defined in general. Nevertheless, the principal term of the expansion is invariant.
\end{remark}

For example, the space of smooth extendable functions across the boundary is a subspace of $\cA^{0}_{phg}(X) $, the spaces of smooth functions which vanish up to order $k$ is a subspace of $\cA^k_{phg}(X)$. Smooth functions which vanish to infinite order form a subspace of conormal polyhomogeneous distributions with the empty index set. But by the analogy with the previous example, we denote this space as $\cA^\infty_{phg}(X)$.

Because the definition of $\cA_{phg}^\cE(X)$ is defined inductively, in order to check that a function $u$ belongs to $\cA_{phg}^{\cE}$ it is often simplest to apply the following criterion, proven in \cite{melrose_book}
\begin{lem}\label{lem:phg-crit}
Assume $u\in \cA_{phg}^{\ast}(X)$, and at every boundary face $\cB_k$ of $X$ there is an expansion of $u$ of the form \eqref{eq:phg_asymptotics}, with coefficients merely belonging to $\cA^\ast(\cB_k)$, but with exponents belonging to $E_k$. Then $u$ in fact belongs to $\cA_{phg}^\cE(X)$.
\end{lem}

\begin{remark}
The spaces $\cA^{\cE}_{phg}(M)$ are very natural for the goal of this article, mainly for studying the operator $\Delta - cS$. As we have already mentioned in the introduction, the indicial operator $I(x^2(\Delta - cS))$ plays an important role in all of our constructions. As formula~\eqref{eq:indicial_op} shows, this is just a second-order ODE with a regular singular point. The classical theory of such equations~\cite{ode} implies that solutions to $I(x^2(\Delta - cS))u = 0$, should have exactly the asymptotics~\eqref{eq:phg_asymptotics}. Later we will see that this kind of asymptotics holds also for the solution of the equation $(\Delta-cS)u=0$ and, more generally, for $\alpha$-elliptic differential operators under certain additional assumptions.

Another useful property of the spaces $\cA_{phg}^{\cE}(X)$ is that, much like the standard classes of symbol functions, they obey a form of asymptotic completeness. As in the classical case this is a consequence of a lemma of Borel which here can be written in the following way.
\end{remark}

\begin{lem}\label{lem:borellemma}
    Let $X$ be a compact manifold with boundary, and $E$ a corresponding index set. If $u_{s,p}\in C^\infty(X)$ is given for each $(s,p)\in E$ then there exists $u\in \cA_{phg}^E(X)$ satisfying \eqref{eq:simplephg_asymptotics}. Further, if $u'\in \cA_{phg}^E(X)$ is another element with the same asymptotic expansion then $u-u'\in \dot C^\infty(X)$.
\end{lem}
\begin{proof}
    Choose $\chi\in C_c^\infty(\mathbb{R}_{\geq 0})$ with $\chi|_{\{t<1/2\}}= 1$ and $\chi|_{\{t>1\}}= 0$. We shall prove the existence of a sequence of constants $\epsilon_{s,p}\in (0,1)$ which decrease sufficiently fast to ensure that for each $N\in \mathbb{N}$ the series
    \begin{equation}\label{eq:borel-converge}
        \sum_{\substack{(s,p)\in E\\ \Re(z)>N}} \chi\left(\frac{r}{\epsilon_{s,p}}\right)\, r^s (\log r)^k \, u_{s,p} \; \text{converges absolutely in }\dot{C}^N(X).  
    \end{equation} 
    Namely, for each $N$, there exists a sequence of constants $\epsilon_{s,p}^{(N)}$ such that \eqref{eq:borel-converge} holds for that $N$ whenever $\epsilon_{s,p}<\epsilon_{s,p}^{(N)}$ for all $\Re(z)>N$. From the definition of an index set, we know that \eqref{eq:indexsetfinitetails} holds, thus any `left segment' $E\cap\{(s,p):\Re(s)<N\}$ is a finite set for each $N$. From this fact, all these conditions imposed on the sequence of constants $\epsilon_{s,p}$ is a finite number of conditions for each $N\in \mathbb{N}$; choosing a diagonal subsequence these constants can be chosen such that \eqref{eq:borel-converge} holds for all $N$. Thus the series converges asymptotically (i.e. in the sense of \eqref{eq:asympsumdef}) to an element $u\in \cA_{phg}^E(X)$ which satisfies \eqref{eq:simplephg_asymptotics}. The claim of uniqueness follows from the definition of $\cA_{phg}^E(X)$ and the fact that
    \[ \dot C^\infty(X) = \bigcap_{N\in \mathbb{N}} \dot{C}^N(X). \]
\end{proof}

\begin{remark} We observe for later purposes that if $u$ satisfies the asymptotic expansion \eqref{eq:simplephg_asymptotics} only \emph{in the weak sense}, i.e. that there is an expansion of this type for the pairing $\int_\cZ u(x,y)\chi(y)dy$ for any smooth test functions on $\cZ$, then  previous lemma holds as well. This is because for any such pairing, the integral $\int_\cZ u(x,y)\chi(y)dy$ is polyhomogeneous conormal in the strong sense, and satisfies \eqref{eq:asympsumdef}, where the difference now lies in $\dot{C}^N([0,1))$. 

Arguing similarly as in the lemma also proves the asymptotic completeness for manifolds with corners, and the spaces $\cA_{phg}^\cE(X)$ of polyhomogeneous conormal functions with respect to index families $\cE$ corresponding to the boundary hypersurfaces of $X$. \\
\end{remark}

The space $\cA^{*}_{phg}$ of such functions with unrestrained asymptotics will be used to construct parametrices to singular differential operators. But first, we need to understand how those spaces change under various mappings. These mappings, however, must have certain properties in order to send polyhomogeneous distributions to polyhomogeneous distributions. First, we consider those maps that can be used as pull-backs.

\begin{definition}
Let $M,N$ be two smooth manifolds with corners and let  $r_j$, $\rho_i$ be the corresponding sets of defining functions of boundary hypersurfaces. A continuous map $f: M \to N$ is called a \textit{$b$-map} if it has a smooth restriction to the interior of $M$ and there exist non-negative real numbers $e(i,j)$ and a smooth non-vanishing function $h\in C^\infty(M)$ such that
$$
f^*(\rho_i) = h \prod r_j^{e(i,j)}.
$$ 
The set of numbers $e(i,j)$ is called the \textit{lifting matrix}.
\end{definition}

\begin{remark}
Note that in our definition a $b$-map may not be smooth up to the boundary. However, $b$-maps we use in this article restrict to diffeomorphisms in the interior of the two manifolds, and the failure of smoothness up to the boundary will be encoded by a possible non-integer `lifting matrix' $e(i,j)$. This allows us to pull-back smooth densities, which in general will fail to be smooth in a very controlled way at the boundary faces. Moreover, the definition of polyhomogeneous expansions already includes non-integer weights. This means that there will be no problem in defining kernels and the construction of parametrices.  
\end{remark}

A principal property of $b$-maps is that they preserve polyhomogeneity. Given two manifolds with corners $M,N$ let $f: M \to N$ be a $b$-map with lifting matrix $e(i,j)$. If $\cF$ is a collection of index sets of $N$, define the following index set $\cE = f^{\#}(\cF)$ as follows:
\begin{equation}
E_j = \left\{\left(\sum_{i}e(i,j)s_i,\sum_{e(i,j)\neq 0}p_i \right): (s_i,p_i)\in F_i \right\}.
\end{equation}

\begin{thm}[\cite{melrose_book}]
Let $f: M \to N$ be a $b$-map and $v\in \cA_{phg}^\cF(N)$. Then $f^* v \in \cA_{phg}^{f^\# (\cF)}(M)$.
\end{thm}

Now we consider the push-forwards. In this case, it is not enough to consider only $b$-maps. More precisely we have to restrict to the following classes of maps.

\begin{definition}
A $b$-map $f:M \to N$ is a \textit{$b$-fibration} if it does not map any boundary hypersurface to a corner and it is a fibration in the interior of each boundary hypersurface.
\end{definition}

\begin{remark}
A reader familiar with the construction of other geometric pseudodifferential operator calculi is probably more used to the definition using the $b$-differential and could naturally ask whether the $b$-differential makes sense in this situation. Given a manifold with corners, we can construct the $b$-tangent bundle ${}^b TM$ and the $b$-cotangent bundle ${}^b T^*M$ in a manner similar to ${}^\alpha TM$ and ${}^\alpha T^*M$ similar to~\eqref{alpha-tgtbdl}. The $b$-tangent bundle comes with a natural bundle map $e_b:{}^b TM \to TM$, called the anchor. When restricted to the interior, the anchor is a bundle isomorphism. Now given a $b$-map $f:M\to N$, we can lift the differential and its dual to maps ${}^b f_{*}: {}^b TM \to {}^b TN $ and ${}^b f^{*}: {}^b T^*M \to {}^b T^*N $ as follows. We use the anchor map $e_b$ to identify ${}^b TM$ with $TM$ in the interior and take ${}^b f_{*}$ to be standard differential. Next, we extend it by continuity to the boundary. As~\cite[Lemma 2.3.1]{melrose_book} shows, such extension is well-defined even for non-integer lifting matrices (see formula (2.3.5) of the same source~\cite{melrose_book}). Thus we can alternatively define the $b$-fibration in a more classical way as a $b$-map with a surjective $b$-differential. This definition is equivalent to ours as~\cite[Proposition 2.4.2]{melrose_book} shows.   
    
\end{remark}

\begin{definition}
\label{def:extended_union}
Given two index sets $E$ and $F$, their \textit{extended union} is the index set defined by
$$
E \overline{\cup} F = E \cup F \cup \{(s,p+q+1):(s,p)\in E, (s,q)\in F\}.
$$
\end{definition}
 
Assume again that we are given two manifolds with corners $M,N$. Let $f: M \to N$ be a $b$-fibration. If $\cE$ is a collection of index sets of $M$, define the index set $\cF = f_{\#}(\cE)$ as follows:
$$
F_i = \overline{\bigcup}_{\cB_j(X_1)\subset f^{-1}\cB_j(X_2)} \left\{\left(\frac{s}{e(i,j)},p \right):(s,p)\in E_j \right\}.
$$

\begin{thm}[\cite{melrose_book}]
\label{prop:push}
Let $f: M \to N$ be a $b$-fibration and let $\cE$ be a collection of index sets $\cE$ on $M$ such that $\Re(E_j)>0$ whenever $f(\cB_j)\nsubseteq \p B$. If $u\in \cA^\cE(N)$, then $f_* u \in \cA^{f_\# (\cE)}(M)$.
\end{thm}

A typical example of a $b$-map is given by a blow-down map.
\begin{lem}
Let $P \subset M$ be a $p$-manifold with blow-up data $(\cF,\alpha)$. Let $B'_i$, $i=1,\dots, q$ be boundary hypersurfaces of $M$ and $r'_i$ be their boundary defining functions. 

The blow-down map is $b$-map, i.e. we have 
$$
b^* r'_i = f\prod_{j=0}^q r_j^{e(i,j)},
$$
where $r_j$ are defining function of boundary hypersurfaces $\cB_j$ of $[M;P]_{\Lambda}$ and $f\neq 0$ is a smooth function on $[M;P]_{(\cF,\Lambda)}$.  
\end{lem}
The exact values of $e(i,j)$ are given in Lemma~\ref{lemm:blow-up_matrix} in Section~\ref{sec:proofs}.

A non-example of a b-fibration (which is still a b-map) is the polar coordinates map,
\[ \beta: \R_+\times [0,\pi/2]\to \R_+^2, \quad (r,\theta)\mapsto (r\cos\theta,r\sin\theta)  \]
as it maps the front face $\{r=0\}$ to a corner. More generally any non-trivial blow-down map will fail to be a b-fibration for this same reason. 

On the other hand, the map $\beta_L:=\pi_L\circ \beta:M_\alpha^2\to M$ will be a b-fibration, as there are no corners in $M$ to map into, and this map fibers over the interiors of each boundary hypersurfaces of $M_\alpha^2$ by construction.

\subsection{$\alpha$-calculus and existence of a parametrix}

\label{subsec:parametrix_statement}

Recall that from the construction of the double space, we have the blow-down map $\beta^2_\alpha: M^2_\alpha \to M^2$. This allows us to lift the diagonal from $M^2$ to the double space as
$$
\Delta_\alpha = \overline{(\beta_\alpha^2)^{-1}(\Delta_{int})},
$$
where $\Delta_{int}$ is the restriction of the diagonal to the interior of $M^2$. Similarly as in the case of the zero-calculus we associate to $M$ a density bundle
$$
\Omega_\alpha(M) = (\overline{r})^{-1-n-\alpha n} \Omega(M),
$$ 
where $n = \dim M - 1$ and $\overline{r}$ is the product of defining functions of all boundaries of $M$. This multiplicative factor is determined by the determinant of the Jacobian of the blow-down map $\beta_\alpha^2$, and is computed in Lemma 5.3. It is also convenient to introduce a different density bundle
$$
\Omega_b(M) = \overline{r}^{-1} \Omega(M),
$$ 
which will be used to understand the mapping properties of operators from the $\alpha$-calculus. From here we can define the half-density bundles $\Omega_\alpha^{1/2}$ and $\Omega_b^{1/2}$, which are useful when studying Schwartz kernels of linear operators. Recall that we can multiply two half-densities which gives a density. Moreover, half-density bundles are one-dimensional and trivializable. So we can choose a reference non-vanishing section of this bundle.

For example, in our case close to the singular set $\cZ$ we have local coordinates $x,y\in \R^n$. We can take a reference section of the half-density bundle $\Omega^{1/2}(M)$ in these coordinates to be
$$
\sqrt{\gamma} = \sqrt{dxdy}.
$$
Any other section would be given by $f\sqrt{\gamma}$, where $f$ is a smooth function. Thus, if we have two sections $f\sqrt{\gamma}$ and $g\sqrt{\gamma}$, their product will be a section of the density bundle $(f g)\gamma \in \Omega(M)$, which can be integrated over $M$. In particular, we can define the space square-integrable sections of $\Omega^{1/2}$ as
$$
L^2(M,\Omega^{1/2}) = \left\{\sigma \in \Gamma(\Omega(M))\,:\,\int_M \sigma^2 < \infty \right\}.
$$
If we take a reference half-density $\sqrt{\gamma}$ then a section $\sigma=f\sqrt{\gamma}$ will be square integrable if and only if $f$ is square integrable with respect to $\gamma$. Thus we have a natural isomorphism
\begin{equation}
\label{eq:density_iso}
L^2(M,\Omega^{1/2})\simeq L^2(M,\gamma).
\end{equation}
We work with half-densities to prove all the main results in the calculus, but in its concrete applications, it is much more convenient to fix a trivializing section as we will see later.

We begin by recalling the characterization of classical pseudodifferential operators to see how they are generalized in the $\alpha$-calculus. Recall that the classical pseudo-differential operators $\Psi^s(M)$ on smooth compact manifolds without boundary we have a principal symbol map 
\[ \sigma: \Psi^r(M)\to C^\infty(T^*M), \quad \text{defined by }  \sigma(L)(x,\xi)= \lim_{t\to \infty} t^{-r} e^{-itf} L e^{itf}, \;\; df_x=\xi, \]
which provides an invariant definition of this algebra homomorphism. The homomorphism takes the highest order part of the operator and produces a function which is degree $r$-homogeneous in the fibers of $T^*M$. Choosing coordinates, it admits a  particularly intuitive description for differential operators, namely if we have a differential operator with smooth coefficients  
$$
L = \sum_{|I| \leq s} a_I(x) \p_x^I,
$$
where $I$ is a multi-index, then its principal symbol is given by
$$
\sigma(L)(x,\xi)= \sum_{|I| = r} a_I(x) (i\xi)^I.
$$
The symbol map is a part of a short exact sequence
\begin{equation}
\label{eq:small_exact}
  \begin{tikzcd}
0 \arrow[r] & \Psi^{r-1}(M) \arrow[r] & \Psi^{r}(M) \arrow[r, "\sigma"] & C^\infty(T^*M) \arrow[r] & 0
\end{tikzcd}  
\end{equation}
which is the central property for an iterative parametrix construction. It allows us to construct an inverse modulo $\Psi^{-\infty}(M)$, whenever the principal symbol $\sigma(L)$ is an invertible operator. It is well-known that in this situation that $\Psi^{-\infty}(M)$ consists of compact operators~\cite{treves}. 

In the case of $\alpha$-calculus we are interested in inverting differential operators of the form
$$
L=\sum_{j + |\beta| \leq m}a_{j,\beta}(x,y)(x\p_x)^j(x^{1+\alpha}\p_y)^\beta,
$$
which is an element of $\Psi^m_\alpha(M)$ for $m\in \N$. Its principal symbol can be defined as
$$
^\alpha \sigma_m(L) = \sum_{j + |\beta| = m}a_{j,K}(x,y) (i\xi)^j (i\eta)^\beta.
$$
Principal symbols in this calculus are also smooth functions on the $\alpha$-cotangent bundle $^\alpha T^* M$, (see \eqref{alpha-tgtbdl}), and we get an analogue of the small exact sequence~\eqref{eq:small_exact}:
\begin{center}
  \begin{tikzcd}
0 \arrow[r] & \Psi^{s-1}_\alpha (M) \arrow[r] & \Psi^{s}_\alpha (M) \arrow[r, "^\alpha \sigma"] & C^\infty(^\alpha T^*M) \arrow[r] & 0 ,
\end{tikzcd}  
\end{center}
and correspondingly we say an operator is $\alpha$-elliptic if its $\alpha$-principal symbol is invertible on ${}^\alpha T^*M\setminus \{0\}$. Classical pseudodifferential constructions, depending only on the $\alpha$-principal symbol, can now produce an inverse modulo $\Psi^{-\infty}_\alpha(M)$, the smoothing operators. However, in contrast to the classical case elements of $\Psi_\alpha^{-\infty}(M)$ increases regularity but is not a compact operator~\cite{melrose_book}. In order to obtain an inverse modulo a compact operator, one needs to continue improving the parametrix using the large-calculus $\Psi^{s,\cE}_\alpha(M)$.

Now we are ready to give a definition of the operators in the $\alpha$-calculus.

\begin{definition}
The large calculus is defined as comprised of linear operators whose Schwartz kernels satisfy precisely
$$
A \in \Psi^{s,\cE}_\alpha(M) \iff \cK_A \in \cA_{phg}^\cE I^s(M^2_\alpha,\Delta_\alpha; \Omega_\alpha^{1/2}),
$$
for any index set $\cE=\{E_{10},E_{01},E_{11}\}$. This is the space of all linear operators with distributional kernels with conormal singularities at the lifted diagonal, which have polyhomogeneous expansions with index set $\cE$. The small $\alpha$-calculus is a special case defined as
$$
A \in \Psi^s_\alpha(M) \iff \cK_A \in \cA_{phg}^{(\infty,\infty, \N_0)}I^s(M^2_\alpha,\Delta_\alpha; \Omega_\alpha^{1/2}).
$$
i.e. its Schwartz kernel vanishes to infinite order at the side faces and has a `Taylor expansion' at the front face.
\end{definition} 

One example is the identity operator, whose Schwartz kernel is a distributional section of $\Omega^{1/2}(M^2)$, expressible in coordinates near $\p M_{\text{diag}}$ as,
\[ \cK_{\id}(x,y,\widetilde{x},\widetilde{y}) = \delta(x-\widetilde{x})\delta(y-\widetilde{y}) \mu \]
where $\mu=\sqrt{dxdy d\widetilde{x} d\widetilde{y}}\in \Omega^{1/2}(M^2)$. Now the half-density $\mu$ pulls back along $\beta_\alpha^{(2)}:M_\alpha^2\to M^2$ to the half-density $r^{1+(1+\alpha)n}\nu$ (a direct consequence of Lemma~\ref{lemm:densities}), where $\nu$ is a smooth half-density on $\Omega^{1/2}(M_\alpha^2)$, i.e. $\nu=\sqrt{drd\theta d\widetilde{y}}$. Thus our operator lifts to one of the form, 
\[ (\beta_\alpha^{(2)})^*\cK_{\id} = \delta(\theta_x-\theta_{\widetilde{x}})\delta(\theta_y) \,r^{-1-(1+\alpha)n}\nu = \delta(\theta_x-\theta_{\widetilde{x}})\delta(\theta_y) \widetilde{\nu}  \]
in the coordinates \eqref{polar-coords}, and $\widetilde{\nu}\in \Omega_\alpha^{1/2}$.

In order to construct a good parametrix, we need to take into account the behavior at the front face. So we consider the restriction of the lift of $L$ to the front face of $M_\alpha^2$. Locally we have projective coordinates
\begin{equation}\label{proj-coords}
(s,u,\tilde{x},\tilde{y}): = \left(\frac{x}{\tilde{x}},\frac{y-\tilde y}{\tilde x^{1+\alpha}},\tilde{x},\tilde{y} \right),
\end{equation}
where $\tilde{x} = 0$ defines the front face. This way we obtain a coordinate expression of restriction of the lift of $L$ to the front face:
\begin{equation}
N(L) = \sum_{j + |\beta| \leq m}a_{j,\beta}(0,\tilde y)(s\p_s)^j(s^{1+\alpha}\p_u)^\beta.
\end{equation}
For each $\tilde y\in \p M$ fixed, this operator is called the \textit{normal operator} of $L$. In Section~\ref{sec:proofs} it is proven that $L$ is invertible modulo a compact operator if it is elliptic in the interior and all $N(L)$ are invertible between suitable Hilbert spaces. Note that each $N(L)$ can be reduced to just a 1D operator by taking partial Fourier transform in the $u$ variable ($\p_u\mapsto -i\eta$). After normalizing this transformed variable $\hat{\eta} = \eta/|\eta|$, we get a family of operators parameterized by $(y,\hat{\eta})\in S^*(\p M)$ via 
$$
\widehat{N(L)}_{(y,\hat{\eta})} = \sum_{j + |\beta| \leq m}a_{j,\beta}(0,y)(s\p_s)^j(is^{1+\alpha}\hat{\eta})^\beta.
$$
In Theorem~\ref{thm:simpleparametrix} we prove that the invertibility of those operators is mostly determined by the behavior of these operators close to $s=0$. More precisely we can define the indicial operators as mentioned in the introduction, by rescaling $s\mapsto \lambda s$ and writing the invariant term, which will be given by
$$
I(L;y) = \sum_{j \leq m}a_{j,0}(0,y)(s\p_s)^j.
$$
This is family of scale invariant differential operators on $\R_{\geq 0}$ and as such can be reduced to a polynomial via the Mellin transform:
\begin{definition}
The Mellin transform of a function $f(s)$ is an integral transform of the form
$$
\cM[f](\lambda)=\int_0^\infty s^{-\lambda} f(s)\frac{ds}{s}.
$$
The inverse Mellin transform is given by 
\begin{equation}
\label{eq:inv_Mellin}
\cM^{-1}[F](x) = \frac{1}{2\pi i}\int_{\Re(\lambda) = c}x^{\lambda}F(\lambda)d\lambda.
\end{equation}
where we integrate over a horizontal line.
\end{definition}
We use a slightly non-standard definition of Mellin transform different from~\cite{mazzeo}, but which is closer to one that can be found, for example, in~\cite{mellin}. It has the following mapping property.  

\begin{thm}
\label{thm:mellin_reg}
Mellin transform $\cM$ is an isomorphism between 
\[ \text{ $x^\delta L^2([0,\infty])$ and $L^2(\left\{\Re(z)=\delta - \tfrac{1}{2}\right\})$. } \]
\end{thm}

Straight from the definition, we get the following property 
$$
\cM[s\p_s u] = \lambda \cM[u].
$$
Thus if we apply the Mellin transform to the indicial operators, we get a family of indicial polynomials, which in turn should be invertible away from their roots. 

This motivates the following definition:
\begin{definition}\label{def:indicialroots}
The boundary spectrum, or \emph{indicial roots}, at $y\in \p M$ of an operator
$$
L=\sum_{j + |\beta| \leq m}a_{j,\beta}(x,y)(x\p_x)^j(x^{1+\alpha}\p_y)^\beta,
$$
is the set 
$$
\spec_b(L;y) =  \left\{ \lambda\in \C \, \biggr| \, \lambda \text{ is a root of the polynomial } I_\zeta(L;y)  = \sum_{j\leq m} a_{j,0}(0,y)\zeta^j   \right\} . 
$$ 
If these roots are independent of the choice of basepoint $y\in \p M$ then we say that $L$ has \emph{constant indicial roots}.

\end{definition}

In general, we shall make hypotheses of the type that we can find a line $\{\Re{\zeta}=c\}$ which does not meet any point in $\text{Spec}_b(L;y)$. Because the inverse of $I(L)_\zeta$ is automatically holomorphic in $\zeta\in \C\setminus \text{Spec}_b(L;y)$ for a fixed $y$, we reduce questions of invertibility of $I(L)$ to finding such lines avoiding the roots as $y\in \p M$ varies. Appealing to the Mellin Inversion Theorem the corresponding value of $c$ defining such a line $\{\Re{\zeta}=c\}$ will in turn define the weight parameter for our choices of Hilbert spaces on which we obtain (semi-)Fredholmness.

Before stating the main parametrix theorem, let us also discuss the mapping properties of elements in the small and large calculi. The first important result is that operators from the large calculus preserve polyhomogeneous conormality. 

\begin{prop}
If $A\in \Psi^{s,\cE}(M)$, then it maps $\cA^*_{phg}(M)$ to itself.
\end{prop}
A more refined version of this result is given in Theorem~\ref{thm:param} in Section~\ref{sec:proofs}. In particular, if we are able to construct a parametrix in the large calculus, as the result it will map functions with polyhomogeneous expansions to functions with polyhomogeneous expansions. In this case $\ker L \subset \cA_{phg}(M)$ -- a fact which lies in the basis of the proof Theorem~\ref{thm:phg_Grushin}, which allows us to characterize the adjoints of elliptic $\alpha$-operators. 

Another important ingredient is the adapted Sobolev spaces.
\begin{definition}\label{alpha-sobolev}
The Sobolev spaces $H^s_\alpha(M,\Omega_\alpha^{1/2})$ are defined as
$$
H^s_\alpha(M,\Omega_\alpha^{1/2}) = \left\{u\in L^2 (M,\Omega^{1/2}_\alpha): L \in \Psi^s_\alpha(M) \Rightarrow Lu \in L^2(M;\Omega_\alpha^{1/2})\right\}.
$$ 
Similarly, we define the weighted Sobolev spaces
$$
x^a H^s_\alpha(M,\Omega_\alpha^{1/2}) = \left\{ x^a v : v  \in H^s_\alpha(M,\Omega_\alpha^{1/2})  \right\},
$$
where $x$ is a defining function for the boundary $\p M$.
\end{definition}

\begin{thm}
Given $a,a' \in \C$ an operator $A\in \Psi^{s,\cE}_\alpha(M)$ extends to a bounded operator:
$$
A: x^a H^t_\alpha (M,\Omega^{1/2}_\alpha)\to x^{a'} H^{t'}_\alpha (M,\Omega^{1/2}_\alpha) 
$$
if $t' \leq t-s$, $\Re (E_{01} + a) > n\alpha $, $\Re (E_{10} - a') > n\alpha $, $\Re (E_{11} - a' + a) > 0$.

Moreover, if $t' < t-s$ and $\Re E_{11} > a' - a$, then this operator is also compact.
\end{thm} 

In order to construct a parametrix the following assumption is usually made
\begin{ass}
The indicial polynomial does not depend on the boundary, i.e. the indicial roots are constant.
\end{ass}

Finally we come back to the Fourier transforms $\widehat{N(L)}$. A change of variables transforms $\widehat{N(L)}$ into an operator of ``Bessel-type" (see the discussion in $\S$ \ref{subsec:param}) is known \cite{mazzeo-uniquecont} that the resulting 1D operators enjoy unique continuation properties, and hence they are: injective as operators from $s^\delta L^2(\R_+)$ to itself for $\delta$ sufficiently big and surjective on the same space for $\delta$ sufficiently small. Thus to each $(y,\hat\eta) \in S^*(\p M)$ we can associate the thresholds $\overline{\delta}(y,\hat\eta)$ and $\underline{\delta}(y,\hat\eta)$
\begin{align*}
\overline{\delta} &= \max_{(y,\hat\eta) \in S^*(\p M)}\overline{\delta}(y,\hat\eta),\\
\underline{\delta} &= \min_{(y,\hat\eta) \in S^*(\p M)}\underline{\delta}(y,\hat\eta).
\end{align*}

We are now ready to state our main consequence of the $\alpha$-calculus, namely the existence of a parametrix for $\alpha$-elliptic differential operators.

\begin{thm}
\label{thm:simpleparametrix}
Suppose that $L \in \emph{Diff}_\alpha^m(M)$ is an elliptic operator with constant indicial roots $\{\lambda_i\}_{i=1}^m$ and define $\overline \delta, \underline \delta \in \R$ for $L$ as above.  

Fix $\delta$ such that the indicial roots do not lie on the lines
$$
\left\{z\in \C : \Re(z) = \delta - \frac{1}{2} \right\}.
$$
If $\delta > \overline{\delta}$, then 
$$
L: x^\delta H^{r+m}_\alpha(M) \to x^\delta H^{r}_\alpha(M)
$$ 
is left semi-Fredholm. Similarly if $\delta < \underline{\delta}$, then $L$ is right semi-Fredholm. Further, in either case, $L$ has closed range. 
\end{thm}
This result is in fact a consequence of a more extensive Theorem \ref{thm:param} which constructs a generalized inverse for $L$ and gives precise asymptotics for its Schwartz kernel near the singular submanifold $\cZ$.

\section{Natural domains for the Laplacian on $\alpha$-Grushin manifolds}

\label{sec:natural_domains}

Our goal is to use the results of the previous section to find the domain of the closure and of the adjoint of the operators $\Delta - cS$. All of the essential information needed for studying self-adjointness and self-adjoint extensions is encoded in the description of the mentioned domains close to the singular set. Indeed, in an open neighborhood around $q\in M\setminus\cZ$ the operator $\Delta - cS$ is elliptic and smooth and hence morally both its closure and adjoint locally should be standard Sobolev space. The difficulty is understanding the behaviour of elements of the kernel close to the singular set. This is done using a parametrix construction from Theorem~\ref{thm:param} that is used in its simplified version given by Theorem~\ref{thm:simpleparametrix}. Note that in order to prove all of the results of this section we can use this theorem as a blackbox.

Let us explicitly write down operators $\Delta - c S$. In the tubular neighborhood of $\cZ$ we introduce a companion metric 
$$
\tilde{g} = dx^2 + g_{(x,\cZ)}, 
$$
which simply removes the singularity. The metric is defined only locally, and we will use it to study local behavior of $P_{\alpha,c}$. Let $\tilde \omega$ be the associated volume form on $M$ and let $\Delta_{x,\cZ}$ be the Laplace-Beltrami operator on the level sets $x=\const$ with respect to $g_{(x,\cZ)}$.

If in the tubular neighborhood of $\cZ$ we have chosen a basis of orthonormal vector fields $X_0,\dots, X_{n}$ and $\omega$ the Riemannian volume associated to the original metric $g$. We can write the Laplace-Beltrami operator $\Delta$ locally as
$$
\Delta= \sum_{i=0}^{n} X_i^2 + \dive_\omega (X_i) X_i.
$$
Let us choose $X_{0} = \p_x$ and the remaining $X_i=x^\alpha Y_i$, where $Y_i$ is a family of vector fields tangent to the level sets of the distance function $x$. Note that this implies that $Y_i$ applied to the function $x$ is zero.

To calculate these divergence terms we notice that by given definitions the volume form is
$$
\omega = \frac{1}{|x|^{\alpha n}}\tilde{\omega}= \frac{1}{|x|^{\alpha n}}\omega_{(x,S)}\wedge dx,
$$
where $\tilde{\omega}$ is the volume form of the companion metric $\tilde g$ and $\omega_{(x,S)}$ are volume forms on the level sets $x=\const$, which can be identified with a smooth family of $n$-forms on $S$. We have by definition of divergence
$$
\dive_\omega (\p_x) \omega= \cL_{\p_x}(|x|^{-\alpha n} \tilde \omega) = -\alpha n |x|^{-\alpha n-1} \sign(x)\tilde \omega + \frac{\cL_{\p_x}\tilde\omega}{|x|^{\alpha n}}= \left(\dive_{\tilde{\omega}} (\p_x) - \frac{\alpha n}{x} \right) \omega.
$$ 
For us the explicit form of $\dive_{\tilde{\omega}} (\p_x)$ is not important. We will only need to know that it is smooth. For the remaining fields we find using Cartan's magic formula
\begin{align*}
\dive_\omega (x^\alpha Y_j) \omega &= \cL_{x^\alpha Y_j}(|x|^{-\alpha n} \omega_{(x,S)}\wedge dx)= |x|^{-\alpha n} d(|x|^\alpha (Y_j \lrcorner \; \omega_{(x,S)})\wedge dx) = \\
&= |x|^\alpha \dive_{\omega_{(x,S)}}(Y_j) \omega.
\end{align*}

Combining everything we get 
$$
\Delta= \p_x^2 + \left(\dive_{\tilde{\omega}} (\p_x) - \frac{\alpha n}{x} \right) \p_x+ |x|^{2\alpha}\sum_{i=1}^n\left(Y_i^2 + \dive_{\omega_{(x,\cZ)}} (Y_i) Y_i \right) 
$$
or in a compact form
\begin{equation}
    \label{eq:local_expression_laplace}
    \Delta= 
\p_x^2 +|x|^{2\alpha} \Delta_{x,\cZ} + \left(\dive_{\tilde{\omega}} (\p_x) - \frac{\alpha n}{x} \right) \p_x,
\end{equation}
where $\Delta_{x,\cZ}$ are the Laplacians on the level sets of the distance function $x$.

From now on we focus only on the behavior from right and remove the modulus sign in $|x|$. Left side is handled in the same manner.

\subsection{Closure of $\Delta - cS$}

\label{subsec:closure}

We can find the closure of $\Delta - cS$ using the following result proved in~\cite[Proposition 4.6]{mendoza}
\begin{prop}
\label{prop:mendoza}
Let $H_1, H_2$ be Hilbert spaces, and let $D$ be a Banach subspace of $H_1$ equipped with a norm $\|\cdot\|_D$ such that the inclusion map $(D,\|\cdot\|_D)\to (H_1,\|\cdot\|_{H_1})$ is continuous in $H_1$. Let $A:D \to H_2$ be a continuous operator and assume that:
\begin{enumerate}
\item The range of $A$ is closed.
\item The kernel $\ker A \subset H_1$ is closed with respect to $\|\cdot\|_{H_1}$.
\end{enumerate}
Then the operator $A$ with domain $D$ is closed, i.e., $D$ is complete with respect to the graph norm $\|u\|_{A} = \|u\|_{H_1} + \|Au\|_{H_2}$.
\end{prop}
In particular, we will take $H_1=H_2=L^2(M,\omega)$ and $D$ to be a certain weighted Sobolev space, so that the operator $\Delta - cS: D \subset L^2(M,\omega) \to L^2(M,\omega)$ will be left semi-Fredholm. In this case the conditions of the Proposition~\ref{prop:mendoza} are satisfied.

We can rewrite $\Delta - cS$ as a weighted operator $Q:=x^2(\Delta - cS)$:
\begin{equation}
\label{eq:general_op}
\Delta - cS = x^{-2} Q: = \frac{1}{x^2}\left(x^2 \p_x^2 + x^{2(1+\alpha)}\Delta_{x,\cZ} + \left(x \dive_{\tilde{\omega}} (\p_x) - \alpha n \right)x\p_x  - cx^2 S\right).
\end{equation}
If we look at the closure of $\Delta - cS$ as an operator from $L^2(M,\omega)$ to itself, we can equivalently look at the closure of
$$
Q: L^2(M,\omega) \to x^2 L^2(M,\omega). 
$$
We have 
$$
\omega = x^{1+n} \omega_\alpha
$$
where $\omega_\alpha \in \Omega_\alpha$ is a compatible density. Consequently
$$
L^2(M,\omega ) = L^2(M,x^{1+n}\omega_\alpha ) \simeq x^{-\frac{1+n}{2}}L^2(M,\omega_\alpha).
$$

We see that $Q\in \text{Diff}^2_\alpha(M)$ and thus continuous as an operator 
$$
Q:x^{2-\frac{1+n}{2}}H^2_\alpha(M,\omega_\alpha)  \to x^{2-\frac{1+n}{2}}L^2(M,\omega_\alpha). 
$$
If $Q$ is also left semi-Fredholm, then the operator $\Delta - cS = x^{-2}Q$ will be also left semi-Fredholm on the same space. Note that by definition we have that $x^{2-\frac{1+n}{2}}H^2_\alpha(M,\omega_\alpha)\simeq x^2 H_\alpha^2(M,\omega)$
 is continuously embedded into $x^{-\frac{1+n}{2}}L^2(M,\omega_\alpha)\simeq L^2(M,\omega)$ and contains $C^\infty_c(M)$ functions as a dense set. Hence proving that $Q$ is left semi-Fredholm would imply that the closure of $\Delta - cS$ on the core $C^\infty_c(M)$ is given by $x^{2-\frac{1+n}{2}}H^2_\alpha(M,\omega_\alpha)$. The rest of this subsection is dedicated to the proof of this fact (this is the content in part of Theorem \ref{thm:closure}).

In order to apply Theorem~\ref{thm:simpleparametrix} we need to check that the operators $\widehat{N(Q)}$ are injective on $x^{\delta}L^2(\R_{\geq 0},dx)$ with
$$
\delta = 2- \frac{1+n}{2}+\frac{1+\alpha n + n}{2} = 2 + \frac{\alpha n}{2}.
$$
We can write down $\widehat{N(Q)}$ explicitly using asymptotics of the scalar curvature~\eqref{eq:scalar_ass}. They are given by
\begin{equation}
\label{eq:fourier_transformed_normal}
\widehat{N(Q)} = x^2 \p_x^2 -\alpha n\, x\p_x + c\alpha  n(\alpha n + \alpha + 2) + \sigma_\cZ(y,\hat{\eta})x^{2(1+\alpha)},
\end{equation}
where $\sigma_\cZ<0$ is the symbol map of $\Delta_{0,\cZ}$.

Before studying this operator let us consider a more general operator of the form
\begin{equation}
\label{eq:not_bessel}
T := x^2 \p_x^2  + ax\p_x  + b-h x^{2\beta} ,
\end{equation}
with $\beta>0$, $h>0$ and $a,b\in \R$ constants. Define the following quantities
$$
\mu = (a-1)^2-4b
$$
and
$$
\nu =\frac{\sqrt{|\mu|}}{2\beta}.
$$
We will need the following lemma. 
\begin{lem}
Let $I(x,p)$ and $K(x,p)$ be modified Bessel functions of  order $p \geq 0$. For imaginary orders define 
$$
\tilde I(x,p)=\Re I(x,ip), \qquad \tilde K(x,p)=K(x,ip). 
$$
Then equation
\begin{equation}
Tu = 0,
\end{equation}
has two independent solutions of the form
\begin{equation}
\label{eq:kernel_solutions}
u_1 = x^{\frac{1-a}{2}} I\left(\frac{\sqrt{h}x^\beta}{\beta},\nu\right), \qquad u_2 = x^{\frac{1-a}{2}} K\left(\frac{\sqrt{h}x^\beta}{\beta},\nu\right),
\end{equation}
when $\mu \geq 0$. When $\mu<0$, replace $I,K$ in the formulas with $\tilde{I},\tilde{K}$.
\end{lem}

\begin{proof}
The proof is just a change of variables of the form
$$
y = \frac{\sqrt{h}}{\beta}x^\beta, \qquad v(y) = y^{(a-1)/2\beta} u(y),
$$
which reduces~\eqref{eq:not_bessel} to a modified Bessel equation.
\end{proof}

We need to understand when $T$ is injective as an operator from $x^\delta L^2(\R_{\geq 0},dx)$ to itself for some real parameter $\delta$. We can reduce this question to $\delta = 0$ by considering the conjugation
$$
T \mapsto x^{-\delta }T x^\delta.
$$
In this case $T$ is transformed to an operator of the same form where one has to replace
\begin{align*}
a &\mapsto 2\delta  + a,\\
b &\mapsto \delta^2 + \delta(a-1) + b.
\end{align*}
After this transformation the elements of the kernel are the same functions~\eqref{eq:kernel_solutions} multiplied by $x^{-\delta}$.

\begin{prop}
\label{prop:injectivity_of_T}
The operator ~\eqref{eq:not_bessel} mapping
$$
T: x^\delta L^2(\R_{\geq 0},dx) \to x^\delta L^2(\R_{\geq 0},dx)
$$
has a non-trivial kernel if and only if
$$
\Re\left(\frac{1-a}{2}-\delta -\frac{\sqrt{\mu}}{2}\right) > -\frac{1}{2}.
$$
\end{prop}

\begin{remark}
One can give a different interpretation of Proposition 3.2. Consider the operator $T$ and write down its indicial polynomial
\begin{equation}
\label{eq:indicial_pol}
p(\lambda) = \lambda^2 +(a-1)\lambda + b.
\end{equation}
Then we can reformulate Proposition~\ref{prop:injectivity_of_T} as follows. Operator $T$ is injective on $x^\delta L^2$ if and only if
$$
\Re \lambda_- - \delta \leq  - 1/2,
$$
where $\lambda_-$ is the root of the indicial polynomial with the smaller real part. %This was already noticed in~\cite{mazzeo}. 
Indeed, the exponential decay of the solutions at infinity imply that the injectivity questions are completely determined by the restriction of the operator to a neighborhood of zero. But in this case we can essentially use Mellin transform to solve approximately the equation $Tu=0$. 

\end{remark}

\begin{proof}
As already mentioned without any loss of generality we can assume that $\delta = 0$. We have the following well-known asymptotics for modified Bessel functions for $\nu \geq 0$~\cite{Bessel}: 
\begin{align*}
I(x,\nu) &\sim \frac{1}{\sqrt{2\pi}}x^{-1/2}e^{x},\\
K(x,\nu) &\sim \frac{1}{\sqrt{2\pi}}x^{-1/2}e^{-x}
\end{align*}
as $x\to +\infty$ and
\begin{align*}
I(x,\nu) \sim \begin{cases}
\dfrac{(x/2)^{|\nu|}}{|\nu|!}, & \nu\in \Z\setminus\{0\},\\
\dfrac{(x/2)^\nu}{\Gamma(\nu+1)}, & \nu\notin \Z\setminus\{0\},
\end{cases}\\
K(x,\nu) \sim \begin{cases}
\dfrac{\Gamma(\nu)}{2}\left(\dfrac{2}{x}\right)^\nu, & \nu > 0,\\
-\ln (x/2) + \gamma , & \nu = 0,
\end{cases}
\end{align*}
where $\gamma$ is the Euler constant and $x\to 0+$. For the imaginary order the asymptotics at infinity are the same, while the asymptotics as $x\to 0+$ are given by
\begin{align*}
\tilde{I}(x,\nu)&\sim \left(\frac{\sinh (\pi\nu)}{\pi\nu} \right)^{1/2} \cos\left(\nu \ln \left(\frac{x}{2}\right)-\gamma_\nu \right),\\
\tilde{K}(x,\nu)&\sim -\left(\frac{\pi\nu}{\sinh (\pi\nu)} \right)^{1/2} \sin\left(\nu \ln \left(\frac{x}{2}\right)-\gamma_\nu \right),
\end{align*}
where $\gamma_\nu$ is a real continuous function of the order $\nu$.

From the asymptotics it follows immediately that the $I$-solution of~\eqref{eq:not_bessel} is never in $L^2$ for both real and imaginary orders because of the exponential growth. So the kernel of this operator can be only a multiple of the $K$-solution. Since the $K$-solution decreases exponentially at infinity, we only need to understand when it is square integrable close to zero. From the given asymptotics it follows that this is the case, if
$$
\frac{1-a}{2}-\beta\sqrt{\frac{(a-1)^2}{4\beta^2}-\frac{b}{\beta^2}} > -\frac{1}{2},
$$ 
when the order of the modified Bessel functions is real and 
$$
\frac{1-a}{2}> -\frac{1}{2},
$$
when the order is imaginary.

\end{proof}

\begin{remark}
\label{rmrk:on_the_critical_value}
Note first that there is a unique critical value of $\delta$, such that for all $\delta$ bigger than this value the operator $T$ has no kernel in $x^\delta L^2$. Secondly, this value does not depend on the constant $h$ in~\eqref{eq:not_bessel}. 
\end{remark}

We can now analyse the operators $\widehat{N(Q)}$. First note that as discussed in Remark~\ref{rmrk:on_the_critical_value} $\widehat{N(Q)}_{(y,\hat{\eta})}$ will be injective on $x^\delta L^2$ for all $(y,\hat{\eta})$ simultaneously. As discussed, we can look at the root of the indicial polynomial~\eqref{eq:indicial_pol} with the smaller real part. Here
\begin{align*}
\delta &=  2 +\frac{\alpha n}{2},\\
a &= -\alpha n,\\
b &= c\alpha n(\alpha n + \alpha + 2).
\end{align*}
The two roots are given by the expression:
$$
\lambda_{\pm} = \frac{1}{2}\left(1 +\alpha n \pm \sqrt{(1+\alpha n )^2-4c\alpha n (2 + \alpha +\alpha n)}\right).
$$
The condition for injectivity is given by
$$
\Re (\lambda_-)-\delta \leq -\frac{1}{2},
$$
but this is always satisfied, because after simplifying the parameters above we always have
$$
\Re (\lambda_-) -\delta \leq -\frac{3}{2}.
$$ 
Thus we have proven
\begin{prop}
The operators $\widehat{N(Q)}_{(y,\hat{\eta})}$ are injective for all possible $\alpha>-1$, $n\in \N$, $c\in\R$ and $(y,\hat\eta)\in S(T^*\p M)$. 
\end{prop}

Thus by Theorem~\ref{thm:simpleparametrix}, we can construct a left parametrix as long as $\delta$ satisfies
$$
\Re(\lambda_+) -\delta + 1/2 \neq 0.
$$
This can be false for certain values of the parameter $c$. Indeed, $\Re(\lambda_+)$ is a continuous monotone non-increasing function of the parameter $c$. Thus for each $\alpha >-1$, $n\in \N$ fixed, there is a unique forbidden value of $c$ given by
$$
c_0 = \frac{-3+2\alpha n+n^2\alpha^2}{4 n \alpha(2+\alpha+\alpha n)},
$$
for which there is no left parametrix in or calculus. We can equivalently write this condition in terms of the discriminant $\mu$ of the indicial polynomial
$$
c=c_0 \iff \mu = 4.
$$
This is a phenomenon that was observed already in the case of the standard 2D Grushin manifolds, for which $\alpha=n=1$ and as a consequence $c=0$~\cite{grushin}.

To include this important class of generic almost-Riemannian structures, we would also like to say something about the closure of those operators for the forbidden values of the parameter $c$. To do this, we consider the closure of smooth functions $C^\infty_c(M\setminus \cZ)$ on a slightly bigger space than $L^2(M,\omega)$. The resulting space will contain the original closure as a subspace. So let us consider the closure of compactly supported smooth functions in $x^{-\varepsilon}L^2(M,\omega)$. This amounts to shifting all of the exponents by $\varepsilon$, which is enough to step outside the forbidden value and obtain a left-parametrix. The resulting closure will be $x^{2-\varepsilon-\frac{1+n}{2}}H^2_\alpha(M,\omega_\alpha)$. Similarly we can consider the closure in $x^{\varepsilon}L^2(M,\omega)$, which is given by $x^{2+\varepsilon-\frac{1+n}{2}}H^2_\alpha(M,\omega_\alpha)$. The inclusions
$$
x^{\varepsilon} L^2(M,\omega)\hookrightarrow L^2(M,\omega)\hookrightarrow x^{-\varepsilon} L^2(M,\omega)
$$
are continuous. Exactly the same discussion as before proves the following result.

\begin{thm}
\label{thm:closure}
Let $M$ be a $\alpha$-Grushin manifold of dimension $n$. If $\mu \neq 4$, then the closure of $\Delta-cS$ or the minimal domain is given by
$$
D_{\min}(\Delta-cS) = x^{2-\frac{1+n}{2}}H^2_\alpha(M,\omega_\alpha).
$$
If $\mu = 4$, then
$$
\bigcup_{\varepsilon>0}  x^{2+\varepsilon-\frac{1+n}{2}}H^2_\alpha(M,\omega_\alpha) \subset D_{\min}(\Delta-cS) \subset \bigcap_{\varepsilon>0}  x^{2-\varepsilon-\frac{1+n}{2}}H^2_\alpha(M,\omega_\alpha).
$$
\end{thm}

\subsection{Adjoint of $\Delta - cS$}

\label{subsec:adjoint}

The domain of the adjoint, or the maximal domain, of $\Delta - cS$ is given by
$$
D_{\max}(\Delta - cS) = \left\{u\in L^2(M,\omega)\,:\, (\Delta-cS)u\in L^2(M,\omega)\right\}.
$$
In order to find when $\Delta-cS$ is essentially self-adjoint and construct its self-adjoint extensions, we need a better description of this space modulo functions in the minimal domain.

%%%%%%%%%%%%%%%%%%%%%%%%%%%%%%%%%%%%%%%%%
%% MELLIN ASYMPTOTICS
%%%%%%%%%%%%%%%%%%%%%%%%%%%%%%%%%%%%%%%%%

Here again, we use our left parametrix. However, we need a slightly refined version given in Theorem~\ref{thm:param}, which compared to Theorem~\ref{thm:simpleparametrix}, states that one can invert $\alpha$-elliptic operators modulo not simply compact operators, but modulo operators which lie in the large calculus. Assume first $\mu \neq 4$. By the main parametrix Theorem~\ref{thm:param} we have a left parametrix $G$ and a very residual harmonic projector $\pi_{ker}$ such that
$$
Gx^2(\Delta - cS) = \id - \pi_{ker}.
$$ 
In particular, if $x^2(\Delta - cS)u = f$, then we obtain
$$
u = Gf + \pi_{ker} u.
$$ 
If $u\in D_{\max}(\Delta - cS)$, by definition $u \in L^2(M,\omega)$ and $f\in x^2 L^2(M,\omega)$, thus by construction $Gf$ will lie in the closure. We are thus reduced to studying $\pi_{ker}u$. This implies that
\begin{equation}
\label{eq:domain_max}
D_{\max}(\Delta - cS) = \left(\ker (\Delta - cS) \cap L^2(M,\omega)\right) + D_{\min} (\Delta - cS).
\end{equation}
A consequence of~\eqref{eq:domain_flag} is that $\Delta - cS$ is essentially self-adjoint if and only if
$$
\left(\ker (\Delta - cS) \cap L^2(M,\omega)\right) \subset D_{\min} (\Delta - cS).
$$

We will determine the precise form of elements in the kernel by considering them as limits of functions with polyhomogeneous expansions. Namely, by Theorem~\ref{thm:phg_Grushin} elements of $\ker (\Delta-cS)$ have a weak asymptotic expansion~\eqref{eq:polyhom_ass}, which we repeat here
\begin{equation}
\label{eq:polyhom_ass_sobolev}
u \sim \sum_{\lambda \in \emph{Spec}_b(L)} \sum_{\theta \in \Theta} \sum_{p=0}^{p_\lambda} x^{\lambda + \theta}(\log x)^p \, u_{\lambda,\theta,p}(y),  \qquad x\to 0+, \quad u_{\lambda,\theta,p}\in H_\alpha^{-r_{\lambda,\theta}}(\cZ)
\end{equation}
where $\lambda$ are roots of the indicial polynomial,  and the order of negative regularity is $r_{\lambda,\theta}=\Re(\lambda) +(i+j)-\delta+\tfrac12$ for the value $\theta=(1+\alpha)i+j\in \Theta$. Here we say that this is a weak asymptotic expansion in the sense that there is such an expansion of the integral $\int_\cZ u(x,y)\chi(y)dy$ for any test functions on $\cZ$. As mentioned in the introduction, we give the proof of this technical result in Section~\ref{sec:proofs}.

%\begin{prop}
%\label{prop:poly_ass_of_kernel}
%Let
%\begin{equation}
%\label{eq:theta_def}
%\Theta := \{(1+\alpha)i + j \in \R_{\geq 0} \,:\, i,j\in \N_0 \}
%\end{equation}
%and $L\in \emph{Diff}_\alpha^r(M)$ be $\alpha$-elliptic. Suppose that $u \in x^\delta L^2(M,\omega_\alpha)$ and $Lu =0 $. Then $u$ admits a weak asymptotic expansion of the form
%\begin{equation}
%\label{eq:polyhom_ass_sobolev}
%u \sim \sum_{\lambda \in \emph{Spec}_b(L)} \sum_{\theta \in \Theta} \sum_{p=0}^{p_\lambda} x^{\lambda + \theta}(\log x)^p \, u_{\lambda,\theta,p}(y),  \qquad x\to 0+, \quad u_{\lambda,\theta,p}\in H_\alpha^{-r_{\lambda,\theta}}(\cZ)
%\end{equation}
%where $\lambda$ are roots of the indicial polynomial,  and the order of negative regularity is $r_{\lambda,\theta}=\Re(\lambda) +(i+j)-\delta+\tfrac12$ for the value $\theta=(1+\alpha)i+j\in \Theta$. Here we say that this is a weak asymptotic expansion in the sense that there is an expansion of the integral $\int_\cZ u(x,y)\chi(y)dy$ for any test function.
%\end{prop}

%The proof is given in Section~\ref{sec:proofs}.

Let us first focus only on the polyhomogeneous elements of the kernel, i.e., functions for which the coefficients $u_{\lambda,\theta,p}(y)$ in the expansion \eqref{eq:polyhom_ass_sobolev} are actually smooth. Notice that in this case the remainder can be an element of the the minimal domain. It is clear that $x^{2-\frac{1+n}{2}}H^2_\alpha(M,\omega_\alpha)\subset x^{2-\frac{1+n}{2}}L^2_\alpha(M,\omega_\alpha)$. If $\mu\neq 4$, then Theorem~\ref{thm:closure} implies that any function $u\in D_{\min}(\Delta - cS)$ must have asymptotics 
\begin{equation}
\label{eq:min_asymptotics}
u(x) = o\left(x^\frac{3+\alpha n}{2}\right), \qquad x \to 0+. 
\end{equation}

Thus we need all of the expansions modulo these terms. On the other hand, if $u\in D_{\max}$, then $u\in L^2(M,\omega)$. Thus functions $u\in D_{\max}$ must have asymptotics satisfying
\begin{equation}
\label{eq:max_asymptotics}
u(x) = o\left(x^\frac{-1+\alpha n}{2}\right), \qquad x \to 0+.  
\end{equation}

Let us use asymptotic expansion~\eqref{eq:polyhom_ass_sobolev} as an ansatz assuming in addition that all of the coefficients are smooth functions. This is essentially the Frobenius method and we are reduced to distinguishing between two cases
\begin{enumerate}
\item The roots of the indicial polynomial are such that $\lambda_+ - \lambda_- \notin \Theta $ (see equation~\eqref{eq:theta_def});
\item The roots of the indicial polynomial are separated by an element of $\Theta$.
\end{enumerate}
In the first case, there are no log terms and we get a pure power series asymptotics. 

 To obtain the desired asymptotic expansions the first step is to rewrite operator $x^2(\Delta-cS)$ as a power series with powers in $\Theta$ using its indicial part and lower order terms. To do so, we expand formally all coefficients of $s^2\Delta$ in Taylor series and use Proposition~\ref{prop:curvature_ass} of the curvature term. Recall the explicit form of $x^2(\Delta -cS)$ given by~\eqref{eq:general_op}. We get:
\begin{align*}
\dive_{\tilde{\omega}}(\p_x)(x,y) &= \sum_{\theta \in \Theta}\dive_\theta(y)x^\theta,\\
S(x,y) &= \frac{S_{0}}{x^2} + \sum_{\theta\in \Theta\setminus\{0\}}S_{\theta}(y)x^{-2+\theta},\\
\Delta_{x,\cZ} &= \sum_{\theta \in \Theta} \Delta_{\theta,\cZ}.
\end{align*}
Note that $\Delta_{0,\cZ}$, as discussed previously, is the Laplace-Beltrami operator on the singular set. For brevity we will often omit the dependence of the coefficients on $y$ until the rest of this subsection. 
Now we rewrite $x^2(\Delta -cS)$ as
$$
x^2(\Delta -cS) = I(x^2(\Delta - cS)) +\sum_{\theta\in\Theta}x^{1+\theta}\dive_\theta x\p_x - \sum_{\theta\in \Theta\setminus\{0\}}c S_{\theta}x^{\theta} + \sum_{\theta\in\Theta} x^{2(1+\alpha)+\theta} \Delta_{\theta,\cZ}.
$$
Note that in the non-indicial part of the exrepssion all of the powers of $x$ belong to $\Theta\setminus \{0\}$.

We can now plug-in the asymptotics ansatz
\begin{equation}
\label{eq:series_nondeger}
u = \sum_{\varphi\in \Theta} a_{\varphi}(y)x^{\lambda+\varphi}
\end{equation}
into $x^2(\Delta -cS)u = 0$. Recall that
$$
I(x^2(\Delta - cS))x^\lambda = p(\lambda) x^\lambda,
$$
where $p$ is the indicial polynomial of $x^2(\Delta - cS)$. Thus 
\begin{align}
\sum_{\varphi \in \Theta}a_\varphi p(\lambda + \varphi)x^{\lambda+\varphi}+\sum_{\varphi \in \Theta}\sum_{\theta \in \Theta \setminus \{0\}} R_{c,\theta}(a_\varphi) x^{\lambda+\varphi+\theta} =0,\label{eq:frobenius_recursion}
\end{align}
where $R_{c,\theta}$ are differential operators on $\cZ$ obtained by combining all of the terms of the same order. We solve this equation term by term starting with the principal one:
$$
p(\lambda)a_{0}x^\lambda = 0.
$$
Thus $\lambda$ must be one of the two roots $\lambda_{\pm}$ of the indicial polynomial. Note that the first coefficient $a_{0}$ can be an arbitrary smooth function on $\cZ$. We then proceed inductively solving for all of the remaining coefficients. For each term we get an equation of the type
$$
p(\lambda+\varphi)a_\varphi + \text{ lower order terms } = 0,
$$
which can be solved for $a_\varphi$ under the assumption that $\lambda+\varphi$ is not an indicial root. This inductive procedure shows that if the first coefficient of the weak expansion is smooth, then all the remaining coefficients are also smooth and $u$ is, in fact, polyhomogeneous.

If we are in one of the critical cases and the two indicial roots are separated by an element of $\Theta$, then clearly we can solve recursively~\eqref{eq:frobenius_recursion} with $\lambda = \lambda_+$. This gives a family of solutions $u_+$. In order to find the second family of solutions $u_-$ we can use a trick from the classical Frobenius method, namely we take
$$
u_- = C u_+\log x + \sum_{\theta\in \Theta} a^-_{\theta}x^{\lambda_-+\theta}.
$$
Where $C \neq 0$ is a constant that must be determined. If we plug this ansatz into $x^2(\Delta-cS)u_- = 0$, we obtain a modification of~\eqref{eq:frobenius_recursion} and we proceed in a similar manner. And there are no other solutions of the form~\eqref{eq:polyhom_ass}.

\begin{remark}
The coefficient $a_{\varphi}^-$, which could not be determined from the recurrence relations can be taken to be any smooth function. It is accustomed to take them to be zero, since a non-zero choice would correspond to a different choice of of the principal coefficient $a_{0}^+$. If $\lambda_+ = \lambda_-$, then $C\neq 0$ can be arbitrary and it is common to take $C=1$ and choose $a_{0}^-$ to be an arbitrary function instead.
\end{remark}

The following characterization of the maximal domain holds.
\begin{thm}
\label{thm:max_domain_full}
The maximal domain of $\Delta - cS$ is given by
\begin{equation}
    \label{eq:max_domain_separation}
    D_{\max}(\Delta - cS) = \overline{(\spann\{u_+\} \oplus \spann\{u_-\})\cap L^2(M,\omega)}^{||\cdot||_{\emph{Gr}}} + D_{\min}(\Delta - cS), 
\end{equation}
where the span is taken over an arbitrary choice of the leading term coefficients as smooth functions on the singular set $\cZ$ and the closure is in graph norm of the operator $\Delta - cS$.
\end{thm}

\begin{proof}
Since the minimal domain by definition is closed in the graph norm, and all elements of $(\spann\{u_+\} \oplus \spann\{u_-\})\cap L^2$ belong to the maximal domain, it is clear that the space on the right hand side of~\eqref{eq:max_domain_separation} is a subspace of $D_{\max}(\Delta - cS)$. Hence, we only need to prove the other direction. 

For this, we will show that any element of $\cD_\text{max}(\Delta-cS)$ can be approximated in the graph norm by polyhomogeneous elements of $\cD_\text{max}(\Delta-cS)$. Consider any $v\in \cD_\text{max}(\Delta-cS)$ and let $\phi$ be local cut off function. Namely, assume that $\phi(x,y)=\phi_1(x)\phi_2(y)$ where $\phi_1\equiv 1$ in a neighborhood of $x=0$, and $\phi_2$ a smooth bump function supported in a coordinate chart $(x,y)$ centered at a point $(0,y_0)=p\in \cZ$; then $\phi v\in \cD_\text{max}(\Delta-cS)$ as well. Now consider for $\psi\in C^\infty(\cZ)$ the convolution
\[ v\ast \psi(x,y)=\int_\cZ v(x,y-y')\psi(y')dy' , \]
which has an asymptotic expansion as $x\to 0$ induced by the expansion of $v$, hence $v\ast \psi\in \cD_\text{max}(\Delta-cS)$. Further, by definition of the weak expansion of $v$, the coefficients of \eqref{eq:polyhom_ass_sobolev} are now smooth in $y$ thus $v\ast\psi\in \cD_\text{max}(\Delta-cS)\cap \cA_{phg}$.

Now choose $\phi\in C^\infty(B)$ to be a smooth bump function as above, compactly supported in a coordinate chart centered around $y_0\in \cZ$, satisfying $\int_\cZ \phi(y)dy=1$; in particular the Fourier transform of $\phi$ satisfies $\hat{\phi}(0)=1$. Thus the sequence $\phi_\varepsilon(y):=\varepsilon^{-n}\phi(y/\varepsilon)$ converges to 1 pointwise, and the convolution $v_\varepsilon=v\ast \phi_\varepsilon$ satisfies
\[ ||\hat{v}_\varepsilon-\hat{v}||_{L^2} \to 0 \text{ as }\varepsilon\to 0 \]
by dominated convergence. Further $(\Delta-cS)v_\varepsilon = (\Delta-cS)v\ast \phi_\varepsilon+[(\Delta-cS),\phi_\varepsilon\ast(-)]v$ is in $L^2$ (since $\psi_\varepsilon\ast(-)$ is constant in $x$, it vanishes upon applying $\partial_x$ derivatives which is where the only singular coefficients of $\Delta-cS$ arise), thus by the same argument as above we have $(\Delta-cS)v_\varepsilon\to (\Delta-cS)v$ in $L^2$. This completes the proof of the claim of graph norm convergence.

Now, having shown this graph norm approximation of $v\in \cD_{\text{max}}(\Delta-cS)$ by polyhomogeneous elements $v_\varepsilon$ of the maximal domain, by projecting such elements into $(\ker\cap L^2)$ we obtain a sequence $u_\varepsilon=\pi_{\ker\cap L^2}v_\varepsilon$ of \emph{polyhomogeneous} functions approximating $\pi_{\ker\cap L^2}v$ in the graph norm. Because an element $u\in \ker(\Delta-cS)\cap L^2$ is polyhomogeneous if and only if the coefficients in \eqref{eq:polyhom_ass_sobolev} are smooth, equivalently if the leading coefficients $u_\pm$ in this expansion are smooth (since the higher coefficients are formally determined by these leading two as $Lu=0$), we have shown the claim.

%it is enough to show that any element in $\ker (\Delta - cS)$, which by Theorem~\ref{thm:phg_Grushin} has a weak asymptotic expansion can be approximated by a polyhomogeneous element of $\ker (\Delta - cS)$. {\color{red}For this...}  

%claim we follow a standard argument. 
\end{proof}

Theorem~\ref{thm:self-adjoint_general} is now a direct consequence of the previous theorem.

\begin{proof}[Proof of Theorem \ref{thm:self-adjoint_general}]
Since $D_{\min}(\Delta - cS)$ is closed in the graph norm, inclusions~\eqref{eq:domain_flag} and Theorem~\ref{thm:max_domain_full} implies that we only need to determine when
$$
\left((\spann\{u_+\} \oplus \spann\{u_-\}) \cap L^2(M,\omega)\right) \subset D_{\min}(\Delta - cS).
$$
Note that since the families $u_{\pm}$ have polyhomogeneous expansions, if $u_{\pm} \in L^2$, then it will be automatically in $H^2_\alpha$ as well. Thus we only need to look at the principal order in the asymptotic expansions. In particular, $\Delta- cS$ will be essentially self-adjoint if and only if either $u_{\pm}$ do not satisfy~\eqref{eq:max_asymptotics} or, if they do, then additionally they satisfy~\eqref{eq:min_asymptotics}. 

The principal terms of $u_{\pm}$ are determined by $\lambda_{\pm}$. In our case the indicial polynomial is given 
$$
p(\lambda) = \lambda^2 - (1+\alpha n)\lambda + c \alpha n(\alpha n + \alpha + 2).  
$$
The discriminant is exactly $\mu$ from the statement and the two roots of $p(\lambda) = 0$ are given by
$$
\lambda_\pm = \frac{1+\alpha n \pm \sqrt{\mu}}{2}.
$$
Assume first that there are no logarithmic terms in the principal term of $u_{-}$, i.e., $\lambda_+ \neq \lambda_-$. Then $u_{\pm}\in D_{\max}(\Delta - cS)$ if and only if
$$
\Re(\lambda_{\pm}) > \frac{-1+\alpha n}{2}.
$$ 
If we write down this condition explicitly, then we see that
\begin{enumerate}
\item $u_+ $ always belongs to $D_{\max}(\Delta - cS)$;
\item $u_- $ belongs to $D_{\max}(\Delta - cS)$ if and only if
$$
\mu < 4. 
$$
\end{enumerate}
Similarly, $u_{\pm} \in D_{\min}(\Delta - cS)$ if and only if
$$
\Re(\lambda_{\pm}) > \frac{3+\alpha n}{2}.
$$
This is equivalent to
\begin{enumerate}
\item $u_+ $ belongs to $D_{\min}(\Delta - cS)$ if and only if
$$
\mu > 4. 
$$
\item $u_- \notin D_{\min}(\Delta - cS)$.
\end{enumerate}

Note that $\mu = 4$ corresponds to the case when $\lambda_+ - \lambda_- = 1$ and thus has logarithmic terms. We now see that if there are no logarithmic terms, then $\Delta - cS$ is essentially self-adjoint if and only if $\mu>4$.

If $\mu \neq 4$, the logarithmic terms are handled using order analysis in a similar manner. 
\end{proof}

\begin{remark}
A natural question is what happens when $\mu = 4$. In this case the closure is slightly bigger compared to the case $\mu \neq 4$ and one can expect that operator $\Delta - cS$ could be actually essentially self-adjoint. This is indeed the case for $\alpha = n =1$, which forces $c=0$ as proven in~\cite{grushin}. The proof uses heavily the positivity of $\Delta$ and a similar proof with the use of Kato-Relich theorem should work in the general case as well. Note that our description of the maximal domain in Theorem~\ref{thm:max_domain_full} still holds. The problem is that we can not obtain using our technique a precise enough description of $D_{\min}$.
\end{remark}

\begin{remark}
Let us see what happens when $c = 0$. We have 
$$
\mu = (1+\alpha n)^2
$$
and $\Delta$ is essentially self-adjoint if $\alpha \in \left (-1,\min\{-1,-3/n\}\right ) \cup (1/n,+\infty)$ and is not essentially self-adjoint if
$\alpha \in \left (\max\{-1,-3/n\},1/n\right )$. Compare this with Theorem~\ref{thm:sa_grushin_standard}, recalling that we only consider the case $\alpha > -1$. 
\end{remark}

\section{Self-adjoint extensions}
\label{sec:sa_extensions}

\subsection{Asymmetry forms method and the symplectic linear algebra of self-adjoint extensions}
\label{subsec:asym_forms}

Let us recall the asymmetry form method which can be used to construct self-adjoint extensions of an operator. Let $H$ be a Hilbert space and $A: D(A) \subset H \to H$ a densely defined symmetric operator. As discussed in the introduction, we can associate to $A$ two natural extensions: the closure $\overline{A}$ and the adjoint $A^*$ defined as usual with
$$
D(A) \subset D(\overline{A}) \subset D(A^*).
$$
Let $\cD_{\pm}$ be the two deficiency spaces of $A$, i.e.
$$
\cD_{\pm}=\ker (A^* {\pm} i I).
$$
Then a result due to Von Neumann \cite[Theorem 3.4]{sa_general} states that
$$
D(A^*) = D(\overline{A}) \oplus \cD_+ \oplus \cD_-
$$

The asymmetry forms of an operator $B$ are defined as:
$$
\omega_B(x,y) := ( x,By) - (Bx,y), \qquad \omega_B(x) := \omega_B(x,x), \qquad \forall x,y\in D(B).
$$
If $B = A$, then $\omega_A = 0$ thanks to the assumption of $A$ symmetric, so we instead consider $B=A^*$. In this case $\omega_{A^*}$ measures how far $A^*$ is far from being a symmetric operator.

\begin{lem}
If $x\in D(\overline{A})$, then $\omega_{A^*}(x,\cdot)=0$. 
\end{lem}
\begin{proof}We know that $\overline{A}=(A^*)^*$ and thus by definition if $z=\overline{A}x$, then
$$
(x,A^*y) - (z,y) = 0, \qquad \forall y \in D(A^*).
$$ 
But $D(\overline{A})\subset D(A^*)$ and hence $z = \overline{A}x= A^*x$ and
$$
\omega_{A^*}(x,y) = (x,A^* y) - (A^* x,y) = 0, \qquad \forall y \in D(A^*).
$$
\end{proof}
We can now see that the only obstruction for $A^*$ to be symmetric is given by the deficiency spaces, and all symmetric extensions are given by subspaces of $D(A^*)/D(\overline A) \simeq \cD_+ \oplus \cD_-$ on which $\omega_{A^*}$ vanishes. We can write explicitly the quadratic form $\omega_A$ using the definition of the deficiency spaces:
\begin{equation}
    \label{eq:asym_form_sympl}
    \omega_{A^*}(x_+ + x_-) = 2i(\|x_+\|^2 - \|x_-\|^2), \qquad x_\pm \in \cD_{\pm}.
\end{equation}
In particular, if one of the deficiency spaces is trivial and the other is not, there are no non-trivial symmetric extensions.

A self-adjoint extension neccesarily must be a maximal symmetric extension. However, this is not sufficient. The form $\omega_{A^*}$ must also be non-degenerate on $D(A^*)/D(\overline A)$~\cite{sa_general}. For example, if the deficiency indexes were different, then the form $\omega_{A^*}$ would have a kernel on $D(A^*)/D(\overline A)$. If the form $\omega_{A^*}$ is non-degenerate, then indeed, every self-adjoint extension is characterized by a \textit{Lagrangian space} inside $D(A^*)/D(\overline A)$, i.e., a maximal subspace on which the asymmetry form vanishes. This allows us to classify all of the self-adjoint extensions of real-differential operators. Namely, the following theorem holds.

\begin{thm}
    Let $H$ be a Hilbert space and $A: D(A) \subset H \to H$ a densely defined symmetric operator. If $\omega_{A^*}$ is non-degenerate on $D(A^*)/D(\overline A)$, then 
    \begin{enumerate}
        \item $(D(A^*)/D(\overline A),\omega_{A^*})$ is a symplectic space;
        \item there is a natural bijection between self-adjoint extensions of $A$ and Lagrangian subspaces of $(D(A^*)/D(\overline A),\omega_{A^*})$.
    \end{enumerate} 
\end{thm}

\subsection{Construction of self-adjoint extensions}

\label{subsec:sa_extensions}

We are now ready to construct the self-adjoint extensions of $\Delta - cS$ in a number of cases. We separate the discussion into two parts: when $\mu<0$ and when $0\leq \mu< 4$. Both cases are studied using an explicit characterization of the asymmetry form. 

Consider a tubular neighborhood around the singular set $\cZ=\{x=0\}$. From the proof of Theorem~\ref{thm:max_domain_full} it follows that when $\mu < 4$ the adjoint $D_{\max}(\Delta - cS)$ consists of functions
$$
u(x,y) = (u^r_+(x,y) + u^r_-(x,y) + u^l_+(x,y) + u^l_-(x,y))\chi(|x|) \mod D_{\min}(\Delta - cS), 
$$
where $\chi\in C^\infty(\mathbb{R})$ is a smooth cut-off function satisfying $\chi|_{|t|<1/2}\equiv 1$ and $\chi|_{|t|>1}\equiv 0$. Further, the coefficient functions are given by
$$
u^r_{\pm}(x,y):= 
\begin{cases}
u_{\pm}(x,y), & x > 0,\\
0, & x < 0.
\end{cases}
$$ 
and similarly
$$
u^l_{\pm}(x,y):= 
\begin{cases}
u_{\pm}(-x,y), & x < 0,\\
0, & x > 0.
\end{cases}
$$ 
Here $u_{\pm}$ are functions with asymptotic expansion~\eqref{eq:series_nondeger}. Assume first that their principal terms are smooth functions of $y\in \cZ$. Functions $u^{l,r}_{\pm}$ can have all principal coefficients different from each other. For the sake of brevity, we denote the coefficients of the principal term as $a_{\pm}^{r,l}$ omitting the zeros. Also since those functions are compactly supported in a small tubular neighborhood around $\cZ$ we extend them by zero.

Using the Green's identity we can write down explicitly the asymmetry form
\begin{align}
\omega_{(\Delta-cS)^*}(u,v) &= \int_{\R \times \cZ} (\overline u \Delta v - v\Delta \overline u) \omega = \lim_{\varepsilon \to 0+}\int_{|x| \geq \varepsilon} (\overline u \Delta v - v\Delta \overline u) \omega = \nonumber \\
&=\lim_{\varepsilon \to 0+} \left(-\int_{x=\varepsilon} (\overline{u}\p_x v - v \p_x \overline{u}) \sigma_\varepsilon - \int_{x=-\varepsilon}(\overline{u}(-\p_x)v - v (-\p_x) \overline{u}) \sigma_{-\varepsilon})\right), \label{eq:green}
\end{align}
where $\sigma_{\pm \varepsilon}$ 
are induced measures on submanifolds $\{x=\pm \varepsilon\}$. Note that all of them are diffeomorphic to $\cZ$. 

\subsubsection{Case $\mu < 0$}

In this case the indicial roots are complex conjugates $\lambda $ and $\bar \lambda$. For $x>0$ let us write the asymptotic expansions for $u$ and $v$:
\begin{align}
    u(x,y) &= a^r_-(y) x^{\bar \lambda} + a^r_+(y) x^{\lambda} + o(x^{\Re (\lambda)}), \label{eq:exp_u_1}\\
    v(x,y) &= b^r_-(y) x^{\bar \lambda} + b^r_+(y) x^{\lambda} + o(x^{\Re (\lambda)}), \label{eq:exp_v_1}
\end{align}
For $x<0$ we have a similar expression, which we omit. Assume for now that coefficients $a^r_\pm$ and $b^r_\pm$ are smooth functions on the singular set $\cZ$. We can then substitute these asymptotic expansions into~\eqref{eq:green}. From this we have
\begin{equation}
\label{eq:holder}
\int_{x=\varepsilon} (\overline{u}\p_x u - u \p_x \overline{u}) \sigma_\varepsilon = \int_{x=\varepsilon} 2i \Im (\lambda) \left( (\bar{a}^r_+b^r_+-\bar{a}^r_-b^r_-)  + o(1) \right)\varepsilon^{2\Re (\lambda) - 1}\sigma_\varepsilon(y).
\end{equation}
Next we write a similar expression for $x<0$ and take the limit when $\varepsilon \to 0+$. This limit exists by an application of the dominated convergence theorem since $2\Re (\lambda)- 1 = \alpha n$, and simultaneously 
$$
\sigma_\varepsilon = \varepsilon^{-\alpha n}\tilde{\sigma}_\varepsilon,
$$
leading to cancelation of the singular orders of $\varepsilon$. Here $\tilde{\sigma}_\varepsilon $ is the restriction of the companion volume form $\tilde \omega$ to $x=\varepsilon$, which can be viewed as a family of smooth volume forms on $\cZ$ with a well-defined limit at $\varepsilon \to 0$. Doing this procedure gives us the final formula for the assymetry form
$$
\omega_{(\Delta-cS)^*}(u,v) = i\sqrt{|\mu|} \int_{\cZ}\left( \bar{a}^r_+b^r_+-\bar{a}^r_-b^r_- + \bar{a}^l_+b^l_+-\bar{a}^l_-b^l_- \right) \tilde \sigma(y).
$$
We introduce new variables
$$
a_\pm = \begin{pmatrix}
   a^r_\pm\\
   a^l_\pm
\end{pmatrix}, \qquad
b_\pm = \begin{pmatrix}
   b^r_\pm\\
   b^l_\pm
\end{pmatrix},
$$
to shorten the last expression to
\begin{equation}
    \label{eq:form_mu0_L2}
    \omega_{(\Delta-cS)^*}(u,v) = i\sqrt{|\mu|} \left( \langle a_+,b_+\rangle_{L^2(\cZ)} - \langle a_-,b_-\rangle_{L^2(\cZ)}\right).
\end{equation}

Note that this expression can be extended to the case when these principal coefficients are functions in $L^2(\cZ)$. However, Theorem~\ref{thm:phg_Grushin} tells us that this is not enough, because these coefficients can lie in some negative order Sobolev spaces and there is no natural extension of the form above to a pair of such functions. Nevertheless, if either $a$ or $b$ is smooth, then this pairing indeed is well-defined. To see this, consider the operator $-\Delta_{0,\cZ}$, the regularized Laplacian operator on the singular set (see~\eqref{eq:local_expression_laplace}). This is a non-negative elliptic operator on a closed manifold and hence has a discrete spectrum $\lambda_1 \leq \lambda_2 \leq \dots$ with the corresponding eigenfunctions $\{\varphi_k\}_{k\in \mathbb{N}}$, which are smooth functions by the classical elliptic regularity. These functions $\varphi_k$, $k\in \N$ form an orthonormal basis of $L^2(\cZ)$ and we get the usual analogue of the classical Fourier theory. In particular, every element can be represented with respect to this basis as
$$
f(y) = \sum_{k=1}^\infty f_k\varphi_k(y)
$$
and one can characterize Sobolev spaces $H^s(\cZ)$ as distributions $f$ on $\cZ$ for which
$$
\sum_{k=1}^\infty (1+\lambda_k)^s |f_k|^2 < \infty.
$$
For smooth functions $f$ the series will converge for any $s\in \R$.

Now assume that $a_\pm$ and $b_\pm$ are smooth. Set $a_{\pm,k}$ and $b_{\pm,k}$ be their Fourier coefficients in the above sense. We can rewrite expression~\eqref{eq:form_mu0_L2} as

\begin{equation}
    \label{eq:form_mu0_fourier}
    \omega_{(\Delta-cS)^*}(u,v) = i\sqrt{|\mu|} \sum_{k=1}^\infty \left( \langle a_{+,k},b_{+,k}\rangle_{\C^2} - \langle a_{-,k},b_{-,k}\rangle_{\C^2}\right).
\end{equation}
The crux of the proof is that if we keep one of the functions smooth, then by a density argument, we can assume that the other function belongs to any $H^s(\cZ)$, as the corresponding series remains convergent. Hence the above formula holds as long as one of the two functions is assumed to be smooth. So we can proceed as follows. First we find a reasonable candidate for a self-adjoint extension, i.e. one that is represented by a linear subspace of $D_{max}(\Delta - cS)/D_{\min}(\Delta - cS)$. Then we prove that this subspace is in fact Lagrangian, by testing against functions with a smooth principal term via~\eqref{eq:form_mu0_fourier}.

Using the formula for the asymmetry form we are now ready to construct self-adjoint extensions of $\Delta - cS$. We start with the following lemma.
\begin{lem}
\label{lem:non-degenerate_form}
    For $\mu<0$ the form $\omega_{(\Delta-cS)^*}$ is non-degenerate on $D_{max}(\Delta - cS)/D_{\min}(\Delta - cS)$.
\end{lem}

\begin{proof}
    It suffices to check that for any $u\in D_{max}(\Delta - cS)$ outside the closure there exists $v\in D_{max}(\Delta - cS)$, such that $\omega_{(\Delta-cS)^*}(u,v) \neq 0$. By our characterization of the adjoint, $u$ must have an asymptotic expansion of the form~\eqref{eq:exp_u_1}, with the principal coefficient being in some Sobolev space by Theorem~\ref{thm:phg_Grushin}. Since $u$ is not in the closure, there exists $k\in \N$ such that at least one of the $k$-th harmonics, $a_{+,k}$ or $a_{-,k}$, is not zero. Then we choose $v \in D_{max}(\Delta - cS)$ with an asymptotic expansion with principal term in~\eqref{eq:exp_v_1} given by $b^{r,l}_{\pm}(y) = \pm a^{r,l}_{\pm,k}\varphi_k(y)$, which exists by Lemma \ref{lem:borellemma}. Then we have
    \[ \omega_{(\Delta-cS)^*}(u,v) = i\sqrt{|\mu|} \left(|a_{+,k}|^2 + |a_{-,k}|^2 \right) > 0. \]
\end{proof}

Now we can construct some natural self-adjoint extensions of $\Delta - cS$.

\begin{proof}[Proof of Theorem~\ref{thm:sa_extensions_mu_leq_0}]

We start by noting that the bilinear form $\omega_{(\Delta-cS)^*}$ is completely determined by its quadratic form. We have from~\eqref{eq:form_mu0_fourier}
$$
\omega_{(\Delta-cS)^*}(u) = i\sqrt{|\mu|} \sum_{k=1}^\infty \left( |a_{+,k}|^2_{\C^2} - |a_{-,k}|^2_{\C^2}\right).
$$
So if $U:\C^2 \to \C^2$ is a unitary operator, we can consider a closed subspace $\Lambda_U$ of $D_{max}(\Delta - cS)$ determined by the relations $a_{+,k} = Ua_{-,k}$. First note that this space is isotropic. Indeed, by definition $\cA_{phg}(M\setminus \cZ)\cap\Lambda_U$ is dense in $\Lambda_U$. For this space we can use formula~\eqref{eq:form_mu0_fourier} to check that $\omega_{(\Delta-cS)^*}|_{\cA_{phg}(M\setminus \cZ)\cap\Lambda_U} = 0$. Thus, we can extend it to whole $\Lambda_U$ by continuity. It only remains to prove that this subspace is maximal. If it is not maximal, then there should exist $v \notin \Lambda_U$ such that $\omega_{(\Delta-cS)^*}(u,v) = 0$ for all $u\in \Lambda_U$. Inverting this statement we see that maximality will follow, if for any $v \in D_{max}(\Delta - cS)$ such that $v\notin \Lambda_U$, one can find $u\in \Lambda_U$ such that $\omega_{(\Delta-cS)^*}(u,v) \neq 0$. 

The proof of this fact is similar to the proof of Lemma~\ref{lem:non-degenerate_form}. If $v \notin \Lambda_U$, then there must exist at least one harmonic $b_k$ such that $b_{+,k} \neq Ub_{-,k}$. Choose $u \in \Lambda_U$ be a polyhomogeneous function whose first coefficient $a_-(y) = a_{-,k}\varphi_k(y)$ and $a_+(y) = Ua_{-,k}\varphi_k(y)$. Then we have
$$
\omega_{(\Delta-cS)^*}(u,v) = i\sqrt{|\mu|} \left( \langle a_{+,k},b_{+,k}\rangle_{\C^2} - \langle a_{-,k},b_{-,k}\rangle_{\C^2}\right).
$$
The quadratic form on the right is a symplectic form on $\C^4$, and the graph of $U$ given by $a_{+,k} = Ua_{-,k}$ determines a Lagrangian subspace $\Lambda_k$, because every isotropic subspace of dimension two is Lagrangian. But by construction, $(b_{+,k},b_{-,k})$ does not belong to $\Lambda_k$, and hence, by maximality there must exist $(Ua_{-,k},a_{-,k})$ for which the expression above is not zero.
\end{proof}

\subsubsection{Case $0\leq \mu<4$}

In this case, we have two real roots $\lambda_- < \lambda_+$. In addition, we work with assumptions of Theorem~\ref{thm:sa_extensions_mu_geq_0}. We use those assumptions, since we want to derive a formula analogous to~\eqref{eq:form_mu0_L2} and exploit a similar strategy for our proofs. To check that the limit~\eqref{eq:green} exists, we need an asymptotic expansion of an element in $\ker(\Delta -cS)^*$, (corresponding to the negative root $\lambda_-$), up to order of the positive root $\lambda_+$. Since $\lambda_+ - \lambda_- = \sqrt{\mu}<2$ we only need to consider powers of two more orders. If we go back to the derivation of the asymptotic expansion in Section~\ref{subsec:adjoint}, we see that since $\alpha>0$, all of the terms of interest will be multiples of the first term, because the first derivation in $y$ variables will be of order $2+2\alpha > 2$. Secondly, for simplicity, the second assumption implies that there roots are not separated by an element of the index set $\Theta$ to avoid log terms. Finally, the last assumption guarantees that we can apply the Fourier argument from the case $\mu<0$ with very little modification. 

All of this implies that we have the following assymptotic expansions for $x>0$:
\begin{align}
    u(x,y) &= a^r_-(y) q(x,y) x^{\lambda-} + a^r_+(y) x^{\lambda_+} + o(x^{\lambda_+}), \label{eq:exp_u_1}\\
    v(x,y) &= b^r_-(y) q(x,y) x^{\lambda-} + b^r_+(y) x^{\lambda_+} + o(x^{\lambda_+}), \label{eq:exp_v_1}
\end{align}
where $q(x,y)$ is a real polynomal in $x$ with $q(0,y) = 1$ and all the other coefficients depend on $\lambda_-$, asymptotic expansion of the scalar curvature $S$ and the divergence of $\p_x$ in such a way that $\p_x q(0,y) > 0$. 

As before, let us assume that the coefficients are smooth functions on the singular set. This asymptotic expression gives us after various simplifications for $x>0$:
$$
\bar u \p_x v - v \p_x \bar u = (\bar a^r_- b^r_+ - \bar a^r_+b^r_-)(q\lambda_+ - q \lambda_- - \p_x q)x^{\lambda_+ + \lambda_- - 1} + o(x^{\lambda_+ + \lambda_- - 1})
$$
Again, recall that $\lambda_+ + \lambda_- - 1 = \alpha n$. Hence, repeating the same procedure to the left of the singular set gives us via~\eqref{eq:green}the following formula:
    $$
\omega_{(\Delta-cS)^*}(u,v) = \int_\cZ (\bar a^r_- b^r_+ - \bar a^r_+b^r_-+\bar a^l_- b^l_+ - \bar a^l_+b^l_-)(\sqrt{\mu} - \p_x q(0,y))\tilde\sigma(y)
    $$
The justification is exactly as in the case $\mu<0$. By assumption $|\sqrt{\mu} - \p_x q|\tilde\sigma$ is a volume form on $\cZ$ coming from a conformal transformation of the companion metric $\cZ$. This new metric has an associated Laplace-Beltrami operator, and we can use this new Laplacian with eigenfunctions $\hat \varphi_k(y)$ to apply the harmonic analysis argument from the case $\mu<0$. Expanding in Fourier terms will result in a formula:
$$
\omega_{(\Delta-cS)^*}(u,v) = \sum_{k=1}^\infty \left(\langle a^r_{-,k}, b^r_{+,k}\rangle_\C - \langle  a^r_{+,k},b^r_{-,k}\rangle_\C+\langle  a^l_{-,k}, b^l_{+,k}\rangle_\C - \langle  a^l_{+,k},b^l_{-,k}\rangle_\C\right).
$$
Again, We see that this allows us to extend this formula to the case when one of the functions belongs to $D_{max}(\Delta - cS)$ while the other one remains smooth. This implies the non-degeneracy of the asymmetry form.

\begin{lem}
\label{lem:non-degenerate_form}
    Under the assumptions of Theorem~\ref{thm:sa_extensions_mu_geq_0} the form $\omega_{(\Delta-cS)^*}$ is non-degenerate on $D_{max}(\Delta - cS)/D_{min}(\Delta - cS)$.
\end{lem}

\begin{proof}
    As before it is enough to check that for any $u\in D_{max}(\Delta - cS)$ which does not lie in the closure, there exists $v\in D_{max}(\Delta - cS)$, such that $\omega_{(\Delta-cS)^*}(u,v) \neq 0$. Assume that $u\in D_{max}(\Delta - cS)$ is not in the closure. Then there exists $k\in N$ such that at least one of the $k$-th harmonics $a^{r,l}_{+,k}$is not zero. Then we choose $v \in D_{max}(\Delta - cS)$ that has polyhomogeneous asymptotics with principal term in~\eqref{eq:exp_v_1} given by $b^{r,l}_{\pm}(y) = \pm a^{r,l}_{\pm,k}\varphi_k(y)$. This gives 
    $$
\omega_{(\Delta-cS)^*}(u,v) = |a^r_{+,k}|^2_\C + |a^r_{-,k}|^2_\C + |a^l_{+,k}|^2_\C + |a^l_{-,k}|^2_\C > 0.
    $$
\end{proof}

Now we are ready to construct self-adjoint extensions. In order to guess how they might look, we introduce the following variables:
\begin{align*}
    A_{1,k} &= \begin{pmatrix}
    a^r_{+,k} + ia^r_{-,k}\\
    a^l_{+,k} + ia^l_{-,k}
\end{pmatrix},
&
A_{2,k} &= \begin{pmatrix}
    a^r_{+,k} - ia^r_{-,k}\\
    a^l_{+,k} - ia^l_{-,k}
\end{pmatrix}.
\\
    B_{1,k} &= \begin{pmatrix}
    b^r_{+,k} + ib^r_{-,k}\\
    b^l_{+,k} + ib^l_{-,k}
\end{pmatrix},
&
B_{2,k} &= \begin{pmatrix}
    b^r_{+,k} - ib^r_{-,k}\\
    b^l_{+,k} - ib^l_{-,k}
\end{pmatrix}.
\end{align*}

Then we get
$$
\omega_{(\Delta-cS)^*}(u) =  \frac{i}{2}\sum_{k=1}^\infty(|A_{1,k}|^2_{\C^2}-|A_{2,k}|^2_{\C^2}).
$$
As before choose a unitary operator $U:\C^2 \to \C^2$. Then we can define a closed subspace $\Lambda_U$ in $\omega((\Delta - cS)^*)$ by gluing the $k$-th harmonics via the relation $A_{2,k} = UA_{1,k}$. The asymmetry form on $\Lambda_U$ intersected with functions that have polyhomogenenous asymptotics is zero. By a density argument we see that it will be zero on all $\Lambda_U$. Hence $\Lambda_U$ will be isotropic. Similarly to the case $\mu<0$ we get the proof of Theorem~\ref{thm:sa_extensions_mu_geq_0}.

\begin{proof}[Proof of Theorem~\ref{thm:sa_extensions_mu_geq_0}]
The proof follows exactly the same lines as the proof of the case $\mu<0$. We have already seen that the asymmetry form is non-degenerate and that $\Lambda_U$ is isotropic. It only remains to check that it is Lagrangian. Let $\pi_k$ be the map that takes a function in $u \in D_{max}(\Delta - cS)$ and gives the $k$-th harmonics of the principal term. However, it is clear that $\omega_{(\Delta - cS)^*}(\pi_k \cdot, \pi_k \cdot)$ is a symplectic form on $\C^4$ and $\pi_k \Lambda_U$ is a Lagrangian space by construction. 

Let $u\notin \Lambda_U$. Choose a $k\in \N$, such that $\pi_k u \notin \pi_k \Lambda_U$. Then there must exist $v\in \pi_k \Lambda_U$, such that $\omega_{(\Delta - cS)^*}(\pi_k u, v) \neq 0$. Extend $v$ to $\hat v\in D_{max}(\Delta - cS)$ by taking all of the remaining harmonics of the principal term to be zero. Then $\hat{v}$ is represented by a polyhomogeneous function and we can use the formula for the asymmetry form to show that
$$
\omega_{(\Delta - cS)^*}(u, \hat v) = \omega_{(\Delta - cS)^*}(\pi_k u, v) \neq 0,
$$
which proves that $\Lambda_U$ must be maximal.
\end{proof}

\section{Construction of the $\alpha$-calculus}

\label{sec:proofs}

\subsection{Triple-stretched product}

The double-stretched product is used to lift Schwartz kernels of operators adapted to the geometry of manifolds and in the construction of a suitable class of pseudo-differential operators. The triple-stretched product is needed to understand the mapping properties in this pseudo-differential calculus. Similarly to $M^2_\alpha$ it is built with a series of blow-ups. As in Subsection~\ref{subsec:double_prod} we assume that $\alpha\geq 0$. The case $-1<\alpha <0$ is handled similarly.

On $M^3$ we have the three projections $\pi_L,\pi_M,\pi_R$ to the left, middle and right factors. The space $M^3_\alpha$ should factor smooth maps $\pi^\Lambda_{LR},\pi^\Lambda_{LM},\pi^\Lambda_{MR}: M^3_\alpha \to M^2_\alpha$ obtained by essentially projecting to the two of three factors, i.e., if $\beta_\alpha^{(i)}: M_\alpha^i \to M^i$ are blow-down maps, then the following diagram should be commutative
\begin{equation}
\label{diag:triple_commute}
    \begin{tikzcd}
	{M_\alpha^3} &&& {M_\alpha^2} \\
	\\
	{M^3} &&& {M^2}
	\arrow["{\beta_\alpha^{(3)}}"', from=1-1, to=3-1]
	\arrow["{\pi_O}", from=3-1, to=3-4]
	\arrow["{\pi_O^\Lambda}", from=1-1, to=1-4]
	\arrow["{\beta_\alpha^{(2)}}", from=1-4, to=3-4]
\end{tikzcd}
\end{equation}

\noindent where $O\in \{LM,MR,LR\}$.

In order to do this, consider the three sets $(\p \Delta)_O = \pi_O^{-1}(\partial \Delta)$, $O\in \{LM,MR,LR\}$ inside $M^3$. To each set we can associate one of the forms
\begin{align*}
\omega_{LM} &= \pi_L^*\omega - \pi_M^*\omega,\\
\omega_{MR} &= \pi_M^*\omega - \pi_R^*\omega,\\
\omega_{LR} &= \pi_L^*\omega - \pi_R^*\omega.
\end{align*}
Thus we need to blow-up each of those diagonal using filtrations $\cF_O$
$$
\{0\}\subset \ker \omega_O \subset N(\p\Delta)_O.
$$
The problem is that the three sets $(\p\Delta)_O$ actually meet at the diagonal at $\p\Delta_T = \{(p,p,p):p\in \p\Delta\}$, where a different filtration must be chosen. Namely, we must take the filtration $\cF_T$
 $$
\{0\}\subset \bigcap_{O\in \{LM,MR,LR\}}\ker \omega_O \subset N(\p\Delta_T).
$$
To make everything compatible we assign to $\ker \omega_O$ and $\bigcap\ker \omega_O$ order $1$, while for $N(\p\Delta)_O$ and $N(\p\Delta_T)$ order $1+\alpha$, which gives the same order vector $\Lambda = (1,1+\alpha)$. 

The triple stretched product $M^3_\alpha$ is obtained by blowing up iteratively $M_3$ first along $\p\Delta_T$, and then along the lifts of $\p\Delta_O$, $O\in \{LM,MR,LR\}$. Putting this in formulas we get

\begin{definition}
The stretched triple product is defined as the iterated blow-up
$$
M_3^\alpha = \left[[M^3;\p \Delta_T]_{(\cF_T,\Lambda)}; \overline{\p\Delta_{LM}}\cup \overline{\p\Delta_{MR}}\cup \overline{\p\Delta_{LR}} \, \right]_{\{(\cF_{LM},\cF_{MR},\cF_{LR}\},\Lambda)},
$$
where $\overline{\p\Delta_O}$ is the closure of the preimage of $\p\Delta_O$ under the blow-down of the blow-up of $\p\Delta_T$.
\end{definition}

The commutativity of the diagram~\eqref{diag:triple_commute} is a direct corollary of the commutativity of blow-ups in this case:
\begin{thm}[\cite{thesis}, Theorem 3.3]
Let $X$ be a smooth manifold and $Z\subset Y \subset X$ be two nested $p$-submanifolds with two blow-up data $(\cF_Y,\Lambda_Y)$, $(\cF_Z,\Lambda_Z)$. Assume that $\cF_Z$ is the restriction of $\cF_Y$ to the normal bundle $NZ$. Then the identity map on $X$ lifts to a diffeomorphism
$$
[[X;Z]_{(\cF_Z,\Lambda_Z)};Y]_{(\cF_Y,\Lambda_Y)} = [[X;Y]_{(\cF_Y,\Lambda_Y)};Z]_{(\cF_Z,\Lambda_Z)}.
$$
\end{thm}

\begin{remark}
In~\cite{thesis} this theorem is proven in a more general setting of clean intersections. But for our purposes, this level of generality is more than enough.
\end{remark}

The triple product has seven faces. Three of them are lifts of the interiors of $\p M \times M^2$, $M \times \p M \times M$, $\p M \times M^2$ which are denoted by $\cB_{100}$, $\cB_{010}$, $\cB_{001}$. Three more are blow-up of $\p \Delta_ {LM}$, $\p \Delta_ {MR}$, $\p \Delta_ {LR}$ which are denoted by  $\cB_{110}$, $\cB_{011}$, $\cB_{101}$. Finally the last face $\cB_{111}$ is the blow up of $\p \Delta_T$.

\begin{figure}[h]
\centering
\includegraphics[scale=.35]{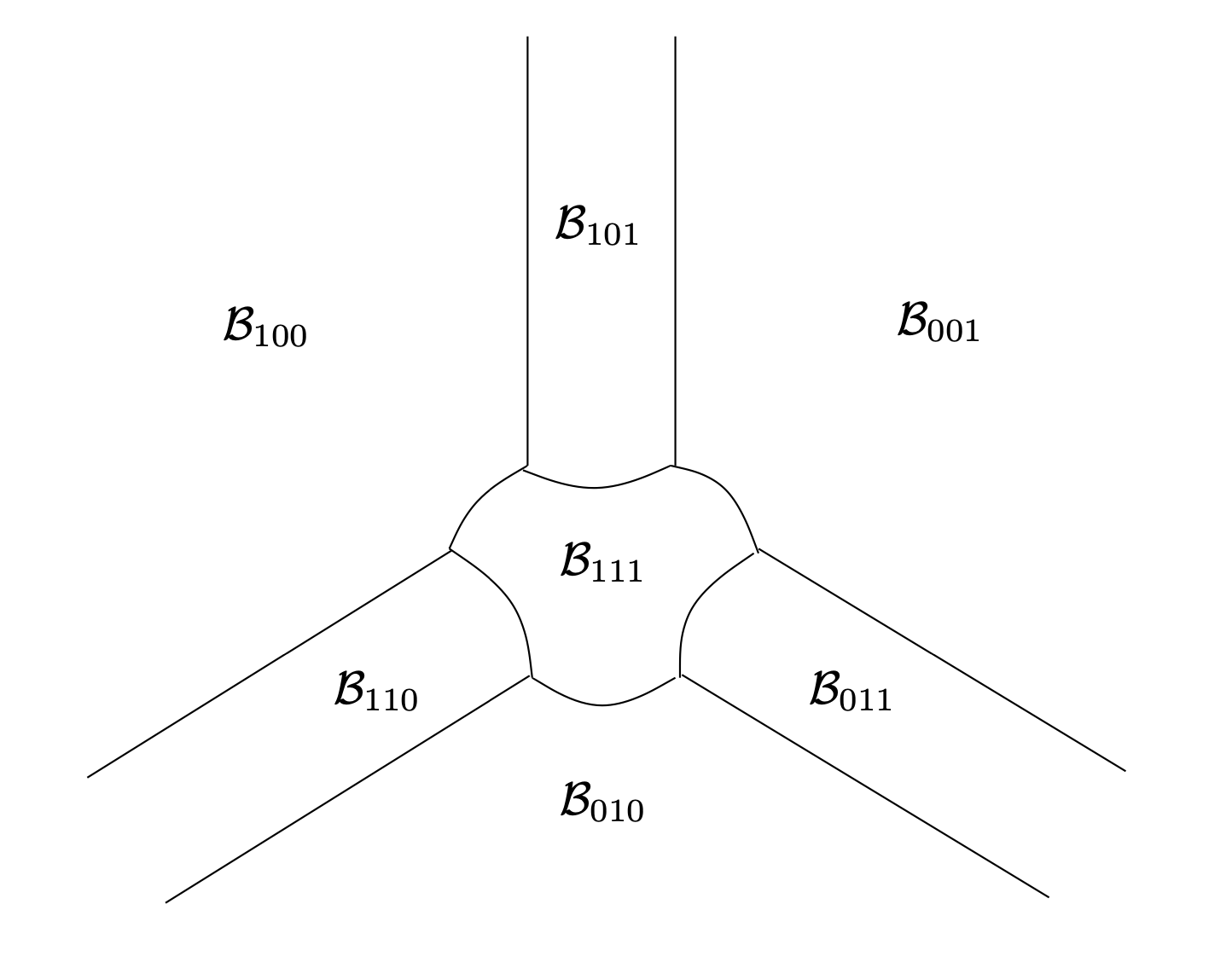}
\caption{The triple space $M_\alpha^3$}
\end{figure}

\subsection{Mapping properties}

We can lift operators which act on functions on $M$ to operators which act on functions on $M_\alpha$, through their kernels. Before explaining the exact procedure, we need a few lemmas concerning the blow-down maps.

\begin{lem}
\label{lemm:blow-up_matrix}
Let $S \subset M$ be a $p$-manifold with blow-up data $(\cF,\alpha)$. Let $B'_i$, $i=1,\dots, q$ be boundary hypersurfaces of $M$ and $r'_i$ be their boundary defining functions. 

The blow-down map is a $b$-map, i.e. we have 
$$
\beta^* r'_i = f\prod_{j=0}^q r_j^{e(i,j)},
$$
where $r_j$ are defining function of boundary hypersurfaces $\cB_j$ of $[M;S]_{(\cF,\alpha)}$ and $f\neq 0$ is a smooth function on $[M;S]_{(\cF,\alpha)}$. In particular, we assume that $r_0$ is the defining function of the front face. More explicitly
$$
e(i,j)=\begin{cases}
0, & i\neq j, j\neq 0,\\
1, & i=j,\\
d(i), & j=0.
\end{cases}
$$
Here $d(i) = 0$ if $\cB_i\cap N = \emptyset$, and $d(i) = \alpha_l$ if $(\cB_i)^\perp \cap V_l \neq \{0\}$, $(\cB_i)^\perp \cap V_{l-1} = \{0\}$, where $V_l$ are the sub-bundles of filtration $\cF$. 
\end{lem}

\begin{proof}
Since $S$ is a $p$-submanifold, we can find a local system of coordinates for which $S$ coincides with $L_I$ from Definition~\ref{def:p-manifold}. Hence we get local coordinates $(y,x)$, such that $S$ is identified with $x=0$. Moreover, With a further change of variables, we can identify the quotients $V_i/V_{i-1}$ of the filtration $\cF$ with $L_J$ for some index sets $J$. Now we can construct projective coordinate charts~\eqref{eq:projective_coordinates} given by $(y,\xi)$. Since we are interested in the pull-backs, we invert this change of coordinates to find for a fixed index $i$
$$
x_i = \xi_i^{\text{ord}(x_i)}, \qquad x_j = \xi_j\xi_i^{\text{ord}(x_j)},
$$
for $j\neq i$ and $x_i\neq 0$. Note that $\xi_i $ corresponds to a boundary defining function on the front-face in this coordinate chart. So assume that $x_j$ is a boundary defining function of the face $\cB_j$. Then clearly, from the formulas above we have $e(j,j)=1$, $e(j,0)=\text{ord}(x_j)$ and the remaining ones are zero. If $y_j$ is a defining function of $\cB_j$, then it lifts simply to $y_j$ and hence $e(j,j)=1$, while the remaining ones are zero. 
\end{proof} 

Spherical coordinates allow us to prove the following transformation law for the densities.

\begin{lem}
\label{lemm:densities}
Let $M$ be a manifold, $S$ be a $p$-submanifold and $(\cF,\alpha)$ the blow-up data. If $\omega \in \Omega(M)$, then
$$
\beta^* \omega = r^{s-1} \sigma 
$$
where $r$ is the defining boundary function for the front face, $\sigma$ is a continuous section of $\Omega([M;S]_{(\cF,\alpha)})$, whose restriction to the interior is smooth and
$$
s = \sum_{i = 1}^n \alpha_i (\dim V_i - \dim V_{i-1}).
$$ 
\end{lem}

\begin{proof}
    As in the proof of the previous lemma, we choose local coordinates $(y,x_1,\dots, x_k)$ in a way, such that $S$ is identified with $x_i = 0$, $1\leq i \leq k$ and $V_i/V_{i-1}$ with $y=0$, $x_j=0$ for $j\neq i$. Note that unlike in the previous lemma, here $x_i$ is a set of coordinates for $V_i/V_{i-1}$  and not individual components. We fix a reference non-vanishing form $\omega \in \Omega(M)$, which in those coordinates locally is of the form
    $$
    \omega = dy dx_1 \dots  dx_k.
    $$
    Now we can simply apply the spherical change of variables~\eqref{eq:spherical_coordinates}, by taking $x_i =r^{\alpha_i}\theta_i$. After straight-forward calculations we obtain 
    $$
    \beta^*\omega = r^{s-1} dy dr  d\theta, 
    $$
    where $d\theta$ is a smooth density on the level set $r=1$ and
    $$
    s = \sum_{i = 1}^n \alpha_i \dim (V_i/V_{i-1}),
    $$
    which finishes the proof.
\end{proof}

Let us now see how the kernels of operators on $M^2$ are lifted to $M^2_\alpha$. Pick a non-vanishing section $\gamma_\alpha$ of $\Omega_{\alpha}$. It induces a non-vanishing section of the density bundle $\Omega(M^2_\alpha)$:
\begin{equation}
\label{eq:density}
\mu_\alpha = \beta^*_L (\gamma_\alpha)\beta^*_R (\gamma_\alpha)
\end{equation}
where $\beta_{L,R} = (\beta_\alpha^2)^* \circ \pi_{L,R} $.

Given an operator $A\in \Psi^{s,\cE}_\alpha(M)$ we can write its kernel as $\cK_A \sqrt{\mu_\alpha}$. Then the action of $A$ on $\Omega^{1/2}_\alpha(M)$ is given by
\begin{equation}
\label{eq:op_action}
A (u \sqrt{\gamma_\alpha}) = (\beta_L)_* (\cK_A \sqrt{\mu_\alpha} \beta_R^*(u\sqrt{\gamma_\alpha})).
\end{equation}

\begin{thm}
\label{thm:action_alpha}
Let $M$ be a manifold of dimension $n+1$ with boundary, $\cE, F$ index sets and $A\in \Psi_\alpha^{s,\cE}$. If $\Re(E_{01}+F)> (1+\alpha) n$, then
$$
A: \cA_{phg}^F(M,\Omega^{1/2}_\alpha) \to \cA_{phg}^{E_{10}\overline{\cup} (E_{11} + F)}(M,\Omega^{1/2}_\alpha).
$$
\end{thm}

\begin{proof}
In order, to simplify the notations, if $\gamma,\mu$ are two densities, we write
$$
\gamma \sim \mu,
$$ 
whenever $\gamma = a \mu$, where $a$ is a smooth non-vanishing function.

In local coordinates close to $\cZ$ we have
$$
\gamma_\alpha \sim \frac{dxdy}{x^{1+(1+\alpha) n}},
$$
This implies that 
$$
\pi_L^* \gamma_\alpha \pi_R^* \gamma_\alpha \sim \frac{dxdyd\tilde{x}d\tilde{y}}{x^{1+(1+\alpha) n} \tilde x^{1+(1+\alpha) n}}.
$$
Both $x$ and $\tilde{x}$ are defining functions for the boundary surfaces of $M^2$. Using Lemmas~\ref{lemm:blow-up_matrix}, \ref{lemm:densities} and Definition~\ref{def:double_product} we find, that there exists $\mu \in \Omega(M^2_\alpha)$ and $\mu_b \in \Omega_b(M^2_\alpha)$ such that
$$
\mu_\alpha = \beta^*(\pi_L^* \gamma_\alpha \pi_R^* \gamma_\alpha) = \frac{\rho_{11}^{2+(1+\alpha) n - 1 }}{\rho_{01}^{1+(1+\alpha) n}\rho_{10}^{1+(1+\alpha) n} \rho_{11}^{2+2(1+\alpha) n}}\mu = (\rho_{01}\rho_{10}\rho_{11})^{-(1+\alpha) n} \mu_b.
$$

We thus have by definition
\begin{align*}
A(u\sqrt{\gamma_\alpha})\sqrt{\gamma_\alpha}&= (\beta_L)_*(\cK_A \sqrt{\mu_\alpha} \beta_R^*(u\sqrt{\gamma_\alpha}))\sqrt{\gamma_\alpha} = (\beta_L)_*(\cK_A \sqrt{\mu_\alpha} \beta_R^*(u\sqrt{\gamma_\alpha})\beta_L^*(\sqrt{\gamma_\alpha}))= \\
&= (\beta_L)_*(\cK_A  \beta_R^*(u)\mu_\alpha) ) = (\beta_L)_*(\beta_R^*(u) \cK_A  (\rho_{01}\rho_{10}\rho_{11})^{-(1+\alpha) n }\mu_b) ).  
\end{align*}
Thus by the push-forward theorem~\ref{prop:push} if $Re(E_{01}+F)> (1+\alpha)n$ it belongs to
$$
\cA^{(E_{10}-(1+\alpha)n)\overline{\cup} (E_{11}+F-(1+\alpha)n)}(M,\Omega_b) = \cA^{E_{10}\overline{\cup} (E_{11}+F)}(M,\Omega_\alpha),
$$
where the equality follows directly from the Definition~\ref{def:extended_union} of the extended union. Dividing by $\sqrt{\gamma_\alpha}$ finishes the proof.

\end{proof}

Next, we would like to consider how the two operators in this calculus compose. Recall that given a triple space we have maps $\beta_{LM},\beta_{MR}, \beta_{LR}:M^3_\alpha \to M^2_\alpha$. Given linear operators $A, B$ and their kernels $K_A$ and $K_B$ we can define their composition $C = A \circ B$ through the kernel
$$
K_C = (\beta_{LR})_* (\beta_{LM}^* (K_A) \beta_{MR}^* (K_B))
$$

\begin{thm}
\label{eq:composition_alpha}
Let $M$ be a manifold and $\dim M = n+1$, $\cE,\cF$ be index sets and $A\in \Psi^{s,\cE}_\alpha$, $B\in \Psi^{t,\cF}_\alpha$. If $\Re(E_{01}+ F_{10})> (1+\alpha) n$, then $A \circ B \in \Psi^{s+t,\cG}_\alpha$, where
\begin{align*}
G_{10}&=(E_{11} + F_{10} ) \overline{\cup} E_{10},\\
G_{01}&=(E_{01} + F_{11} ) \overline{\cup} F_{01},\\
G_{11}&=(E_{11} + F_{11} ) \overline{\cup} (E_{10}+F_{01}).
\end{align*}
 
\end{thm}

\begin{proof}
The proof is similar to the proof of Theorem~\ref{thm:action_alpha}. We use $\gamma_\alpha$, $\gamma_b$, $\mu_\alpha$, $\mu_b$ as in the previous proof, and
\begin{align}
    \nu_\alpha = \beta_{LR}^*(\mu_\alpha)\beta_{LM}^*(\mu_\alpha)\beta_{MR}^*(\mu_\alpha)
\end{align}
Similarly to the previous proof, using Lemmas~\ref{lemm:blow-up_matrix},~\ref{lemm:densities} we obtain, that there exists a $b$-density $\nu_b$ on $\Omega^3_\alpha$
$$
\nu_\alpha = \frac{1}{\overline{\rho}^{(1+\alpha)n}}\nu_b,
$$
where $\overline{\rho}$ is the product of defining functions for each face of $\overline{\rho}$.

By definition 
$$
\cK_{A\circ B} \sqrt{\mu_\alpha} = (\beta_{LR})_*(\beta_{LM}^*(\cK_A \sqrt{\mu_\alpha})\beta_{MR}^*(\cK_A \sqrt{\mu_\alpha})).
$$
We multiply both sides by $\sqrt{\mu_\alpha}$ and simplify:
$$
\cK_{A\circ B} \mu_\alpha = (\beta_{LR})_*(\beta_{LM}^*(\cK_A)\beta_{MR}^*(\cK_B)\nu_\alpha) = (\beta_{LR})_*(\beta_{LM}^*(\cK_A)\beta_{MR}^*(\cK_B)\overline\rho^{-(1+\alpha)n}\nu_b).
$$
Once again we apply the push-forward theorem~\ref{prop:push}. Then we find that the pushforward will be in $\cA^{\tilde \cG}(M_\alpha^2, \Omega_b)$, where
\begin{align*}
\tilde G_{10}&=(E_{11} + F_{10} - (1+\alpha)n) \overline{\cup} (E_{10}- (1+\alpha)n),\\
\tilde G_{01}&=(E_{01} + F_{11} - (1+\alpha)n) \overline{\cup} (F_{01}- (1+\alpha)n),\\
\tilde G_{11}&=(E_{11} + F_{11} - (1+\alpha)n) \overline{\cup} (E_{10}+F_{01}- (1+\alpha)n).
\end{align*}
But we have $\cA^{\tilde \cG}(M_\alpha^2, \Omega_b) = \cA^{\cG}(M_\alpha^2, \Omega_\alpha)$. The result now follows from dividing by $\sqrt{\mu_\alpha}$.

\end{proof}

Finally, we need the regularity properties of operators in the large calculus.

\begin{thm}\label{thm:mapprops}
Given $a,a' \in \C$ an operator $A\in \Psi^{s,\cE}_\alpha(M)$ extends to a bounded operator:
$$
A: x^a H^t_\alpha (M,\Omega^{1/2}_\alpha)\to x^{a'} H^{t'}_\alpha (M,\Omega^{1/2}_\alpha) 
$$
if $t' \leq t-s$, $\Re (E_{01} + a) > (1+\alpha) n$, $\Re (E_{10} - a') > (1+\alpha) n$, $\Re (E_{11} - a' + a) > 0$. 

Moreover, if $t' < t-s$ and $\Re (E_{11} -a' + a) > 0$, then this operator is also compact.
\end{thm}

The proof is analogous to \cite[Thm 3.25]{mazzeo}. We include here a version for completeness.

\begin{proof}
First, we prove the continuity result. Instead of $A$ we can consider the operator $x^{-a'}A x^{a}$. Then we have to consider $\Re (E_{01}+a)$ instead of $\Re (E_{01})$, $\Re (E_{10}-a')$ instead of $\Re E_{01}$, $\Re (E_{11}+a-a')$ instead of $\Re E_{11}$. So we only look at the operator
$$
A:  H^t_\alpha (M,\Omega^{\frac{1}{2}}_\alpha)\to  H^{t'}_\alpha (M,\Omega^{\frac{1}{2}}_\alpha) 
$$ 
with properly adjusted index set $\cE$.

Next note that $H^{t'}$ is embedded continuously in $H^{t-s}$ for $t' \leq t-s$. Thus it would be sufficient to prove the Theorem when $t' = t - s$. Also by definition of large and small calculus we can write 
$$
A=A' + A'',
$$
where $A'\in \Psi^s_\alpha(M)$ and $A'' \in \Psi_\alpha^{-\infty,\cE}(M)$. The statement for $A'$ follows by definition, so it only remains to prove the statement for $A''$. We can simplify even more by taking two invertible operators $B\in \Psi^{-(t-s)}_\alpha(M)$, $C\in \Psi^{-t}_\alpha (M)$ and considering $B A'' C$ as an operator from $L^2$ to itself. Thus we can also assume that $t=s=0$.

Finally, the boundedness result follows essentially from Schur's inequality. Using Cauchy-Schwartz in the second inequality and Fubini's theorem in the last equality below, we find
\begin{align*}
|\langle A''u \sqrt{\gamma_\alpha}, v \sqrt{\gamma_\alpha}\rangle| & = \left|\int_M (A''u\sqrt{\gamma_\alpha})v\sqrt{\gamma_\alpha} \right| = \left|\int_{M_\alpha^2} \beta_L^*\left((A''u\sqrt{\gamma_\alpha})v\sqrt{\gamma_\alpha}\right) \right| = \\
&= \left|\int_{M_\alpha^2} \cK_{A''} (\beta_R^* u) (\beta_L^*v) \mu_\alpha \right|  \leq  \int_{M_\alpha^2} |\cK_{A''}| |\beta^*_R u| | \beta^*_L v| \mu_\alpha \leq \\
&\leq \left(\int_{M_\alpha^2} |\cK_{A''}| |\beta^*_R u|^2 \mu_\alpha \right)^{\frac{1}{2}}\left( \int_{M_\alpha^2} |\cK_{A''}| |\beta^*_L v|^2 \mu_\alpha \right)^{\frac{1}{2}} = \\
&= \left(\int_{M}  | u|^2 ((\beta_R)_*|\cK_{A''}|\mu_\alpha) \right)^{\frac{1}{2}}\left( \int_{M}  | u|^2 ((\beta_L)_*|\cK_{A''}|\mu_\alpha) \right)^{\frac{1}{2}}.
\end{align*}
In particular, if $(\beta_{L,R})_*(|\cK_{A''}|\mu_\alpha) \leq C\gamma_\alpha$, then we get 
$$
|\langle A''u \sqrt{\gamma_\alpha}, v \sqrt{\gamma_\alpha}\rangle| \leq C^2\|u\|\|v\|.
$$
So we need to understand when $(\beta_{L,R})_*(|\cK_{A''}|\mu_\alpha)$ exist and are bounded. Notice that ~\eqref{eq:density} gives
$$
(\beta_{L,R})_*(|\cK_{A''}|\mu_\alpha) = (\beta_{L,R})_*(|\cK_{A''}|\beta_{R,L}^*(\gamma_\alpha))\gamma_\alpha.
$$
It is straightforward to compute the index sets of $(|\cK_{A''}|\beta_{R,L}^*(\gamma_\alpha))$. In particular, we will have existence of the pushforward $(\beta_{L})_*(|\cK_{A''}|\beta_{R}^*(\gamma_\alpha))$ if $\Re (E_{01}+a) - (1+\alpha) n > 0$ and of $(\beta_{R})_*(|\cK_{A''}|\beta_{R}^*(\gamma_\alpha))$ if $\Re (E_{10}-a') - (1+\alpha) n > 0$. Now we can apply the push-forward theorem~\ref{prop:push}, which will give the boundness:
\begin{align*}
(\beta_{L})_*(|\cK_{A''}|\mu_\alpha) \leq C\gamma_\alpha &\iff \Re\left( (E_{10}-a'-(1+\alpha) n )\bar\cup (E_{11}+a - a') \right) > 0,\\
(\beta_{R})_*(|\cK_{A''}|\mu_\alpha) \leq C\gamma_\alpha &\iff \Re\left( (E_{01}+a-(1+\alpha) n )\bar\cup (E_{11}+a - a') \right)>0.
\end{align*}
The index sets on the left of the extended union must be positive as we saw previously. Thus by definition of the extended union to guarantee boundness we must additionally have
$$
\Re (E_{11} + a - a') > 0.
$$

Compactness follows from the compactness of embeddings of weighted Sobolev spaces on manifolds of bounded geometry (see, for example,~\cite[Theorem 4.6]{sobolev}).
\end{proof}

\subsection{Construction of a Parametrix}\label{subsec:param}

The goal of this section is to finish the proof that an $\alpha$-elliptic differential operator admits an $\alpha$-pseudodifferential inverse modulo compact operators, acting on weighted Sobolev spaces with weights governed by the indicial roots of the operator. 

Consider a differential operator $P$ which in local coordinates can be written as
$$
P=\sum_{j + |\beta| \leq m}a_{j,\beta}(x,y)(x\p_x)^j(x^{1+\alpha}\p_y)^\beta,
$$
which belongs to $\Psi^m_\alpha(M)$. Its principal symbol in this coordinates is given by
$$
^\alpha \sigma_m(P) = \sum_{j + |\beta| \leq s}a_{j,\beta}(x,y) \xi^j \eta^\beta.
$$
We define the normal operator $N(P)$ at a point $q\in \p M$ to be defined as an operator obtained by restricting to the front face the kernel of the lift of $P$ to $M^2_\alpha$. Using local coordinates
$$
(s,u,\tilde{x},\tilde{y}) = \left(\frac{\tilde{x}}{x},\frac{y-\tilde y}{x^{1+\alpha}},\tilde{x},\tilde{y} \right),
$$   
we obtain
\begin{equation}
\label{eq:normal}
N(P) = \sum_{j + |\beta| \leq m}a_{j,\beta}(0,\tilde y)(s\p_s)^j(s^{1+\alpha}\p_u)^\beta.
\end{equation}

We also need for the parametrix construction the family of indicial family $I_\zeta(P)$, which is defined as 
$$
P(x^z (\log x)^p f(x,y))= x^z(\log x)^p I_\zeta(P; y)f(0,y)+ O(x^z(\log x)^{p-1}),
$$
for all $f\in C^\infty(M), \zeta \in \C, p \in \N_0$. In local coordinates this arises as the Mellin transform of the operator
$$
I(P;y) = \sum_{j\leq m} a_{j,0}(0,y)(s\p_s)^j,
$$
thus 
\[ I_\zeta(P;y) = \sum_{j\leq m} a_{j,0}(0,y)\zeta^j \]

\begin{definition}
Recall from Definition \ref{def:indicialroots} the definition of the boundary spectrum. We also have a refined notion of boundary spectrum which keeps track of multiplicities; this is defined as
$$
\widetilde{\text{Spec}}_b(P,y) = \{(z,l)\in \C \times \N_0: I_\zeta(P;y)^{-1}\textit{ has a pole at $z$ of order $\geq l+1$ }. \} 
$$
As before we say that $P$ has \emph{constant indicial roots} if the discrete set $\text{Spec}_b(P,y)\subset \C \times \N_0$ is independent of the choice of $y\in \p M$.
\end{definition}

To finally state the result we introduce some sets which will arise as the index sets of the parametrix and projectors: let $C\in \R$ and denote
\begin{align*}
\Sigma^+(C) &= \{(\gamma,p)\in \widetilde{\textrm{Spec}}_b(P): \Re \gamma > C\},\\
\Sigma^-(C) &= \{(-\gamma,p)\in \widetilde{\textrm{Spec}}_b(P): \Re \gamma < C\}, \\
\Sigma(C) &= \Sigma^+(C)\cup \Sigma^-(C) .
\end{align*}
The reason such sets arise is because when constructing the parametrix we proceed in a way formally similar to the construction of an elliptic parametrix on a closed manifold: in constructing an Neumann series for $PQ-\id-R$, i.e. asymptotically summing $\id+R+R^2+R^3+...$, with each term improving the error, both symbolically and by improving decay. In order for these compositions of $R$ to even be defined in the $\alpha$-calculus, we refer to Theorem \ref{eq:composition_alpha}; we observe that $R$ must vanish to infinite order on at least one of the side faces in order for this sum to converge even asymptotically. But further, each composition of $R$ produces more complicated index sets at the front face via the extended union; the above sets are introduced merely to account for every possible term in the asymptotic expansion.

We can now finally state and prove
\begin{thm}\label{thm:param}
Assume that $P\in \emph{Diff}_\alpha^m(M)$ is an $\alpha$-elliptic differential operator with constant indicial roots, and choose
\[ \delta\not\in \{ \Re(\zeta) + \tfrac{1}{2} \, : \, \zeta\in \emph{Spec}_b(P)  \}   \]
such that either $\delta<\underline{\delta}(P)$ or $\delta>\overline{\delta}(P)$. Then $P$ has a one-sided generalized inverse $G$ and orthogonal projections $\pi_{\ker}$ to the kernel of $P$ and $\pi_{\coker}$ to the kernel of $P^*$. More precisely, there are elements of the large $\alpha$-calculus such that,
\begin{align*}
G &\in \Psi_\alpha^{-m,\cH}(M)+ \Psi_\alpha^{-\infty,\cH'}(M),\\
\pi_{\ker} &\in \Psi_\alpha^{-\infty,\cE}(M),\\
\pi_{\coker} &\in \Psi_\alpha^{-\infty,\cF}(M),
\end{align*}
and
\begin{align*}
\id - GP &= \pi_{\ker}, \qquad\text{for } \delta > \overline{\delta},\\
\id - PG &= \pi_{\coker}, \qquad \text{for } \delta < \underline{\delta}.
\end{align*}
Here for $\delta > \overline{\delta}$, $\pi_{\ker}$ is a compact operator, and further
$$
\cE = \{E_{10},E_{01},E_{11} \} = \{\Sigma^+(\delta),\Sigma^+(\delta)-2\delta,\infty\},
$$
while for $\delta < \underline{\delta}$, $\pi_{\coker}$ is a compact operator, and further
$$
\cF = \{F_{10},F_{01},F_{11} \} = \{\Sigma^-(\delta)+2\delta,\Sigma^-(\delta),\infty\},
$$
and in either case $\cH = \{H_{10},H_{01},H_{11} \} = \{\Sigma(\delta),\Sigma(\delta),\N_0\}$, and $\cH^\prime=\{\Sigma(\delta),\Sigma(\delta),\infty\}$.
\end{thm}

\begin{proof}
Let $P\in \text{Diff}_\alpha^m(M)\subset \Psi^m_\alpha(M)$. Assume first for simplicity that we are in the special case that $\overline{\delta}(P)<\delta<\underline{\delta}(P)$, and hence $N(P)$ is actually invertible on $s^\delta L^2$. (Note that this is not in general possible; at the end of the proof we will reduce the proof for the semi-Fredholm case to this simpler one).

Using the fact that $P$ is $\alpha$-elliptic, i.e. $^\alpha\sigma_m(P)$ is invertible, we can construct a parametrix of $P$ in the small calculus as is done in the standard pseudodifferential theory. Thus there exists $Q_0\in \Psi^{-m}_\alpha(M)$, such that
$$
P Q_0 = \id  - R_0,
$$
where $R_0\in \Psi^{-\infty}_\alpha(M)$. Recall that by Theorem \ref{thm:mapprops} we know both $Q_0: x^\delta H_\alpha^{r}(M)\to x^\delta H_\alpha^{r+m}(M)$ and $R_0: x^\delta H_\alpha^{r}(M) \to x^\delta H_\alpha^{\ell}(M)$ are bounded maps for all $r,\ell$. However, the remainder $R_0$ cannot yet be a compact operator because it vanishes only to finite order at the front face.

The next step is to use the normal operator to produce a correction to the parametrix which iteratively improves the decay at the front face. Namely we find an operator $Q_1\in \Psi^{-m, \cH}_\alpha$ such that
$$
(P Q_1)|_{\cB_{11}} = N(P)(Q_1|_{\cB_{11}}) = R_0|_{\cB_{11}}, 
$$ 
then $Q_1$ will allow us to remove the contribution to the remainder from $R_0|_{\cB_{11}}$ and improve the order of vanishing at $\cB_{11}$ by one.

Assuming for now that such an operator $Q_1$ exists. Then we get that
$$
P(Q_0 + Q_1) = \id - R_1,
$$
where $R_1\in \Psi_{\alpha}^{-\infty,\widetilde{\cH}}$, where $\widetilde{\cH}$ is some index set satisfying $\widetilde{H}_{11}=1$, i.e. that the kernel of $R_1$ vanishes up to first order at the front face. Finally, we use the fact that the induced action of the operator $P$ at the side faces is given by the indicial operator,
\begin{equation}\label{eq:ind-eqn}
(PQ_2)|_{\cB_{01}} = I(P)(Q_2|_{\cB_{01}}) ,  
\end{equation}
which is invertible except when $Q_2$ has a term in its expansion at $\cB_{01}$, with exponent $\zeta_0\in \text{Spec}_b(P)$; if so we can use meromorphy of the indicial polynomial $I_\zeta(P)^{-1}$ to consider a loop around $\zeta_0\in \C$. Let $v\in C^\infty(\cB_{01})$, and consider
\[ \frac{1}{2\pi i}\oint_{\Gamma_{\zeta_0}} \frac{x^{-iz}I_z(P)^{-1}}{z-\zeta_0} v(y) dz = \sum_{j=1}^{\text{ord}(p)} x^{\zeta_0} (\log(x))^j u_{\zeta_0,j}  \]
hence because $v$ can be chosen arbitrarily we can 
\[ Pu =x^{\zeta_0} +\sum_{j=1}^{\text{ord}(p)} x^{\zeta_0} (\log(x))^j v_{\zeta_0,j} = x^{\zeta_0}v + O(x^{\zeta_0+1}),   \]
hence at the risk of modifying the index sets of $Q_2$ at term of order $(\zeta_0,p)\in \widetilde{\text{Spec}}_b(P)$ we can construct $Q_2$ such that $PQ_2 - R_1 = R_2$, an error which now vanishes to infinite order at $\cB_{10}$ and still to first order at $\cB_{11}$. This step also explains why $F_{01}= \Sigma^-(\delta)$, and self-adjointness of $\pi_{\coker}$ on $x^\delta H_\alpha^r(M)$ implies $F_{10}=F_{01}-2\delta$, since $\pi_{\ker}^* = x^\delta \pi_{\ker} x^{-\delta}$. 

Having now found operators $Q_0, Q_1, Q_2$ such that
\[ P(Q_0 + Q_1 + Q_2) = \id - R_2, \]
where $R_2$ vanishes up to infinite order at the left face and still vanishes to first order at $\cB_{11}$.  and an operator $R_2 \in \Psi^{-\infty,\{\infty,F_{01},1+ \N_0\}}$. Now we can use the Neumann series to define
$$
S = \sum_{l \in \N_0} R_2^l. 
$$
Note that by Theorem \ref{eq:composition_alpha} $R_2^l \in \Psi^{-\infty,\{\infty,F_l,l+ \N_0\}}$, where 
$$
F_l = F_{01} \overline{\cup} (F_{01}+1) \overline{\cup} \dots \overline{\cup} (F_{01}+l-1).  
$$

So we can take $F_-$ to be the extended union of all $(F_{01}+l)$, $l\in \N_0$. Hence our remainder is obtained as
\[ R_r =   \id - P(Q_0+Q_1+Q_2)(\id + S) \in \Psi_\alpha^{-\infty,\{\infty, F_-, \infty\}}. \]
Thus $G = (Q_0+Q_1 +Q_2)(\id+S) \in \Psi_\alpha^{-m,\cG}$ is a right parametrix. The left parametrix is constructed in a complete analogy. Having obtained that $P$ has closed range, we can obtain the remainders as the projections onto $\ker(P), \coker(P)$, which have index sets as in the conclusion of the theorem, by the self-adjointness argument above.

$ $\linebreak
Now let us come back to the construction of the operator $Q_1$ on which the rest of the construction relies. As already mentioned $Q_1$ must solve
\begin{equation}\label{model-problem}
N(P)(Q_1|_{\cB_{11}})= R_0|_{\cB_{11}}
\end{equation}
If we can find an inverse of $N(P)$, then we could simply take
$$
Q_1|_{\cB_{11}}= (N(P))^{-1}R_0|_{\cB_{11}}
$$
and then simply extend the kernel of this operator smoothly off the boundary face $\cB_{11}$. 

In order to invert the normal operator, we perform a partial Fourier transform on the $u$ variable in~\eqref{eq:normal} and obtain
$$
\widehat{N(P)}= \sum_{j+ |\beta| \leq 2m}a_{j,\beta}(0,y)(s\p_s)^j(s^{1+\alpha}i\eta)^\beta.
$$
Next we rescale the variable $s$ by switching to $\tau=s^{1+\alpha}|\eta|$ and $\hat{\eta}=\eta/|\eta|$. This gives us an operator of the form
$$
P_0(y,\hat{\eta}) = \sum_{j+|\beta| \leq 2m} a_{j,\beta}(0,y)\;((1+\alpha)\tau \p_\tau)^j(i\tau \hat{\eta})^\beta = \sum_{j+|\beta| \leq 2m}\widetilde{a}_{j,\beta}(0,y)(\tau\p_\tau)^j(\tau\hat{\eta})^\beta,
$$
for some modified coefficients $\widetilde{a}_j$. Thus we obtain a family of ``model Bessel-type" operators parametrized by $y\in \p M$ and $\hat\eta\in S^*_y(\p M)$. Although the Fourier transform depends on the various coordinate choices made in this representation, the conjugation of $N(P)$ by this Fourier transform does not, and $P_0$ is invariantly defined as a $b$-differential operator acting on $N_p^+\partial M=:(T_p\partial M)/\mathbb{R}_+$, the inward-pointing normal bundle (this quotient is by the $\mathbb{R}_+$-dilation action on the fibers).

For this family of Bessel-type operators \cite{mazzeo} establishes Fredholm properties for each value of $\hat{\eta}$ on certain weighted Sobolev spaces, with weight factor depending on the indicial roots of $P_0$. Define the spaces
\[ H_{b}^{r,\delta,\ell}(\mathbb{R}_+):=\left(\frac{\tau}{1+\tau}\right)^\delta (1+\tau)^{-\ell}H_{b}^r(\mathbb{R}_+)  \]
then we say a weight $\delta\in \mathbb{R}\setminus \text{Re}(\text{Spec}_b(P_0))$ is injective (resp. surjective) if the model Bessel operator 
\begin{equation}\label{model_bessel} P_0(p,\hat{\eta}): \tau^\delta H_b^{m} (N_p^+ \partial M) \to \tau^\delta L_b^2(N_p^+\partial M) \end{equation}
is injective (resp. surjective). In fact, \cite[Lemma 5.5]{mazzeo} proves that for $\delta\in \mathbb{R}$ satisfying 
\[ \delta \not\in \{\text{Re}(\zeta) + \tfrac{1}{2} \, : \, \zeta \in \text{Spec}_b(P_0) \} \]
the model operator \eqref{model_bessel} is Fredholm. Choosing $\ell=-\delta$, we can as in \cite{mazzeo}, construct a generalized inverse on $H_b^{0,\delta,-\delta}=\tau^\delta L_b^2$, given as
\[ G_0P_0=\id - \pi_{\ker}, \quad P_0 G_0=\id - \pi_{\coker}  \]
\[ G_0: H_b^{r,\delta,\ell-m}\to H_b^{r+2m,\delta,\ell}, \quad \pi_{\text{(co)ker}}: H_b^{r,\delta,\ell-2m} \to H_b^{r',\delta,\ell'}  \]
which are bounded for any $r,r',\ell,\ell'$, and further have kernels satisfying ${G_0}(\zeta,\zeta^\prime)$ and $\cK_{\pi_{i}}(\zeta,\zeta^\prime)$ are rapidly decreasing in $\zeta$ locally uniformly in $\zeta^\prime$, and rapidly decreasing in $\zeta^\prime$ locally uniformly in $\zeta$, (with the exception of $\cK_{G_0}$ having this rapid decrease property for $\zeta\neq \zeta^\prime$).

Having obtained a generalized inverse for $P_0$, which was obtained from $\widehat{N(P)}$ by the substitution $\tau = s^{1+\alpha}|\eta|$, we can reverse this substitution to obtain a generalized inverse for $\widehat{N(P)}$: namely $\widehat{N(G)}(s,\tilde{s},\eta)= (1+\alpha) G_0(s^{1+\alpha}|\eta|,\tilde{s}^{1+\alpha}|\eta|,\hat{\eta}) \, |\eta|$. The difficulty involved in inverting $N(P)$ now involves mostly ensuring that the integral
\begin{equation} \label{eq:inv-FT}
\cK_{N(G)}(s,u,\widetilde{s},\widetilde{u}) = \int e^{i(u-\tilde{u})\eta}G_0(s^{1+\alpha}|\eta|,\tilde{s}^{1+\alpha}|\eta|,\hat{\eta})|\eta| d\eta  
\end{equation}
maps boundedly between the appropriate $L^2$ spaces, i.e. that $\widehat{N(G)}$ enjoys such mapping properties as implied by those of $G_0$. %well-defined, i.e. that the integral kernel be Schwartz in $s^{1+\alpha}|\eta|$, and not simply in $s$ and $|\eta|$ separately.

\begin{lem}
The operators $\widehat{N(G)}, \widehat{\pi}_{ker}/\widehat{\pi}_{coker}$ are bounded operators from $s^\delta\hat{H}_\alpha^{r}$ to $s^\delta \hat{H}_\alpha^{r+2m}$ and $s^\delta \hat{H}^{\ell}$ resp., for all $r,\ell$ and $\eta\neq 0$, with bounds independent of $\eta$. The spaces $s^\delta \hat{H}_\alpha^{r}$ are defined as in \ref{alpha-sobolev}, but with $L\in \Psi_\alpha^r$ replaced with $\widehat{N(L)}$ (in essence replacing derivatives with respect to $s^{1+\alpha}\p_u$ with multiplication by $s^{1+\alpha}\eta$)
\end{lem}
\begin{proof}
We begin by proving boundedness on $s^\delta L^2(ds)$. The change of variables 
\[ u(s)\mapsto u_{|\eta|}(t):= \tfrac{1}{1+\alpha} \left(\tfrac{1}{|\eta|}\right)^{1-\frac{\alpha}{1+\alpha}} u\left( t \right),  \quad t=(\tfrac{s}{|\eta|})^{\frac{1}{1+\alpha}} \]
is an isometry from $L^2(ds)$ to $t^{\frac{\alpha}{2(1+\alpha)}}L^2(dt)$, hence this coordinate change is an isomorphism from $s^\delta L^2(ds)$ to $t^{\frac{\delta}{1+\alpha} + \frac{\alpha}{2(1+\alpha)}}L^2(dt)$. If we denote this latter exponent by $\delta_0 = \tfrac{\delta}{1+\alpha} + \tfrac{\alpha}{2(1+\alpha)}$ we observe
\[   ||\widehat{G}u||_{s^\delta L_s^2}=||(\widehat{G}u)_{|\eta|}||_{t^{\delta_0} L_t^2} = ||G_0(u_{|\eta|}) ||_{t^{\delta_0}L_t^2}\leq C||u_{|\eta|}||_{t^{\delta_0} L_t^2} = C||u||_{s^\delta L_s^2}   \]
The inequality here will follow from by the boundedness of $G_0$ on $t^{\delta_0} L^2$ to itself, with constant $C$ independent of $|\eta|$, the operator norm of $G_0$. To conclude boundedness of $G_0$ in spite of this new exponent $\delta_0$, notice that our condition for $\delta$ regarding the indicial roots of $P$ is equivalent to the respective condition on $\delta_0$ and the indicial roots of $P_0$. Namely, from $(1+\alpha)s\p_s=\tau\p_\tau$ we have $\zeta \in \text{Spec}_b(P)$ if and only if $\tfrac{\zeta}{1+\alpha}\in \text{Spec}_b(P_0)$. So our original condition
\[ \delta\not\in \{\Re(\zeta) + \tfrac{1}{2}: \zeta\in \text{Spec}_b(P) \} \]
is equivalent to 
\[\delta_0 \not\in \{\Re(\tfrac{\zeta}{1+\alpha}) + \tfrac{1}{2}: \tfrac{\zeta}{1+\alpha}\in \text{Spec}_b(P_0)\}    \]
because
\[ \delta_0 - \tfrac{1}{2} = \tfrac{\delta}{1+\alpha} + \tfrac{\alpha}{2(1+\alpha)} - \tfrac{1}{2} = \tfrac{\delta - 1/2}{1+\alpha}.   \]

Similarly the boundedness from $\hat{H}^r$ to $\hat{H}^{r+m}$ follows from the boundedness of $\tau^m G_0$ and $(\tau\partial_\tau)^k G_0$ for all $k\leq m$. One obtains bounds for the remaining values of $r$ by duality and interpolation as usual. Arguing similarly proves the claim for $\widehat{\pi}_{ker}/\widehat{\pi}_{coker}$. 
\end{proof}

We obtain immediately by the lemma and Plancherel's formula the following corollary
\begin{cor} 
Given $\delta$ chosen as above, $N(P): s^\delta H_\alpha^{r+m}\to s^\delta H_\alpha^{r}$ has closed range. In particular there exists a generalized inverse $N(G)$ and projectors $\pi_{\ker(N(P))}$, $\pi_{\coker(N(P))}$ such that
\[ N(G): s^\delta H_\alpha^r \to s^\delta H_\alpha^{r+m}   \]
\[ \pi_{\emph{co/ker}(N(P))}: s^\delta H_\alpha^r \to s^\delta H_\alpha^\ell  \]
are bounded for all $r,\ell$.
\end{cor}

It remains to prove that the operators $N(G), \pi_{{\emph{co/ker}(N(P))}}$ as in \eqref{eq:inv-FT}, are in fact operators in the $\alpha$-calculus and thus will allow us to solve the model problem \eqref{model-problem} to iteratively improve the vanishing of our parametrix at the front face $\cB_{11}$. So far we have treated their Schwartz kernels as functions; multiplying by our section $\mu=\sqrt{ds\,du\,d\widetilde{s}\,d\widetilde{u}}$ of the half-density bundle induced by the choice of coordinates near $\cB_{11}$, we obtain $\cK_{N(G)} r^{-(1+(1+\alpha)n)/2}\mu$ and $\cK_{\pi_{(i)}} r^{-(1+(1+\alpha)n)/2}\mu$. Notice that we have taken shifted the degree $1+(1+\alpha)n$ homogeneity of the kernels onto the density factors, leaving $\cK_{N(G)}$ and $\cK_{\pi_{(i)}}$ homogeneous of order 0. Since $\cK_{N(G)}$ and $\cK_{\pi_{(i)}}$ are smooth in the interior and homogeneous of degree 0, they are smooth down to $\cB_{11}$, except for the singularity of $\cK_{N(G)}$ at the diagonal. 

To conclude these kernels define operators in the $\alpha$-calculus we need only show they are polyhomogeneous conormal at the side faces. We detail the proof for $\cB_{10}$ as the case of $\cB_{01}$ is identical. First we show that the asymptotic expansions defining their Schwartz kernels are well defined at the side faces. From smoothness in the interior and rapid decrease in either variable as $G_0$ and $\cK_{\pi_{i}}$ approach either $\cB_{10}$ (i.e. $s\to 0$), away from $\cB_{11}$, we can replace $G_0$ in the formula \eqref{eq:inv-FT} with any finite asymptotic sum plus a rapidly decreasing remainder. The rapid decrease in $|\eta|$ of this kernel, and smooth dependence on the variables $(s,u,\widetilde{u})$, implies the finite sum of integrals is well-defined and similarly for the remainder. 

Near the corners, the points in $\cB_{10}\cap \cB_{11}$ (resp. $\cB_{01}\cap \cB_{11}$). Now, as $s\to 0$, we can consider in lieu of $G_0$ in formula \eqref{eq:inv-FT} its expansion in $s$, and every term is of the form
\[ \cK_{N(G)}^{(\ell)} (\widetilde{s},u,\widetilde{u})= \int e^{i (u-\widetilde{u})\cdot\eta} G_{0,\ell}(\widetilde{s}^{1+\alpha}|\eta|, \hat{\eta})|\eta|^\ell d\eta,  \]
where $\ell\in \R$, and $G_{0,\ell}$ is rapidly decreasing and polyhomogeneous in $\widetilde{s}^{1+\alpha}|\eta|$. From this we conclude that $\cK_{N(G)}$ has a polyhomogeneous expansion. To conclude that the coefficients are all conormal i.e. of stable regularity measured with respect to vector fields tangent to $\cB_{10}\cap \cB_{11}$. 
Further, since we are away from $\cB_{01}$ we can use coordinates in which $\{ \tfrac{\widetilde{x}}{x} = 1, \tfrac{y-\widetilde{y}}{x^{1+\alpha}}=0 \} = \{ \widetilde{s}=1, \widetilde{u}=0 \}$. Each term in the expansion $s\to 0$ is evaluated at $s=0$, and $\widetilde{s}=1$ in these coordinates, so each term in the asymptotic expansion is of the form,
\[ \cK_{N(G)}^{(\ell)} (\widetilde{s},u,\widetilde{u})\biggr|_{\{\widetilde{s}=1, \widetilde{u}=0 \}}  = \int e^{i u\cdot\eta} G_{0,\ell}(|\eta|, \hat{\eta})|\eta|^\ell d\eta. \]
A brief calculation gives that the vector fields $\cV_b(\cB_{10}\cap \cB_{11})$ are spanned by projections of $u_j\p_{u_k}$ which preserve the regularity of $\cK_{N(G)}^{(\ell)}$. From this we can conclude that $N(G)\in \Psi_\alpha^\cH$, with index set as yet undetermined. 

It is clear from the previous discussion that the index set of $N(G)$ is the same as that of $G_0$, the model Bessel operator. This has already been computed, in \cite[Thm 4.20,Thm 6.1]{mazzeo}, and is as our theorem's conclusion.

$ $\linebreak
Finally, we can return to the semi-Fredholm case assumed away at the beginning. Assume that $\delta < \underline{\delta}$ so that $N(P)$ is surjective. The case $\delta>\overline{\delta}$ is treated analogously. Form a new operator $\mathfrak{P} = PP^*$, which is an operator of order $2m$ that can be extended to be self-adjoint. Since $N(\mathfrak{P})=N(P)N(P)^*$ this operator is an isomorphism as required. \end{proof}

Finally, we sketch the proof of Theorem~\ref{thm:phg_Grushin}
\begin{proof}[Proof of Theorem~\ref{thm:phg_Grushin}]
The proof is very analogous to the proof of Theorem 7.3 in \cite{mazzeo}, so we only give here an idea of the proof. Note that when $\alpha = 0$, $\Theta$ reduces to $\N$ as in \cite{mazzeo}.

One rewrites $Lu=0$ as
\begin{equation}
\label{eq:reduction}
I(L)u=Eu,
\end{equation}
where $E$ contains all of the terms of the form $ x^l b_{j,\beta}(x\p_x)^j(x^{1+\alpha}\p_y)^\beta$, where $l\geq 1$ if $j\neq 0$, and with $b_{j,\beta}$ are smooth. Since we have a parametrix in the small calculus and zero function belongs to $x^\delta H_\alpha^\infty(M,\omega_\alpha)$, we have $u\in x^\delta H_\alpha^\infty(M,\omega_\alpha) $ as well, and so we can assume without any loss of generality that $Eu\in x^\delta L^2(M,\Omega^{1/2}_\alpha)$. 

We take the Mellin transform on both sides of~\eqref{eq:reduction} giving us
$$
p(s)u_M(s) = (Eu)_M(s).
$$
For simplicity we assume that the indicial polynomial $p$ has no roots on the lines
$$
\Re s = \delta -\frac{1}{2} + \frac{1+n+\alpha n}{2}+ \theta, \qquad \theta \in \Theta.
$$
For brevity we denote
$$
\delta' = \delta + \frac{1+n+\alpha n}{2}.
$$
Since the roots of the indicial polynomial form a finite set, we have that
$$
u_M(s) = p(s)^{-1}(Eu)_M(s)
$$
is a meromorphic function on the half-plane 
$$
\Re s < \delta' -\frac{1}{2} .
$$
But by assumption $u\in x^\delta L^2(M,\omega_\alpha)$, so it is actually holomorphic on this half-plane and zeroes of $(Eu)_M$ cancel out the poles.

In order to obtain the first term of asymptotics, let $\theta_1$ be the first non-zero index in $\theta$. We have that $Eu$ is a sum of terms $x^l(x\p_x)^j (x^{1+\alpha} \p_y)^\beta u$. If $\beta = 0$, then since $u\in x^\delta H^\infty_\alpha(M,\omega_\alpha)$, we can interpret this terms as belonging to $x^{\delta' + 1} L^2(dx,L^2(dy)) \subset x^{\delta' + \theta_1}L^2(dx,H^{-1}(dy))$, since $\theta_1 \leq 1$. If $\beta\neq 0$, then $x^l$ is absorbed into the principal term $x^{1+\alpha}\p_y$, which again gives us a term of the form in $x^{\delta' + \theta_1}L^2(dx,H^{-1}(dy))$. Thus we see that $u_M$ can be meromorphically extended to the half-plane $\Re s < \delta'- 1/2+ \theta_1$ to a function that takes values in $H^{-1}(dy)$ and taking the inverse along $\Re s = \delta'- 1/2+ \theta_1$
$$
u - \sum x^{s}(\log x)^p u_{j,0,p}(y) \in x^{\delta' + \theta_1} L^2(dx, H^{-1}(dy)).
$$
where $s_j$ are the roots of the indicial polynomial in the half-space $\Re s < \delta'- 1/2+ \theta_1$. 

We then can continue this procedure by extending $u_M$ on each new stripe $\delta' - 1/2 +\theta_{N-1} < \Re s < \delta' - 1/2 +\theta_{N}$ to a meromorphic function with values in $H^{-N}(dy)$, which gives the full asymptotic expansion~\eqref{eq:polyhom_ass_sobolev}.

\end{proof}

\section*{Appendix A: Scalar curvature computations}

In this Appendix we prove Proposition~\ref{prop:curvature_ass} and compute the scalar curvature for the model space from the introduction. We use Cartan's method of moving frames for computing the scalar curvature in an orthonormal frame. The resulting formula is likely well-known to experts, but for the sake of completeness we give a derivation here. We will use Einstein summation convention over repeated indices unless it is stated explicitly otherwise.

Let $M$ be a smooth $n$-dimensional Riemannian manifold. Unless stated otherwise all lower-case letter indices range from one to $n$. Let $X_i$ be a local orthonormal frame of a Riemannian manifold and $\theta^i$ be the dual frame. Given a Levi-Cevita connection $\nabla$, we can define the Christofel symbols $\Gamma^i_{jk}$ as
$$
\nabla_{X_i}X_j = \Gamma^k_{ij}X_k.
$$
We can equivalently define a connection as a family of one forms $\theta^i_j$, via
$$
\nabla_X(X_i) = \theta_i^j(X) X_j,
$$
which translates to 
\begin{equation}
\label{eq:forms}
\theta^i_j = \Gamma^i_{kj}\theta^k.
\end{equation}
The metric compatibility condition in terms of Christofel symbols reads as
$$
\Gamma^i_{jk} + \Gamma^k_{ji} = 0,
$$
which in the language of forms is equivalent to
$$
\theta^i_j + \theta^j_i = 0.
$$
Note that in particular $\theta^i_i = 0$ and $\Gamma^i_{ji} = 0$.

The following theorem is the base of Cartan's moving frame method.
\begin{thm}[\cite{dg_cartan}, Theorem 2.1]
Let $M$ be a Riemannian manifold. Suppose that $\theta^i$ is a co-frame dual to an orthonormal frame and that $\theta^i_j$ are one-forms which define the Levi-Cevita connection of this metric. Then
\begin{equation}
\label{eq:cartan}
d\theta^i = \theta^j \wedge \theta^i_j.
\end{equation} 
\end{thm}

Using the connection forms $\theta^i_j$ we can compute the curvature. Define
\begin{equation}
\label{eq:curvature_form}
\Omega^j_i = d\theta^j_i + \theta^j_k \wedge \theta^k_i.
\end{equation}
\begin{prop}[\cite{dg_cartan}, Proposition 2.3f]
We have
$$
\Omega^i_j = R^i_{jkl} \theta^k \wedge \theta^l,
$$
where $R^i_{jkl}$ are the components of the Riemann tensor in the orthonormal frame $X_i$.
\end{prop}
This correspondence allows us to compute the scalar curvature of $M$ as
\begin{equation}
\label{eq:scalar}
S = \sum_{i=1}^n\Omega^j_i(X_j,X_i).
\end{equation}
\begin{lem}
Scalar curvature of $M$ can be computed as
\begin{equation}
\label{eq:scalar_new}
S = \sum_{i=1}^n 2X_j[\Gamma^j_{ii}] + \Gamma^j_{ki}\Gamma^k_{ij} + \Gamma^j_{jk}\Gamma^k_{ii} - \Gamma^j_{ki}\Gamma^k_{ji} - \Gamma^j_{ik}\Gamma^k_{ji}.
\end{equation}
\end{lem}

\begin{proof}
Using~\eqref{eq:curvature_form} and~\eqref{eq:forms} we have
$$
\Omega^j_i = d(\Gamma^j_{ki}\theta^k) + \Gamma^j_{kl}\Gamma^k_{mi} \theta^l \wedge \theta^m = X_l[\Gamma^j_{ki}] \theta^l \wedge \theta^k + \Gamma^j_{ki}d\theta^k + \Gamma^j_{lk}\Gamma^k_{mi} \theta^l \wedge \theta^m.
$$
Combining this with~\eqref{eq:cartan} and~\eqref{eq:forms} we obtain
\begin{align*}
\Omega^j_i &= X_l[\Gamma^j_{ki}] \theta^l \wedge \theta^k + \Gamma^j_{ki}\Gamma^k_{ml}\theta^l \wedge \theta^m + \Gamma^j_{lk}\Gamma^k_{mi} \theta^l \wedge \theta^m = \\
&= (X_l[\Gamma^j_{mi}] + \Gamma^j_{ki}\Gamma^k_{ml} + \Gamma^j_{lk}\Gamma^k_{mi}) \theta^l \wedge \theta^m.
\end{align*}
So we find
$$
S = \sum_{i=1}^n X_j[\Gamma^j_{ii}] - X_i[\Gamma^j_{ji}]+\Gamma^j_{ki}\Gamma^k_{ij} + \Gamma^j_{jk}\Gamma^k_{ii} - \Gamma^j_{ki}\Gamma^k_{ji} - \Gamma^j_{ik}\Gamma^k_{ji}.
$$
Using the skew-symmetry of Christofel symbols gives us exactly the formula~\eqref{eq:scalar_new}.
\end{proof}

Let us compute the scalar curvature of the model $\alpha$-Grushin plane. Recall that is simply $\R^{n+1}$ with coordinates $(x,y_1,\dots,y_n)$ a metric given by
$$
g = \begin{pmatrix}
1 & 0 \\
0 & |x|^{-2\alpha} \id_n 
\end{pmatrix}.
$$
We have the following orthonormal vector fields
$$
X_0 = \p_{x}, \qquad X_i = |x|^\alpha \p_{y_{i}}
$$
and the dual forms
$$
\theta^0 = dx, \qquad \theta^i = \frac{dy_i}{|x|^\alpha}.
$$
We plug those into the Cartan structure equations~\eqref{eq:cartan}. Using the skew-symmetry we obtain
$$
d\theta^0 = \theta^i \wedge \theta_i^0  \iff 0 =\Gamma_{ji}^0 \theta^i \wedge \theta^j.
$$
Thus $\Gamma_{ji}^0= 0 $ if $i\neq j$. Before computing the second set of structure equations note that
$$
d\theta^i = -\alpha\frac{\sign(x) dx \wedge dy_i}{|x|^{1+\alpha}} = -\frac{\alpha}{x} \theta^0 \wedge \theta^i.
$$
Thus the second set of structure equations gives us
$$
d\theta^i = \theta^0 \wedge \theta_0^i + \theta^j \wedge \theta_j^i  \iff -\frac{\alpha}{x}\theta^0 \wedge \theta^i =\Gamma_{j0}^i \theta^0 \wedge \theta^j + \Gamma^i_{kj} \theta^j \wedge \theta^k.
$$
Thus $\Gamma^i_{kj}=0$ and the only non-zero Christofel symbol is
$$
\Gamma^i_{i0} = -\frac{\alpha}{x}.
$$
We can now apply directly the formula~\eqref{eq:scalar_new}. We obtain writing explicitly all of the summations
\begin{equation}
\label{eq:grushin_cruvature}
S = \sum_{i=1}^n \left(2X_0[\Gamma_{ii}^0] - \Gamma^i_{i0}\Gamma_{i0}^i + \sum_{j=1}^n \Gamma^j_{j0}\Gamma^0_{ii}\right) = - \frac{2\alpha n + \alpha^2 n + \alpha^2 n^2}{x^2}.
\end{equation}

We are now ready to prove Proposition~\ref{prop:curvature_ass}.

\begin{proof}[Proof of Proposition~\ref{prop:curvature_ass}]
Consider the same coordinates as in the definition~\ref{def:grushin}. Define a companion metric in this tubular neighborhood
$$
\tilde{g} = dx^2 + g_{x}.
$$
Since the neighborhood is foliated by the level sets of the distance function to the boundary, we can take $X_0 = \p_x$ and $Y_i $ to form a local orthonormal frame for the companion metric in a way, that $Y_i$ generate the tangent space to the leafs. 

Denote by capital letters $A,B,C,...$ indices in the range $0,\dots,n$. Since $Y_i$ form an involutive distribution, we have that structure constants $c_{AB}^0$ must be zero. Also, since $Y_0 = \p_x$, we have $c_{0A}^0= 0$.

We introduce now an orthonormal frame for the metric~\eqref{eq:metric_grushin}:
$$
X_0 = Y_0, \qquad X_i = |x|^\alpha Y_i.
$$
We have
\begin{align*}
[X_0,X_i] &= \frac{\alpha \sign(x)}{|x|}X_i+c^j_{0i}X_j,\\
[X_i,X_j] &= |x|^\alpha c^k_{ij}X_k.
\end{align*}
Hence we can define the new structure constants
$$
\bar c_{0i}^j = \left(\frac{\alpha}{x} \delta^j_i+c^j_{0i}\right), \qquad \bar c_{ij}^k = |x|^\alpha c_{ij}^k. 
$$

In order to compute the scalar curvature using formula~\eqref{eq:scalar_new} we need to find the Christofel symbols. Using Koszul's formula one can find (see~\cite[Theorem 16.21]{abb}):
$$
\bar{\Gamma}^C_{AB} = \frac{1}{2}\left(\bar c_{AB}^C - \bar c^A_{BC} + \bar c^B_{CA} \right).
$$
We denote by $\Gamma^C_{AB}$ the Christofel symbols of the Levi-Cevita connection of the companion metric. We have 
\begin{align*}
\bar{\Gamma}^0_{ij} &= \frac{\alpha}{x}\delta_{ij} + \Gamma^0_{ij},\\
\bar{\Gamma}^i_{0j} &= \Gamma^i_{0j},\\
\bar{\Gamma}^k_{ij} &= |x|^\alpha \Gamma^k_{ij}.
\end{align*}

We can now use formula~\eqref{eq:scalar_new} to obtain the main term of asymptotics of $S$ for $\alpha> -1$ as $x\to 0+$. First of all, notice that the from the explicit formula the Christofel symbols it follows that $S$ has an asymptotic expansion with the index set $-2 + \Theta$, where $\Theta$ is the index set from~\ref{thm:phg_Grushin}. Now we discuss the principle order of each term in~\eqref{eq:scalar_new}. We have 
$$
X_j[\bar\Gamma^j_{ii}] = O(|x|^{2\alpha}),
$$
as $x\to 0$. Similarly we have
$$
X_0[\bar\Gamma^0_{ii}] = O(x^{-2}),
$$
All the remaining terms are multiplications of the Christofel symbols. If $\alpha>-1$, then clearly the principal term should be of order $x^{-2}$, which comes from multiplications of $\bar\Gamma^0_{ii}$. But then the leading term is exactly the second expression in~\eqref{eq:grushin_cruvature} and the same calculation shows that  
$$
S = -\frac{2\alpha n + \alpha^2 n +\alpha^2n^2}{x^2} + o(|x|^{-2}), \qquad x\to 0.
$$

\end{proof}

\section*{Acknowledgements}
We would like to thank Matteo Gallone for useful discussions regarding the description of natural domains that allowed us to correct some errors in the first draft of this text. 

The first author was supported by the French government through the France 2030 investment plan managed by the
National Research Agency (ANR), as part of the Initiative of Excellence of Université Côte d’Azur under reference
number ANR-15-IDEX-01 as well as through the CIDMA Center for Research and Development in Mathematics and Applications, and the Portuguese Foundation for Science and Technology (“FCT - Funda\c c\~ao para a Ci\~encia e a Tecnologia”) within the project UIDP/04106/2020 and UIDB/04106/2020.

\bibliographystyle{plain}
\bibliography{references}

\begin{thebibliography}{10}

\bibitem{Bessel}
Milton Abramowitz and Irene~A. Stegun.
\newblock Handbook of mathematical functions with formulas, graphs, and
  mathematical tables. {Reprint} of the 1972 ed.
\newblock A {Wiley}-{Interscience} {Publication}. {Selected} {Government}
  {Publications}. {New} {York}: {John} {Wiley} \& {Sons}, {Inc}; {Washington},
  {D}.{C}.: {National} {Bureau} of {Standards}. xiv, 1046 pp.; {\$} 44.95
  (1984)., 1984.

\bibitem{abb}
Andrei Agrachev, Davide Barilari, and Ugo Boscain.
\newblock {\em A comprehensive introduction to sub-{R}iemannian geometry},
  volume 181 of {\em Cambridge Studies in Advanced Mathematics}.
\newblock Cambridge University Press, Cambridge, 2020.
\newblock From the Hamiltonian viewpoint, with an appendix by Igor Zelenko.

\bibitem{albin2022sub}
Pierre Albin and Hadrian Quan.
\newblock Sub-riemannian limit of the differential form heat kernels of contact
  manifolds.
\newblock {\em International Mathematics Research Notices}, 2022(8):5818--5881,
  2022.

\bibitem{sobolev}
Bernd Ammann and Victor Nistor.
\newblock Weighted {S}obolev spaces and regularity for polyhedral domains.
\newblock {\em Comput. Methods Appl. Mech. Engrg.}, 196(37-40):3650--3659,
  2007.

\bibitem{thesis}
Malte Behr.
\newblock {\em Quasihomogeneous Blow-Ups and Pseudodifferential Calculus on
  $\overline{SL(n;\R)}$}.
\newblock Master thesis, University of Oldenburg, 2021.

\bibitem{bes_grushin}
Ivan Beschastnyi.
\newblock Closure of the laplace-beltrami operator on 2d almost-riemannian
  manifolds and semi-fredholm properties of differential operators on lie
  manifolds.
\newblock {\em Results in Mathematics}, 78(2):59, 2023.

\bibitem{me}
Ivan. Beschastnyi, Ugo Boscain, and Eugenio. Pozzoli.
\newblock Quantum confinement for the curvature laplacian {$-\Delta +cK$} on
  2d-almost-riemannian manifolds.
\newblock {\em Potential Anal.}, 2021.

\bibitem{grushin}
Ugo Boscain and Camille Laurent.
\newblock The {L}aplace-{B}eltrami operator in almost-{R}iemannian geometry.
\newblock {\em Ann. Inst. Fourier (Grenoble)}, 63(5):1739--1770, 2013.

\bibitem{Prandi}
Ugo Boscain and Dario Prandi.
\newblock Self-adjoint extensions and stochastic completeness of the
  {L}aplace-{B}eltrami operator on conic and anticonic surfaces.
\newblock {\em J. Differential Equations}, 260(4):3234--3269, 2016.

\bibitem{nistor_fred}
Catarina Carvalho, Victor Nistor, and Yu~Qiao.
\newblock Fredholm conditions on non-compact manifolds: theory and examples.
\newblock In {\em Operator theory, operator algebras, and matrix theory},
  volume 267 of {\em Oper. Theory Adv. Appl.}, pages 79--122.
  Birkh\"{a}user/Springer, Cham, 2018.

\bibitem{ode}
Earl~A. Coddington and Norman Levinson.
\newblock {\em Theory of ordinary differential equations}.
\newblock McGraw-Hill Book Co., Inc., New York-Toronto-London, 1955.

\bibitem{georgescu3}
Jan Derezi\'{n}ski and Vladimir Georgescu.
\newblock On the domains of {B}essel operators.
\newblock {\em Ann. Henri Poincar\'{e}}, 22(10):3291--3309, 2021.

\bibitem{dewitt}
C.~DeWitt-Morette, K.~D. Elworthy, B.~L. Nelson, and G.~S. Sammelman.
\newblock A stochastic scheme for constructing solutions of the
  {S}chr\"{o}dinger equations.
\newblock {\em Ann. Inst. H. Poincar\'{e} Sect. A (N.S.)}, 32(4):327--341,
  1980.

\bibitem{epstein_melrose}
C.~L. Epstein, R.~B. Melrose, and G.~A. Mendoza.
\newblock Resolvent of the {L}aplacian on strictly pseudoconvex domains.
\newblock {\em Acta Math.}, 167(1-2):1--106, 1991.

\bibitem{Marilena}
Paolo Facchi, Giancarlo Garnero, and Marilena Ligab\`o.
\newblock Self-adjoint extensions and unitary operators on the boundary.
\newblock {\em Lett. Math. Phys.}, 108(1):195--212, 2018.

\bibitem{luca3}
Valentina Franceschi, Dario Prandi, and Luca Rizzi.
\newblock Recent results on the essential self-adjointness of sub-laplacians,
  with some remarks on the presence of characteristic points.
\newblock {\em S\'eminaire de th\'eorie spectrale et g\'eom\'etrie}, 33:1--15,
  2015-2016.

\bibitem{luca2}
Valentina Franceschi, Dario Prandi, and Luca Rizzi.
\newblock On the essential self-adjointness of singular sub-{L}aplacians.
\newblock {\em Potential Anal.}, 53(1):89--112, 2020.

\bibitem{alessandro}
Matteo Gallone and Alessandro Michelangeli.
\newblock Quantum particle across grushin singularity.
\newblock {\em Journal of Physics A: Mathematical and Theoretical}, 2021.

\bibitem{ale_gallone_book}
Matteo Gallone and Alessandro Michelangeli.
\newblock {\em Self-adjoint extension schemes and modern applications to
  quantum {H}amiltonians}.
\newblock Springer Monographs in Mathematics. Springer, Cham, [2023] \copyright
  2023.
\newblock With a foreword by Sergio Albeverio.

\bibitem{EU}
Matteo Gallone, Alessandro Michelangeli, and Eugenio Pozzoli.
\newblock On geometric quantum confinement in {G}rushin-type manifolds.
\newblock {\em Z. Angew. Math. Phys.}, 70(6):Paper No. 158, 17, 2019.

\bibitem{pozzoli2}
Matteo Gallone, Alessandro Michelangeli, and Eugenio Pozzoli.
\newblock Quantum geometric confinement and dynamical transmission in {G}rushin
  cylinder.
\newblock {\em Rev. Math. Phys.}, 34(7):Paper No. 2250018, 91, 2022.

\bibitem{mendoza}
Juan~B. Gil, Thomas Krainer, and Gerardo~A. Mendoza.
\newblock On the closure of elliptic wedge operators.
\newblock {\em J. Geom. Anal.}, 23(4):2035--2062, 2013.

\bibitem{sa_general}
D.~M. Gitman, I.~V. Tyutin, and B.~L. Voronov.
\newblock {\em Self-adjoint extensions in quantum mechanics}, volume~62 of {\em
  Progress in Mathematical Physics}.
\newblock Birkh\"{a}user/Springer, New York, 2012.
\newblock General theory and applications to Schr\"{o}dinger and Dirac
  equations with singular potentials.

\bibitem{dg_cartan}
Leonor Godinho and Jos\'{e} Nat\'{a}rio.
\newblock {\em An introduction to {R}iemannian geometry}.
\newblock Universitext. Springer, Cham, 2014.
\newblock With applications to mechanics and relativity.

\bibitem{grieser2001basics}
Daniel Grieser.
\newblock Basics of the b-calculus.
\newblock {\em Approaches to Singular Analysis: A Volume of Advances in Partial
  Differential Equations}, pages 30--84, 2001.

\bibitem{sa_riem}
Alberto Ibort, Fernando Lled\'{o}, and Juan~Manuel P\'{e}rez-Pardo.
\newblock Self-adjoint extensions of the {L}aplace-{B}eltrami operator and
  unitaries at the boundary.
\newblock {\em J. Funct. Anal.}, 268(3):634--670, 2015.

\bibitem{mazzeo}
Rafe Mazzeo.
\newblock Elliptic theory of differential edge operators. {I}.
\newblock {\em Comm. Partial Differential Equations}, 16(10):1615--1664, 1991.

\bibitem{mazzeo-uniquecont}
Rafe Mazzeo.
\newblock Unique continuation at infinity and embedded eigenvalues for
  asymptotically hyperbolic manifolds.
\newblock {\em American Journal of Mathematics}, 113(1):25--45, 1991.

\bibitem{mazzeo1998pseudodifferential}
Rafe Mazzeo and Richard~B Melrose.
\newblock Pseudodifferential operators on manifolds with fibred boundaries.
\newblock {\em arXiv preprint math/9812120}, 1998.

\bibitem{mazzeo1987meromorphic}
Rafe~R Mazzeo and Richard~B Melrose.
\newblock Meromorphic extension of the resolvent on complete spaces with
  asymptotically constant negative curvature.
\newblock {\em Journal of Functional analysis}, 75(2):260--310, 1987.

\bibitem{melrose1993atiyah}
Richard Melrose.
\newblock {\em The Atiyah-Patodi-singer index theorem}.
\newblock CRC Press, 1993.

\bibitem{melrose_book}
Richard Melrose.
\newblock {\em Differential analysis on manifolds with corners}.
\newblock 1996.

\bibitem{melrose_b}
Richard~B. Melrose and Paolo Piazza.
\newblock Families of {D}irac operators, boundaries and the {$b$}-calculus.
\newblock {\em J. Differential Geom.}, 46(1):99--180, 1997.

\bibitem{mellin}
Alexander~D. Poularikas, editor.
\newblock {\em The transforms and applications handbook}.
\newblock The Electrical Engineering Handbook Series. CRC Press, Boca Raton,
  FL; IEEE Press, New York, 1996.

\bibitem{luca1}
Dario Prandi, Luca Rizzi, and Marcello Seri.
\newblock Quantum confinement on non-complete {R}iemannian manifolds.
\newblock {\em J. Spectr. Theory}, 8(4):1221--1280, 2018.

\bibitem{guerra}
R{\'u}ben {Sousa}, Manuel {Guerra}, and Semyon {Yakubovich}.
\newblock {Product formulas and convolutions for two-dimensional
  Laplace-Beltrami operators: beyond the trivial case}.
\newblock {\em arXiv e-prints}, page arXiv:2006.14522, June 2020.

\bibitem{treves}
Fran\c{c}ois Tr\`eves.
\newblock {\em Introduction to pseudodifferential and {F}ourier integral
  operators. {V}ol. 1}.
\newblock University Series in Mathematics. Plenum Press, New York-London,
  1980.
\newblock Pseudodifferential operators.

\bibitem{vaillant2001index}
Boris Vaillant.
\newblock Index and spectral theory for manifolds with generalized fibred
  cusps.
\newblock {\em arXiv preprint math/0102072}, 2001.

\bibitem{woodhouse}
N.~M.~J. Woodhouse.
\newblock {\em Geometric quantization}.
\newblock Oxford Mathematical Monographs. The Clarendon Press, Oxford
  University Press, New York, second edition, 1992.
\newblock Oxford Science Publications.

\end{thebibliography}

\end{document}